\documentclass[11pt]{article}
\usepackage[numbers,sort&compress]{natbib}
\usepackage{enumerate}
\usepackage{amscd}
\usepackage{array}
\usepackage{amsmath}
\usepackage{latexsym}
\usepackage{amsfonts}
\usepackage{amssymb}
\usepackage{amsthm}
\usepackage{verbatim}
\usepackage{mathrsfs}
\usepackage{enumerate}
\usepackage{hyperref}

\theoremstyle{plain}\newtheorem{definition}{Definition}[section]
\theoremstyle{definition}\newtheorem{theorem}{Theorem}[section]
\theoremstyle{plain}\newtheorem{lemma}[theorem]{Lemma}
\theoremstyle{plain}
\theoremstyle{plain}
\theoremstyle{remark}\newtheorem{remark}{Remark}[section]
\theoremstyle{definition}
\theoremstyle{plain}
\usepackage{xcolor}

\newcommand{\Div}{\mathrm{div}\,}
\newcommand{\B}{\Big}

\newcommand{\be}{\begin{equation}}
\newcommand{\ee}{\end{equation}}
\newcommand{\ba}{\begin{aligned}}
\newcommand{\ea}{\end{aligned}}

\newcommand{\f}{\frac}

\newcommand{\ben}{\begin{enumerate}}
\newcommand{\een}{\end{enumerate}}

\newcommand{\Rmnum}[1]{\expandafter\@slowromancap\romannumeral #1@}

\allowdisplaybreaks

\numberwithin{equation}{section}
\begin{document}%\begin{CJK*}{GBK}{fs}
%%%%%%%%%%%%%%%%%%%%%%%%%%%%%%%%%%%%%%%%%%%%%%%%%%%%%%%%%%%%%%%%%%%%%%%%%%%%%%%%%%%%%%%%%%%%%%%%%%%%
\title{Moments propagation in the Vlasov equation   and application to the Navier-Stokes-Vlasov equations}
\author{Yanqing Wang\footnote{School of Mathematics and   Information Science, Zhengzhou University of Light Industry, Zhengzhou, Henan  450002,  P. R. China Email: wangyanqing20056@gmail.com},\;~~Wei Wei\footnote{ School of Mathematics and Center for Nonlinear Studies,  Northwest University, Xi'an, Shaanxi 710127,  P. R. China  Email: ww5998198@126.com },\;~~ Yulin Ye\footnote{Corresponding author. School of Mathematics and Statistics,
 Henan University,
 Kaifeng, 475004,
 P. R. China. Email: ylye@vip.henu.edu.cn   }
 \;~~and\;~~  Huan Yu\footnote{
School of Applied Science, Beijing Information Science and Technology University, Beijing, 100192, P.
R. China Email: yuhuandreamer@163.com} 
}
\date{}
\maketitle \vspace{-1.2cm}
\begin{abstract}
In this paper, we study how the propagation of particle moments depends on the fluid velocity in the Vlasov equation. We present some new moment estimates involving the velocity in this equation.    As an application, we establish Ladyzhenskaya-Prodi-Serrin-type regularity criteria involving only the velocity field for weak solutions of the 3D Navier-Stokes-Vlasov system.

\end{abstract}
 \vspace{-0.2cm}
%%%%%%%%%%

\section{Introduction}
\label{intro}
 This work is to examine
 the minimal regularity required of the fluid velocity to
ensure the propagation of moments in the  Vlasov equation
\be\label{Vlasov}
\partial_t f+v\cdot\nabla_{x}f+\nabla_{v}\cdot[f(u-v) ] =0,
\ee
where $f=f(t,x,v)$ represents the particle distribution function  and $u=u(t,x)$ denotes the fluid velocity.
This equation arises      the    fluid-particle  equations such as  the Navier-Stokes-Vlasov system:
\be\left\{\ba\label{NSV}
&\partial_t u +u \cdot \nabla u -  \Delta u +\nabla \Pi =-\int_{\mathbb{R}_{v}^{3}} f(u-v)dv, \\
&\partial_t f+v\cdot\nabla_{x}f+\nabla_{v}\cdot[f(u-v) ] =0, \\
&\Div u = 0, \\
&u(0,x) = u_{0}, \quad f(0,x,v) =f_{0},
\ea\right.\ee
Here, $\Pi=\Pi(t,x)$ is the pressure.
Equations \eqref{NSV} describes the interaction between a viscous incompressible fluid and a dispersed phase of non-collisional particles (see e.g. \cite{[Koch],[CHL]}).
The fluid dynamics follow the incompressible Navier-Stokes equations, while the particles evolve according to a Vlasov-type kinetic equation.
The coupling arises through a friction force of the form $F = F_{0}(u-v)$.
In this work, we assume that $ F_0=1$.
This coupled system, originally derived by Walliams in the context of combustion theory \cite{[Walliams]}, has practical applications in biosprays and dilute gas-solid suspensions (see, e.g., \cite{[BBJM],[CPP], [CHL], [Koch]}).

In a seminal work, Hamdache \cite{[Hamdache]} established the global existence of solutions to the coupled           Navier-Stokes-Vlasov equations \eqref{NSV}  with the advective terms
removed from the velocity equations and analyzed their large-time asymptotic behavior in both two and three spatial dimensions.
Furthermore, Hamdache derived a priori estimates for the particle density $n=\int_{\mathbb{R}^{d}_{v}}fdv$ and the mean velocity $j=\int_{\mathbb{R}^{d}_{v}}fvdv$ from the Vlasov equation in \eqref{NSV} as follows:
\be\ba\label{1.3}
& \|n \|_{L_{t}^{\infty}(0,T;L_{x}^{\f{d +k}{d }})} \leq C (\|f_{0}\|_{L^{\infty}_{ x,v}}+1)\B[ \|f_{0}|v|^{k } \|_{L_{x,v}^{1}} ^{\f{1 }{d +k}}+ (\|f_{0}\|_{L^{\infty}_{ x,v}}+1) \B(\int_{0}^{T}\|u\|_{L_{x}^{ d +k }}ds\B) \B]^{d},
\\
& \|j \|_{L_{t}^{\infty}(0,T;L_{x}^{\f{d +k}{d +1}})} \leq C (\|f_{0}\|_{L^{\infty}_{ x,v}}+1)\B[ \|f_{0}|v|^{k } \|_{L_{x,v}^{1}}^{\f{1 }{d +k}}+ (\|f_{0}\|_{L^{\infty}_{ x,v}}+1) \B(\int_{0}^{T}\|u\|_{L_{x}^{ d +k }}ds\B) \B]^{d+1}.
\ea\ee
Hamdache's moment estimates \eqref{1.3} have been application to the
existence of weak solutions of the homogeneous Navier-Stokes-Vlasov equations and
  inhomogeneous case (see \cite{[BDGM],[MV],[Yu],[YY],[WY]}),
 regularity of weak solutions of the   Navier-Stokes-Vlasov equations (see \cite{[Yu2],[Dechicha],[HMMM]}), well-posedness of the Navier-Stokes-Vlasov-Fokker-Planck system  (see \cite{[CKL]}),
 uniqueness and energy identity of weak solutions,   large time behavior  (see \cite{[HMM]}). Propagation of moments plays an important role in   the Vlasov type equation (see \cite{[PP]}).
The first objective is to generalize Hamdache's estimate  \eqref{1.3} and study the higher moment estimates of the Vlasov equations \eqref{Vlasov}.
\begin{theorem}\label{newthe1.1}
Let the pair $(u, f)$ be a solution
of the Vlasov equation \eqref{Vlasov}.
Then, there holds
\begin{align}
& \f{d}{dt}\| f |v|^{k}\|_{L^{1}_{x, v} } \leq C k \|u\|_{L^{k +d }_{x}} \|f_{0} \|^{\f{1}{k +d }}_{ L^{\infty}_{x,v}}\B\|\int_{ \mathbb{R}_{v}^{d}}
f  |v|^{k  } dv\B \|_{L_{x}^{1}}^{\f{k +d-1}{k +d }};\label{pri1}\\
 &
  \frac{d}{d t} \left\|\langle v\rangle^{k} D f\right\|_{L^{p}_{x,v}}^{p}  \leq
C\left(\|u\|_{L^{\infty}_{x} }+\|\nabla u\|_{L^{\infty}_{x} }+1\right)\left\|\langle v\rangle^{k} D f\right\|_{L^{p}_{x,v}}^{p};\label{pri2}
\\&
 \frac{d}{d t}\left\|\langle v\rangle^{k}D^{2}f \right\|_{L^{2}_{x,v}}^{2}
 \leq
C\|\partial_{i}\partial_{j} u\|_{L^{4} _{x} }\left\|\langle v\rangle^{k} D^{2}f \right\|_{L^{2}_{x,v}}^{\f{n}{4}+1}
+ C(\|\nabla u\|_{L^{\infty} _{x} }+\| u\|_{L^{\infty}_{x}}+1)\left\|\langle v\rangle^{k} D^{2}f \right\|_{L^{2}_{x,v}}^{2};\label{pri3}
 \end{align}
\end{theorem}
As a direct consequence of \eqref{pri1}, for $1\leq p<\infty$, one deduces the following estimate via kinetic interpolation inequality \eqref{2.3}
\be\label{geHamdache's}
\B\|\int_{\mathbb{R}^{d}_{v}}
f |v|^{k} dv\B\|_{L^{\infty}(0,T;L_{x}^{p})} \leq C \|f_{0}\|_{L^{\infty}_{ x,v}}^{\f{ p-1}{p}}\B[ \|f_{0} |v|^{pd+pk-d } \| _{L^{1}_{x,v} }+ C \|f_{0} \|_{L^{\infty}_{x,v}} \B(\int_{0}^{T}\|u\|_{L_{x}^{pd+pk }}ds\B)^{pd+pk}\B]^{\f{1}{p} },
\ee
which covers Hamdache's moment estimates \eqref{1.3}. In addition, it is worth pointing out that moment estimate \eqref{pri2} can be found in \cite{[CKL]}.
We deduce the moments estimate \eqref{pri1} concerning $f$ by the measure
$v dv dx$. It seems that it is required   stronger regularity assumptions on the velocity field $u$ via the measure $\langle v\rangle dvdx $ for  \eqref{pri1}, where $\langle v\rangle=\sqrt{1+|v|^{2}}$ (see Lemma \ref{higermomentsestimate2}).
However, in studying higher moment propagation involving $D f$,   and $D^{2} f$, we will use the measure $\langle v\rangle dvdx $ in stead of $ v dvdx,$  which was used by Chae-Kang-Lee \cite{[CKL]} and is particularly effective in dealing with the higher moment estimates as well as
the Navier-Stokes-Vlasov type equations (see \cite{[CKL]} and the proof of Theorem \ref{the1.1}). We refer the reader to Lemma \ref{a.2} in appendix for  higher moment estimates $D^{s}   f\langle v\rangle^{k}$ for $s\geq 3$.
Roughly speaking, a straightforward motivation for studying moment propagation of the Vlasov equations is to gain regularity information on the particle density $n$ and the mean velocity $j$ in $\eqref{NSV}_{1}$.
We would like to mention that recent moment  estimates for the  Vlasov equation \eqref{Vlasov} with
  compact support in velocity was deduced  by
Dechicha  in \cite{[Dechicha]}, where
two blow-up criteria based on both the velocity and  the particle distribution function for Navier-Stokes-Vlasov equations \eqref{NSV}
are derived.

Next, we present an application of moment estimates  in Theorem \ref{newthe1.1} to the regularity of weak solutions of the 3D
Navier-Stokes-Vlasov system. Before we state our result, we recall the mathematical
progress of the Navier-Stokes-Vlasov equations \eqref{NSV}.
Global weak solutions of the system \eqref{NSV} on the torus $\mathbb{T}^{d}$ for $d=2, 3$ were constructed by Boudin et al. \cite{[BDGM]}.
For bounded domains, Yu proved global existence of weak solutions satisfying the energy inequality in \cite{[Yu]}, and also proved uniqueness in two dimensions (see also \cite{[Yu2]}).
With the addition of a velocity-space diffusion term $\Delta_{v}f$ to the left-hand side of $\eqref{NSV}_{2}$, the kinetic equation becomes the Vlasov-Fokker-Planck equation.
In this framework, Mellet and Vasseur \cite{[MV]} proved the existence of weak solutions to the compressible Navier-Stokes-Vlasov-Fokker-Planck system.
The hydrodynamic limit of this compressible system was studied in \cite{[GJV],[GJV1]}. Chae, Kang, and Lee \cite{[CKL]} studied higher regularity results for solutions of the 3D Vlasov-Stokes system in the whole space, using the weighted measure $\langle v\rangle dvdx$.
They also proved global existence of weak solutions for the Navier-Stokes-Vlasov-Fokker-Planck system in both two and three dimensions  in \cite{[CKL]}.
Boudina, Grandmont, and Moussa \cite{[BGM]} proved the existence of global weak solutions to the 3D incompressible Navier-Stokes-Vlasov system \eqref{NSV} in a time-dependent domain with absorption boundary conditions for the kinetic part.
The existence of global weak solutions to the initial-boundary value problem for the inhomogeneous Navier-Stokes-Vlasov system was studied by Wang and Yu \cite{[WY]}.
Local and global strong solutions to the inhomogeneous Vlasov-Navier-Stokes system with small initial data were constructed  by Choi and Kwon in \cite{[CK]}, where the large-time behavior of solutions was also considered.
Han-Kwan, Miot, Moussa, and Moyano \cite{[HMMM]} studied the uniqueness and energy identity of weak solutions to the two-dimensional Vlasov-Navier-Stokes equations \eqref{NSV}.
The large time behavior of Fujita-Kato type solutions to the Vlasov-Navier-Stokes system in a periodic domain was considered by Han-Kwan, Moussa, and Moyano in \cite{[HMM]} (see \cite{[Kwan]} for the whole space case).
We refer the reader to \cite{[DLM],[GY],[HMM],[YY],[WY],[MW],[LWW],[LS],[LMW],[He],[GHMZ],[CG]} for further studies on various aspects of the Navier-Stokes-Vlasov type equations and turn our attention to the contributions in \cite{[Hamdache],[Yu],[Yu2],[CKL],[BGM],[HMMM]}.

Navier-Stokes-Vlasov equations \eqref{NSV} without the   particle distribution function reduce to the Navier-Stokes equations.
For the classical 3D Navier-Stokes equations, a foundational result concerning the regularity and uniqueness of Leray--Hopf weak solutions is the Ladyzhenskaya-Prodi-Serrin (LPS) condition (see \cite{[Galdi]}).
We recall the LPS criteria of the 3D Navier--Stokes equations.
A weak solution $u$ is smooth on $(0,T]$ if it satisfies one of the following conditions:
\begin{enumerate}[(1)]
\item Serrin \cite{[Serrin]}, Struwe \cite{[Struwe]}, Escauriaza, Seregin and \v{S}ver\'{a}k \cite{[ESS]}
\be \label{serrin1}
u\in L^{p} (0,T;L^{q}( \mathbb{R}^{3})) ~~~ \text{with}~~~~2/p+3/q=1, ~~q\geq3.
\ee
\item Beirao da Veiga \cite{[Beirao da Veiga],[Beirao da Veiga2]}
\be \label{serrin2}
\nabla u\in L^{p} (0,T;L^{q}( \mathbb{R}^{3})) ~~~ \text{with}~~~~2/p+3/q=2 , ~~q>3/2.
\ee
\item Berselli and Galdi \cite{[BG]}, Struwe \cite{[Struwe2]}, Zhou \cite{[Zhou1],[Zhou2]}
\be \label{serrin3}\ba
&\Pi \in L^{p} (0,T;L^{q}( \mathbb{R}^{3})) ~~~ \text{with}~~~~2/p+3/q=2 , ~~q>3/2;\\
&
 \nabla\Pi \in  L^{p} (0,T;L^{q}( \mathbb{R}^{3})) ~~~ \text{with}~~~~2/p+3/q=3 , ~~q>1.
  \ea\ee
\end{enumerate}
This naturally raises the question of whether the LPS criteria can be extended to weak solutions of the 3D Navier--Stokes--Vlasov system \eqref{NSV}.
The next   aim of this paper is to investigate such extensions via the above propagation of moments of the Vlasov. Our second result is the following:
\begin{theorem}
\label{the1.1}
Let $(u,\,\Pi,\,f)$ be a weak solution to the Navier-Stokes-Vlasov system \eqref{NSV}
as defined by Lemma \ref{wsdefi}.
Assume that $u_{0} \in H^{2}(\mathbb{R}^{3})$,  $f_{0}\in L^{\infty}(\mathbb{R}^{3} \times \mathbb{R}^{3}_{v}),$
$|v|^{k}f_{0} \in L^{1}(\mathbb{R}^{3} \times \mathbb{R}^{3}_{v})$, and
$\langle v \rangle^{k} D f_{0}, \langle v\rangle^{k} D ^{2} f_{0}\in L^{2}\left(\mathbb{R}^{3} \times \mathbb{R}^{3}_{v}\right)$ for any $k \geq 0$.
Then the pair $(u(x,t), f(x,v,t))$ is a strong solution on $(0,T]$ if one of the following conditions holds:
\begin{enumerate}[(1)]
\item \label{coro1}
$ u \in L^{p }(0,T; L ^{q }(\mathbb{R}^{3}))
\quad \text{with} \quad
2/p+3/q=1 , ~\f{24}{7} \leq q\leq\infty;$

\item
$\nabla u \in L^{p }(0,T; L ^{q }(\mathbb{R}^{3}))
\quad \text{with} \quad
2/p+3/q=2, \f85\leq q\leq\infty;$

\item
$\Pi\in L^{p }(0,T; L ^{q }(\mathbb{R}^{3}))
\quad \text{with} \quad
2/p+3/q=2 , ~\f{9}{4} \leq q\leq3;$

\item
$\nabla\Pi \in L^{p }(0,T; L ^{q }(\mathbb{R}^{3}))
\quad \text{with} \quad
2/p+3/q=3 , ~\f97\leq q\leq3$.
\end{enumerate}
\end{theorem}
We present several comments regarding Theorem \ref{the1.1}.

The proof of this theorem relies on the moment propagation in Theorem    \ref{newthe1.1} of the Vlasov equation $\eqref{Vlasov}$ and Giga and Sohr's $L_{t}^{p}L_{x}^{q}$ estimates of the time-dependent Stokes equations.
Our first observation is that the conditions $u\in L^{1}(0,T;L^{8}(\mathbb{R}^{3}))$ and $f_{0}|v|^{3}\in L^{1}(\mathbb{R}^{3}\times\mathbb{R}_{v}^{3})$ imply that the particle density and momentum satisfy $n\in L^{\infty}(0,T;L^{2}(\mathbb{R}^{3}))$ and $j\in L^{\infty}(0,T;L^{2}(\mathbb{R}^{3}))$ by moment propagation \eqref{geHamdache's}.
This enables us to get $H^{1}(\mathbb{R}^{3})$ estimates for the fluid velocity from equations $\eqref{NSV}_1$, leading in turn to $u\in L^{2}(0,T;L^{\infty}(\mathbb{R}^{3}))$.
By using moment propagation \eqref{geHamdache's} once more, we can enhance the regularity of both the particle density $n$ and the mean velocity $j$.
Using the $L_{t}^{p}L_{x}^{q}$ estimates of the non-stationary Stokes equations developed by Giga and Sohr, we obtain the key estimate $\nabla u\in L^{1}(0,T;L^{\infty}(\mathbb{R}^{3}))$.
This estimate allows us to get  the analysis of the gradient of the particle distribution function $f$ by moment propagation \eqref{pri2}, leading to the conclusion $\nabla^{2} u\in L^{2}(0,T;L^{4}(\mathbb{R}^{3}))$ via elliptic estimates for the stationary Stokes system.
Based on this, we can investigate the moments propagation of $D^{2} f$   by higher
moment estimate  \eqref{pri3}.
Finally, using the steady Stokes estimate once again, we deduce the desired estimate $\nabla^{2} \partial_{i}u,\partial_{i}\Pi\in L^{2}(0,T;L^{2}(\mathbb{R}^{3}))$. In this process, it seems that it is necessary to deduce the moments estimate $j\in L^{\infty}(0,T;L^{2}(\mathbb{R}^{3}))$ from $u\in L^{\f{2q}{q-3}}(0,T;L^{q}(\mathbb{R}^{3}))$ with $q\geq8$.
Relaxing the lower bound on $q$ is a challenge.
To address this, our second observation is that the time-dependent Stokes equations
\be
\partial_t v - \Delta v +\nabla \Pi = j \ee
have regularity $\nabla^{2}v\in L^{2}(0,T;L^{\f32}(\mathbb{R}^{3}))$ provided $j \in L^{2}(0,T;L^{\f32}(\mathbb{R}^{3}))$.
The Sobolev embedding combined with the Gagliardo-Nirenberg inequality implies $v\in L^{2}(0,T;L^{\tilde{q}}(\mathbb{ R}^3)$ with $\tilde{q}<\infty.$ Based on this, we reformulate the velocity equations as
\be
 \partial_t u   - \Delta u +\nabla \Pi =-u \cdot \nabla u-nu+j=hu+j
\ee
 where $h\in L^{2}(0,T;L^{2}(\mathbb{R}^{3}))$.
By applying Giga-Sohr's $L_{t}^{p}L_{x}^{q}$ estimates of time depended Stokes equations, we derive $u\in L^{1}(0,T;L^{8}(\mathbb{R}^{3}))$ under the weaker assumption
\[
u\in L^{\f{2q}{q-3}}(0,T;L^{q}(\mathbb{R}^{3}))
\quad \text{with} \quad
\f{24}{7} \leq q<8.
\]
A natural direction for future work is to improve the lower bound on the spatial Lebesgue exponent $q$ in Theorem \ref{the1.1}.

Since the Navier-Stokes-Vlasov system \eqref{NSV} become the
Navier-Stokes-Fokker-Planck equations by adding a diffusion term $-\Delta_{v}f$ in the Vlasov equations $\eqref{NSV}_{2}$, with a slight modification of the proof, Theorem \ref{the1.1} can be adapted to the 3D Navier-Stokes-Vlasov-Fokker-Planck equations.
In addition, the reader is referred to \cite{[YXW]} for an application of moment estimates \eqref{pri1} in the study of energy equality of the weak solutions of the Navier-Stokes-Vlasov equations \eqref{NSV}.

The paper is organized as follows.
In Section 2, we introduce the notations and analytical tools used throughout the paper.
Section 3 is devoted to the proof of the moment propagation of the Vlasov equation  \eqref{Vlasov} via a series of lemmas.
In Section 4, we consider the regularity of weak solutions of the Navier-Stokes-Vlasov system via moment estimates proved in previous section.
Ladyzhenskaya-Prodi-Serrin-type regularity criteria involving only
the velocity field for weak solutions of this system is derived.
   Finally, an appendix is dedicated
to  an analogue of  moment estimate \eqref{geHamdache's} via the measure $\langle v\rangle dvdx $ and higher moment propagation in the Vlasov equation  \eqref{Vlasov}.

\section{Preliminaries}

We begin with the notations used throughout this paper.
For $q\in [1,\,\infty]$, the space $L^{q}(0,\,T;\,X)$ denotes the set of measurable functions on the time interval $(0,\,T)$ taking values in a Banach space $X$ such that $\|g(t,\cdot)\|_{X} \in L^{q}(0,\,T)$. Let $\alpha=(\alpha_{1},\cdots, \alpha_{n})$ and $\beta=(\beta_{1},\cdots, \beta_{n})$, where $\alpha_{i}$ and $\beta_{i}$ are nonnegative integer.
We write
$$D^{s}f(x,v)= \partial_{x}^{\alpha}\partial_{v}^{\beta}f(x,v)=\partial_{x_{1}}^{\alpha_{1}}\cdots
\partial_{x_{n}}^{\alpha_{n}}\partial_{v_{1}}^{\beta_{1}}\cdots
\partial_{v_{n}}^{\beta_{n}}f(x,v)$$
and
$$
\|D^{s}f(x,v)\|_{L^{p}(\mathbb{R}^{n}\times\mathbb{R}_{v}^{n})}
=\sum_{|\alpha|+|\beta|=s}\|\partial_{x}^{\alpha}\partial_{v}^{\beta}
f(x,v)\|_{L^{p}(\mathbb{R}^{n}\times\mathbb{R}_{v}^{n})},
$$
where
$|\alpha|+|\beta|=\alpha_{1}+\cdots+\alpha_{n}+\beta_{1}+\cdots+\beta_{n}=s.$

We next provide the definitions of weak solutions and strong solutions to the Navier-Stokes-Vlasov system \eqref{NSV}, respectively.
\begin{definition}
[Weak solution]
\label{wsdefi}
A pair $(u, \,f)$ is called a weak solution to the Navier--Stokes--Vlasov system \eqref{NSV} if the following conditions hold:
\begin{itemize}%\addtolength{\itemsep}{-0.2\baselineskip}
\item $u\in L^{\infty}(0,T;L^{2}(\mathbb{R}^{3}))\cap L^{2}(0,T;H^{1}(\mathbb{R}^{3}))$;
\item $f(t,x, v)\geq 0, \text{ for any } (t,x,v)\in (0,T)\times\mathbb{ R}^3\times\mathbb{ R}_{v}^3$;
\item $f\in L^{\infty}(0,T;L^{\infty}(\mathbb{ R}^3\times\mathbb{ R}_{v}^3)\cap L^{1}(\mathbb{ R}^3\times\mathbb{ R}_{v}^3))$;
\item $f|v|^2\in L^{\infty}(0,T;L^{1}(\mathbb{ R}^3\times\mathbb{ R}_{v}^3))$;
 \item
The functions $u$ and $f$ solve the Navier--Stokes--Vlasov equations \eqref{NSV} in the sense of distributions;
\be\label{ws1}
\begin{aligned}
&\int_{\mathbb{R}^{3}}(u \cdot \Psi)(T, x) d x+\int_{0}^{T} \int_{\mathbb{R}^{3}}\left(\nabla u: \nabla \Psi-u \otimes u: \nabla \Psi-u \cdot \partial_{t} \Psi\right)(s, x) d x d s \\
&=
\int_{\mathbb{R}^{3}} u_{0}(x) \Psi(0, x) d x+\int_{0}^{T} \iint_{\mathbb{R}^{3}\times\mathbb{R}_{v}^{3}} (v-u) f d v \Psi(s, x) d x d s
\end{aligned}
\ee
with $\nabla u: \nabla \Psi=\sum_{j, k=1}^{3} \partial_{j} u^{k} \partial_{j} \Psi^{k}$ and $u \otimes u: \nabla \Psi=\sum_{j, k=1}^{3} u^{j} u^{k} \partial_{j} \Psi^{k}$,
\be\label{ws2}
\int_{0}^{T} \iint_{\mathbb{R}^{3} \times \mathbb{R}^{3}_{v}} f\left(\partial_{t} \varphi+v \cdot \nabla_{x} \varphi+(u-v) \cdot \nabla_{v} \varphi\right) d x d v d t=\iint_{\mathbb{R}^{3} \times \mathbb{R}^{3}_{v}} f_{0} \varphi(0, x, v) d x d v
\ee
for all $\Psi \in C^{1} (0,T; (C_{c}^{\infty} (\mathbb{R}^{3} ) )^{3} )$ with $\nabla \cdot \Psi=0$ and $\varphi \in C^{1} (0,T; C_{c}^{\infty} (\mathbb{R}^{3} \times \mathbb{R}_{v}^{3} ) )$ with $\varphi(T, \cdot, \cdot)=0$.\\

\item
The functions $u$ and $f$ satisfy the energy inequality
\be \label{ei}
\ba
&\int_{\mathbb{ R}^3}\frac{1}{2}|u|^2\;dx+\iint_{\mathbb{ R}^3\times\mathbb{ R}_{v}^3}f |v|^2 \;dx dv
+ \int_{0}^T\int_{\mathbb{ R}^3}|\nabla u|^2\,dxdt+\int_{0}^{T}\iint_{\mathbb{ R}^3\times\mathbb{ R}^3_{v}}f|u-v|^2\,dv dx dt \\
\leq&
\int_{\mathbb{ R}^3}\frac{1}{2}|u_0|^2\;dx+\iint_{\mathbb{ R}^3\times\mathbb{ R}^3_{v}} f_0 | v|^2 \;dxdv
\ea
\ee
holds for almost every $t \in [0,T]$.
\end{itemize}
\end{definition}
The existence of weak solutions to the Navier-Stokes-Vlasov system \eqref{NSV} under initial data satisfying $u_{0}\in L^{2}(\Omega)$, $f_{0}\geq0$, $f_{0}\in L^{\infty}(\Omega\times\mathbb{ R}_{v}^d)\cap L^{1}(\Omega\times\mathbb{ R}_{v}^d)$, $f_{0}|v|^2\in L^{\infty}(0,T;L^{1}(\Omega\times\mathbb{ R}_{v}^d))$ can be found in \cite{[Yu],[BDGM],[HMMM],[BGM]}, where $\Omega\subset \mathbb{R}^{d}$ with $d=2, 3$ is torus, bounded domain, a time-dependent domain or the whole space.

The strong solution to the Navier-Stokes-Vlasov system \eqref{NSV} is defined as follows:

\begin{definition}(\cite{[CK]})\label{ss}
A pair $(u,f)$ is called a strong solution of the Navier-Stokes-Vlasov system \eqref{NSV} if it satisfies the equations in the sense of distributions, and fulfills the following regularity conditions:
\begin{itemize}%\addtolength{\itemsep}{-0.2\baselineskip}
\item
$f\in C([0,T]; H^{2}(\mathbb{ R}^3\times\mathbb{ R}_{v}^3))$,

\item
$u\in C([0,T]; H^{2}(\mathbb{ R}^3))\cap L^{2}([0,T]; H^{3}(\mathbb{ R}^3))$,

\item
$u_{t}\in C([0,T];L^{2}(\mathbb{ R}^3))\cap L^{2}([0,T];H^{1}(\mathbb{ R}^3))$,

\item
$\nabla \Pi\in C([0,T];L^{2}(\mathbb{ R}^3))\cap L^{2}([0,T];H^{1}(\mathbb{ R}^3))$.
\end{itemize}
\end{definition}

For the reader's convenience, we present the following kinetic interpolation inequality, which is widely  used in the analysis of Vlasov type equations (see \cite{[PLions],[CKL],[MV]}).
\begin{lemma}
Let $g \in L^{p}(\mathbb{R}^{d}\times \mathbb{R}_{v}^{d})$ and $g|v|^{\ell}\in L^{1}(\mathbb{R}^{d}\times \mathbb{R}_{v}^{d})$.
Then for $k\leq\ell$ and $g\geq0$, we have $\int_{\mathbb{R}_{v}^{d}}g|v|^{k}dv \in L^{q}(\mathbb{R}^{d})$ with the estimate
\be\ba \label{2.2} &
\B\|\int_{\mathbb{R}_{v}^{d}}
g |v|^{k} dv\B\|_{L^{q}(\mathbb{R}^{d})}\leq C\|g\|_{L^{p}(\mathbb{R}^{d}\times \mathbb{R}_{v}^{d})}^{\f{\ell-k}{d(1-\f1p)+\ell}}
\|g|v|^{\ell}\|_{L^{1}(\mathbb{R}^{d}\times \mathbb{R}_{v}^{d})}^{\f{d(1-\f1p)+k}{d(1-\f1p)+\ell}},\\
&\B\|\int_{\mathbb{R}_{v}^{d}}
g \langle v\rangle^{k} dv\B\|_{L^{q}(\mathbb{R}^{d})}\leq C\|g\|_{L^{p}(\mathbb{R}^{d}\times \mathbb{R}_{v}^{d})}^{\f{\ell-k}{d(1-\f1p)+\ell}}
\|g\langle v\rangle^{\ell}\|_{L^{1}(\mathbb{R}^{d}\times \mathbb{R}_{v}^{d})}^{\f{d(1-\f1p)+k}{d(1-\f1p)+\ell}},
\ea
\ee
where the exponent $q$ is determined by
$\f1q=\f{d(1-\f1p)+k}{d(1-\f1p)+\ell}+\f{\f{\ell-k}{d(1-\f1p)+\ell}}{p}$.\\
Moreover, if $g \in L^{\infty}(\mathbb{R}^{d}\times \mathbb{R}_{v}^{d})$, then
\be\ba \label{2.5v3} \B\|\int_{\mathbb{R}_{v}^{d}}
g |v|^{k} dv\B\|_{L^{q}(\mathbb{R}^{d})}\leq C\|g\|_{L^{\infty}(\mathbb{R}^{d}\times \mathbb{R}_{v} ^{d})}^{\f{\ell-k}{d +\ell}}
\|g|v|^{\ell}\|_{L_{x}^{\f{q(d +k)}{d +\ell}}( \mathbb{R}^{d})L_{v}^{1}( \mathbb{R}_{v}^{d})}^{\f{d +k}{d +\ell}}.
\ea
\ee
\end{lemma}

\begin{remark}
As a special case of the preceding lemma, we have the estimate
\be\ba\label{2.3}
\B\|\int_{\mathbb{R}^{d}_{v}}
g |v|^{k} dv\B\|_{L ^{\f{d +\ell}{d +k}}(\mathbb{R}^{d})} \leq C \|g\|_{L^{\infty}  (\mathbb{R}^{d}\times \mathbb{R}_{v}^{d})}^{\f{\ell-k}{d +\ell}}\|g|v|^{\ell}\|_{L^{1}(\mathbb{R}^{d}\times \mathbb{R}_{v}^{d})}^{\f{d +k}{d +\ell}},
\ea\ee
which will be frequently used throughout this paper. Furthermore, the borderline cases considered in \cite[Lemma 9]{[CKL]} and \cite[Lemma 3.2]{[MV]} can be recovered from \eqref{2.3}.
\end{remark}

\begin{remark}
Since $|v|\leq\langle v\rangle$, one may replace $|v|$ by $\langle v\rangle$ on the right-hand side of $\eqref{2.2}_{1}$ $\eqref{2.5v3}$, and $\eqref{2.3}$.
\end{remark}

\begin{proof}
For any $R > 0$, we compute
   \be\ba \label{2.4} \int_{\mathbb{R}^{d}}
  g  |v|^{k  }  dv=&\int_{|v|\leq R}
  g  |v|^{k  }  dv+\int_{|v|> R}
  g  |v|^{k  }  dv\\\leq&CR^{k  }\int_{|v|\leq R}
  g    dv +R^{-(\ell-k)}\int_{  |v|> R}
  g  |v|^{\ell  }  dv\\\leq&CR^{k  +(1-\f1p)d}\B(\int_{|v|\leq R}
  g^{p}    dv\B)^{\f1p}+R^{-(\ell-k)}\int_{  \mathbb{R}^{d}}
  g  |v|^{\ell  }  dv\\\leq&CR^{k  +(1-\f1p)d}\B(\int_{\mathbb{R}^{d}}
  g^{p}    dv\B)^{\f1p}+R^{-(\ell-k)}\int_{  \mathbb{R}^{d}}
  g  |v|^{\ell  }  dv.
  \ea\ee
 To minimize the right-hand side, we choose $R = (\int_{ \mathbb{R}_{v}^{d}} g |v|^{\ell} dv)^{\f{1}{d(1-\f1p)+\ell}}$, which yields
\be\ba\label{2.8}
\int_{\mathbb{R}_{v}^{d}}
  g  |v|^{k  }  dv \leq C \|g\|_{L^{p}(\mathbb{R}_{v}^{d})}^{\f{\ell-k}{d(1-\f1p)+\ell}}
  \|g|v|^{\ell}\|_{L^{1}(\mathbb{R}_{v}^{d})}
  ^{\f{d(1-\f1p)+k}{d(1-\f1p)+\ell}}.
\ea
\ee
The classical interpolation inequality in Lebesgue spaces implies $\eqref{2.2}_{1}$.

A slight modification of the argument used in \eqref{2.4} yields
\be\ba \label{2.9} \int_{\mathbb{R}^{d}}
  g  \langle v\rangle^{k  }  dv=&\int_{|v|\leq R}
  g  \langle v\rangle^{k  }  dv+\int_{|v|> R}
  g  \langle v\rangle^{k  }  dv\\\leq&C(1+R^{2})^{\f{k}{2}  }\int_{|v|\leq R}
  g    dv +(1+R^{2})^{\f{-(\ell-k)}{2}  } \int_{  |v|> R}
  g  \langle v\rangle^{\ell  }  dv\\\leq&C(1+R^{2})^{\f{k+(1-\f1p)d}{2}    }\B(\int_{|v|\leq R}
  g^{p}    dv\B)^{\f1p}+(1+R^{2})^{\f{-(\ell-k)}{2}  }\int_{  \mathbb{R}_{v}^{d}}
  g  \langle v\rangle^{\ell  }  dv\\\leq&C(1+R^{2})^{\f{k+(1-\f1p)d}{2}    }\B(\int_{\mathbb{R}_{v}^{d}}
  g^{p}    dv\B)^{\f1p}+(1+R^{2})^{\f{-(\ell-k)}{2}  }\int_{  \mathbb{R}_{v}^{d}}
  g \langle v\rangle^{\ell  }  dv.
\ea\ee
Repeating the same minimization argument as above, we obtain the estimate $\eqref{2.2}_{2}$.

 It follows from
 \eqref{2.8} that
$$
\ba
\int_{\mathbb{R}_{v}^{d}}
g |v|^{k} dv
\leq C \|g\|_{L^{\infty}(\mathbb{R}_{v}^{d})}^{\f{\ell-k}{d +\ell}}
\|g|v|^{\ell}\|_{L^{1}(\mathbb{R}_{v}^{d})}
^{\f{d +k}{d +\ell}},
\ea
$$
from which it follows that   \eqref{2.5v3}  immediately.
\end{proof}

The following result highlights the advantage of working with the weighted measure $\langle v\rangle dvdx$.

\begin{lemma}
\label{lem2.2}
Let $g\geq0$, $r\geq0$, and $k>r+\f{d(p-1)}{p}$.
If $\langle v\rangle^{k}g\in L^{p}(\mathbb{R}^{d}\times\mathbb{R}_{v}^{d})$, then
\be\ba \label{2.10}
\B\| \int_{\mathbb{R}_{v}^{d}} g |v|^{r} dv\B\|_{L^{p}(\mathbb{R}^{d})}\leq C\|\langle v\rangle^{k}g \|_{L^{p}(\mathbb{R}^{d}\times\mathbb{R}_{v}^{d})}.
\ea\ee
where the constant $C > 0$ depends only on $d$, $p$, $r$, and $k$.
\end{lemma}

\begin{proof}
We use $|v|\leq\langle v\rangle$ and the H\"older inequality to estimate
\be\ba
\int_{\mathbb{R}_{v}^{d}} g |v|^{r} dv
\leq
\int_{\mathbb{R}_{v}^{d}} g \langle v\rangle^{r} dv
\leq
\|\langle v\rangle ^{k} g \|_{L^{p}(\mathbb{R}_{v}^{d})} \B(\int_{\mathbb{R}_{v}^{d}}\f{1}{\langle v\rangle^{\f{p(k-r)}{p-1} }}dv\B)^{\f{p-1}{p}},
\ea\ee
where the last integral is finite under the assumption $k>r+\f{d(p-1)}{p}$.
Taking the $L^p$-norm in $x$ yields \eqref{2.10}.
\end{proof}
For the convenience of readers, we present   some Gr\"onwall type inequalities, which is a generalization of corresponding result  in \cite{[CKL]}.
\begin{lemma}\label{gronwall}
Let $T>0$ and nonnegative functions $f, g_{i}:[0, T] \rightarrow \mathbb{R}^{+}$ for $i=1, 2, 3$. Assume $g_{i}$ is integrable  and $$ f^{\prime}(t) \leqslant C_{1} g_{1}(t)f(t)+C_{2} g_{2}(t) f^{a}(t)+C_{3} g_{3}(t)$$ for $0 \leqslant a \leqslant 1$, where $C_{1}$, $C_{2}$ and $C_{3}$ are positive constants. We have the following:\begin{enumerate}[(1)]
\item
  If $a=1$, then
  \be\label{gronwall1}
  f(t)  \leq \exp^{\int_{0}^{t}[C_{1}g_{1}(s)+C_{2}g_{2}(s)]ds} \B[f(0)+C_{3}\int_{0}^{t}g_{3}(s) ds\B].\ee

  \item If $a=0$, then
  \be\label{gronwall2}
  f(t)  \leq \exp^{C_{1}\int_{0}^{t}g_{1}(s) ds} \B[f(0)+\int_{0}^{t}\B(C_{2}g_{2}(s)+C_{3}g_{3}(s)\B) ds\B].\ee
\item If $0 < a<1$ and $C_3=0$, then
\be\label{gronwall3}
f(t)\leq \exp^{C_{1}\int_{0}^{t}g_{1}(s) ds}\B[f^{1-a}(0)+(1-a)C_{2}\int_{0}^{t} g_{2}(s)ds\B]^{\f{1}{1-a}}.
\ee
    \end{enumerate}
\end{lemma}
\begin{proof}
(1) It is straightforward to show that
\be\ba
&\f{d}{d\tau}\B[f(\tau)\exp^{-\int_{0}^{\tau}[C_{1}g_{1}(s)+C_{2}g_{2}(s)]ds}\B]
\\=&\exp^{-\int_{0}^{\tau}[C_{1}g_{1}(s)+C_{2}g_{2}(s)]}[f'(\tau)-f (\tau)(C_{1}g_{1}(\tau)+C_{2}g_{2}(\tau))]
\\\leq & C_{3} g_{3}(\tau).
\ea\ee
  Integrate this inequality leads to \eqref{gronwall1}.

  (2) It is worth remarking that this case reduces to the first case.

(3) To proceed further, we denote
\be\ba
F(\tau)=f(\tau)\exp^{-C_{1}\int_{0}^{\tau}g_{1}(s)ds}.
\ea\ee
We compute

\be\ba
\f{d}{d\tau}F(\tau)=&[f'(\tau)-C_{1}g_{1}(\tau)f (\tau) ]\exp^{-C_{1}\int_{0}^{\tau}g_{1}(s)ds}\\
\leq& C_{2} g_{2}(\tau) f^{a}(\tau)\exp^{-C_{1}\int_{0}^{\tau}g_{1}(s)ds}.
\\
=&
C_{2} g_{2}(\tau)[f(\tau)\exp^{-C_{1}\int_{0}^{\tau}g_{1}(s)ds}]^{a}  \exp^{(a-1)C_{1}\int_{0}^{\tau}g_{1}(s)ds}\\
\leq&C_{2} g_{2}(\tau)F^{a}(\tau),
\ea\ee
which turns out that
$$
F^{1-a}(t)\leq F^{1-a}(0)+(1-a)C_{2} \int_{0}^{t}g_{2}(s)ds.
$$
The above analysis verifies  inequality \eqref{gronwall3}. This completes the proof of this lemma.
\end{proof}

\section{Moments propagation  in the Vlasov equations}

In this section, we aim to identify the minimal regularity required of the fluid velocity to ensure the propagation of moments in the particle equation.
 The proof of  Theorem \ref{newthe1.1} is given in a sequence of three lemmas.

\subsection{Basic moment estimates of the Vlasov equations}

In this subsection, we revisit Hamdache's  moment estimate \eqref{1.3} proved in \cite{[Hamdache]}.
We will derive a new moment estimate \eqref{03.2}, which recovers \eqref{1.3} via Young's inequality.

From the Vlasov equation $\eqref{NSV}_{2}$ and standard $L^{p}$ estimates, we obtain
$$
\frac{d}{d t}\|f\|_{L^{p}\left(\mathbb{R}^{d} \times \mathbb{R}^{d}_{v}\right)} \leqslant \frac{d(p-1)}{p}\|f\|_{L^{p}\left(\mathbb{R}^{d} \times \mathbb{R}^{d}_{v}\right)},$$
which implies the bound
\be\label{lwuqiong}
\|f(t)\|_{L^{\infty} (\mathbb{R}^{d} \times \mathbb{R}^{d}_{v} )} \leq C\|f_{0} \|_{L^{\infty} (\mathbb{R}^{d} \times \mathbb{R}^{d}_{v} )}.
\ee
In what follows, for the sake of conciseness, we write
$$L^{p}_{x}:=L^{p}(\mathbb{R}^{d}),~~ L^{p}_{x,v}:=L^{p}(\mathbb{R}^{d}\times\mathbb{R}^{d}_{v}).$$
We present the basic moment estimate for the Vlasov equation $\eqref{NSV}_{2}$.
\begin{lemma}\label{higermomentsestimate}
Let $(u,f)$ be a weak solution of the Vlasov equation $\eqref{NSV}_{2}$ on $\mathbb{R}^{d}\times \mathbb{R}^{d}_{v}$.
If the initial data $f_0$ satisfies
$$
\left\|| v|^{k} f_{0}\right\|_{L^{\ell}\left(\mathbb{R}^{d} \times \mathbb{R}_{v}^{d}\right)}<\infty, \quad \ell \geq1,
$$
and $u \in L^{1}(0,T;L^{k\ell+d }(\mathbb{R}^{d}))$, then
\be\label{03.1}
\B\|\| f |v|^{k}\|^{\ell}_{L_{x,v}^{\ell} }
\B\|_{L^{\infty}(0,T)}\leq e^{(d-\f{d}{\ell})T}\B[\|f_{0} |v|^{k } \|^{\ell}_{L^{\ell}_{x,v} } +C\|f_{0}^{\ell}\|_{L^{\infty}_{x,v} }\B(\int_{0}^{T}\|u\|_{L_{x}^{k\ell+d } }ds\B)^{k\ell+d}\B],
\ee
\be\ba\label{03.2}
\B\|\int_{\mathbb{R}_{v}^{d}}
f |v|^{p} dv\B\|_{L ^{\infty}(0,T;L_{x}^{\f{d +k}{d +p}})} \leq C \|f_{0}\|_{L^{\infty}_{ x,v}}^{\f{k-p}{d +k}}\B[ \|f_{0} |v|^{k } \| _{L^{1}_{x,v} }+ C \|f_{0} \|_{L^{\infty}_{x,v}} \B(\int_{0}^{T}\|u\|_{L_{x}^{k+d }}ds\B)^{k+d}\B]^{\f{d +p}{d +k} }.
\ea\ee
\end{lemma}

\begin{remark}
Applying Young's inequality, we observe that
$$
\ba
&\|f_{0}\|_{L^{\infty}_{ x,v}}^{\f{k-p}{d +k}}\leq C(\|f_{0}\|_{L^{\infty}_{ x,v}}+1),\\
&\|f_{0}\|_{L^{\infty}_{ x,v}}^{\f{1}{d +k}}\leq C(\|f_{0}\|_{L^{\infty}_{ x,v}}+1).
\ea
$$
Substituting these bounds into \eqref{03.2}, we recover \eqref{1.3}.
Therefore, this lemma provides a generalization of       Hamdache's moment estimates \eqref{1.3} in \cite{[Hamdache]}.

As a direct consequence, for $p\geq1$, the following estimate holds:
\be\ba\label{3.2}
\B\|\int_{\mathbb{R}^{d}_{v}}
f |v|^{k} dv\B\|_{L^{\infty}(0,T;L_{x}^{p})} \leq C \|f_{0}\|_{L^{\infty}_{ x,v}}^{\f{ p-1}{p}}\B[ \|f_{0} |v|^{pd+pk-d } \| _{L^{1}_{x,v} }+ C \|f_{0} \|_{L^{\infty}_{x,v}} \B(\int_{0}^{T}\|u\|_{L_{x}^{pd+pk }}ds\B)^{pd+pk}\B]^{\f{1}{p} }.
\ea\ee
\end{remark}

\begin{proof}
The classical energy estimate of the particle equation $\eqref{NSV}_{2}$ yields
\be\label{2.7}
\f{1}{\ell}\f{d}{dt}\| f |v|^{k}\|^{\ell}_{L^{\ell}(\mathbb{R}^{d}\times\mathbb{R}_{v}^{d}) }= -\iint_{\mathbb{R}^{d}\times\mathbb{R}_{v}^{d}}\nabla_{v}[f(u-v) ] f^{\ell-1} |v|^{\ell k}dvdx,
\ee
where we have used the identity $ \iint_{\mathbb{R}^{d}\times\mathbb{R}^{d}_{v}}v\cdot\nabla_{x}f f^{\ell-1} |v|^{\ell k}dvdx=0$.
Since
\[
\nabla_{v}\cdot[f(u-v) ] f^{\ell-1}
= \frac{1}{\ell} (u-v) \cdot \nabla_{v} f^{\ell} - d f^{\ell},
\]
we have
\begin{align*}
&-\iint_{\mathbb{R}^{d}\times\mathbb{R}_{v}^{d}} \nabla_{v}\cdot[f(u-v) ] f^{\ell-1} |v|^{\ell k} dvdx \\
&= k \iint_{\mathbb{R}^{d}\times\mathbb{R}_{v}^{d}} u \cdot\nabla_{v}|v| f^{\ell } |v|^{\ell k-1} dvdx
- k \iint_{\mathbb{R}^{d}\times\mathbb{R}_{v}^{d}} f^{\ell } |v|^{\ell k} dvdx \\
&\quad + (d-\f{d}{\ell}) \iint_{\mathbb{R}^{d}\times\mathbb{R}_{v}^{d}} f^{\ell } |v|^{\ell k} dvdx
\end{align*}
using integration by parts and $v \cdot\nabla_{v}|v| = |v|$.
Thus, \eqref{2.7} becomes
\be\label{3.6}\ba
&\f{1}{\ell}\f{d}{dt}\| f |v|^{k}\|^{\ell}_{L^{\ell}(\mathbb{R}^{d}\times\mathbb{R}_{v}^{d}) }+ k \iint_{\mathbb{R}^{d}\times\mathbb{R}_{v}^{d}} f^{\ell } |v|^{\ell k }dvdx\\
=& k \iint_{\mathbb{R}^{d}\times\mathbb{R}_{v}^{d}}u \cdot\nabla_{v}|v| f^{\ell } |v|^{\ell k-1}dvdx +(d-\f{d}{\ell})\iint_{\mathbb{R}^{d}\times\mathbb{R}_{v}^{d}} f^{\ell } |v|^{\ell k}dvdx.
\ea\ee
Using H\"older's inequality, kinetic interpolation inequality \eqref{2.3}, and \eqref{lwuqiong}, we obtain
\begin{align*}
 \iint_{\mathbb{R}^{d}\times\mathbb{R}_{v}^{d}} |u \cdot\nabla_{v}|v| f^{\ell } |v|^{\ell k-1}| dvdx
 \leq C \|u\|_{L_{x}^{k\ell+d }}
\|f_{0}^{\ell}\|^{\f{1}{k\ell+d }}_{L^{\infty}_{x,v} }\B\|\int_{ \mathbb{R}^{d}_{v}}
f^{\ell} |v|^{k \ell } dv\B \|_{L_{x}^{1}}^{\f{k\ell+d-1}{k\ell+d }}.\end{align*}
Substituting this into \eqref{3.6}, we have
\be\ba\label{3.7v1}
&\f{1}{\ell}\f{d}{dt}\| f |v|^{k}\|^{\ell}_{L^{\ell}(\mathbb{R}^{d}\times\mathbb{R}_{v}^{d})}+ k \iint_{\mathbb{R}^{d}\times\mathbb{R}_{v}^{d}} f^{\ell } |v|^{\ell k }dvdx
\\
&\leq C k \|u\|_{L^{k\ell+d }_{x}} \|f_{0}^{\ell}\|^{\f{1}{k\ell+d }}_{ L^{\infty}_{x,v}}\B\|\int_{ \mathbb{R}_{v}^{d}}
f^{\ell} |v|^{k\ell } dv\B \|_{L_{x}^{1}}^{\f{k\ell+d-1}{k\ell+d }}+(d-\f{d}{\ell})\iint_{\mathbb{R}^{d}\times\mathbb{R}_{v}^{d}} f^{\ell } |v|^{\ell k}dvdx.
\ea\ee
Applying Gr\"onwall's inequality \eqref{gronwall3} in Lemma \ref{gronwall}, we obtain \eqref{03.1}.
It follows directly from \eqref{03.1} that
\be\ba\label{2.16}
\B\|\iint_{\mathbb{R}^{d}\times\mathbb{R}_{v}^{d}}f|v|^{k } dvdx\B\|_{L^{\infty}(0,T)}\leq C\|f_{0} |v|^{k}\|_{L^{1}_{x,v}}+ C \|f \|_{L^{\infty}_{x,v}}\B(\int_{0}^{T}\|u\|_{L_{x}^{k+d }}ds\B)^{k+d}.
\ea\ee
According to the kinetic  interpolation inequality \eqref{2.3} and \eqref{lwuqiong}, we obtain
\be\ba\label{2.17}
\B\|\int_{\mathbb{R}^{d}_{v}}
f |v|^{p} dv\B\|_{L_{t}^{\infty}(0,T;L_{x}^{\f{d +k}{d +p}})} \leq C \|f\|_{L^{\infty}_{ x,v}}^{\f{k-p}{d +k}}\|f|v|^{k}\|_{L ^{\infty}(0,T;L_{x,v}^{1})}^{\f{d +p}{d +k} }\leq C \|f_{0}\|_{L^{\infty}_{t,x,v}}^{\f{k-p}{d +k}}\|f|v|^{k}\|_{L_{t}^{\infty}(0,T;L_{x,v}^{1})}^{\f{d +p}{d +k} }.
\ea\ee
Substituting \eqref{2.16} into \eqref{2.17}, we further deduce
\be\ba
\B\|\int_{\mathbb{R}^{d}_{v}}
f |v|^{p} dv\B\|_{L^{\infty}(0,T;L_{x}^{\f{d +k}{d +p}})} \leq C \|f\|_{L^{\infty}_{t,x,v}}^{\f{k-p}{d +k}}\B[\|f_{0} |v|^{k}\|_{L^{1}_{x,v}}+ C \|f_{0} \|_{L^{\infty}_{x,v}} \B(\int_{0}^{T}\|u\|_{L_{x}^{k+d }}ds\B)^{k+d}\B]^{\f{d +p}{d +k} }.
\ea\ee
This completes the proof of the lemma.
\end{proof}

\subsection{Higher moment estimates of the Vlasov equations}

This subsection is devoted to deriving moment estimates for the gradient of the distribution function under the Lipschitz
control of the velocity field.
We also investigate moment estimates involving the second derivatives of the distribution function.

\begin{lemma}\label{lemma3.3}
Let $(f, u)$ be a weak solution to the Vlasov equation $\eqref{NSV}_{2}$ on $\mathbb{R}^{d}\times \mathbb{R}^{d}_{v}$, and assume the initial data $f_{0}$ satisfies
$$
\langle v\rangle^{k} D f_{0} \in L^{p} (\mathbb{R}^{d} \times \mathbb{R}^{d}_{v} ), ~ \ell \geq1, ~ k\geq0.
$$
If $u\in L^{1}(0,T;L^{\infty} (\mathbb{R}^{d} ))$ and $\nabla u\in L^{1}(0,T;L^{\infty} (\mathbb{R}^{d} ))$, then
\be \langle v\rangle^{k} D f \in L^{\infty} (0, T ; L^{p} (\mathbb{R}^{d} \times \mathbb{R}^{d}_{v} ) ).
\ee\end{lemma}

\begin{remark}
For weak solutions of the Naier--Stokes--Vlasov equation $\eqref{NSV}$, the Lipschitz
control of the velocity field implies the condition $u\in L^{1}(0,T;L^{\infty}\left(\mathbb{R}^{d}\right))$.
Indeed, the interpolation inequality guarantees
$$
\|u\| _{L^{1}(0,T;L^{\infty} (\mathbb{R}^{d} ))}\leq C\|u\| _{L^{\infty }(0,T;L^{2} (\mathbb{R}^{d} ))}^{\f{2}{d+2}}\|\nabla u\| _{L^{1}(0,T;L^{\infty} (\mathbb{R}^{d} ))}^{\f{d}{d+2}}.
$$
\end{remark}

\begin{proof}
The kinetic equation $\eqref{NSV}_{2}$ can be rewritten as
$$
f_{t}+v\cdot\nabla_{x}f+(u-v)\cdot\nabla_{v}f=df.
$$
Taking the derivative $\partial_{i}$ to both sides of the equation and multiplying by $\left|\partial_{i}f\right|^{p-2} \partial_{i}f\langle v\rangle^{k p}$ yields
\begin{align*}
 \frac{1}{p} \frac{d}{d t}\left\|\langle v\rangle^{k} \partial_{i}f \right\|_{L^{p}_{x,v} }^{p}
&=
-\iint_{\mathbb{R}^{d} \times \mathbb{R}_{v}^{d}} \partial_{i}
u\nabla_{v} f\left|\partial_{i}f\right|^{p-2} \partial_{i}f \langle v\rangle^{kp}d v d x \\
&\quad
+ \iint_{\mathbb{R}^{d} \times \mathbb{R}_{v}^{d}} (u-v)\cdot\nabla_{v} \partial_{i}f\left|\partial_{i}f\right|^{p-2} \partial_{i}f\langle v\rangle^{k p} d v d x \\
&\quad
- d\iint_{\mathbb{R}^{d} \times \mathbb{R}_{v}^{d}} \partial_{i}f\left|\partial_{i}f\right|^{p-2} \partial_{i}f\langle v\rangle^{k p} d v d x \\
&=: I + II + III.
\end{align*}
Here, we also used the fact $\iint_{\mathbb{R}^{d} \times \mathbb{R}_{v}^{d}}v\cdot\nabla_{x}\partial_{i}f\left|\partial_{i}f\right|^{p-2} \partial_{i}f\langle v\rangle^{k p} d v d x =0$.

By H\"older's inequality, we arrive at
\be\ba\label{3.46}
|I_{1}|\leq&C\|\nabla u\|_{L^{\infty}_{x} }\left\|\langle v\rangle^{k} \nabla_{v} f\right\|_{L^{p}_{x,v}}\left\|\langle v\rangle^{k} \nabla_{x} f\right\|_{L^{p}_{x,v}}^{p-1}.
\ea\ee
We compute
\be\ba
II =&-\f1p\iint_{\mathbb{R}^{d} \times \mathbb{R}_{v}^{d}} u \cdot\nabla_{v} \left|\partial_{i}f\right|^{p } \langle v\rangle^{k p} d v d x+\iint_{\mathbb{R}^{d} \times \mathbb{R}_{v}^{d}} v\cdot\nabla_{v} \left|\partial_{i}f\right|^{p } \langle v\rangle^{k p} d v d x
\\=&:II_1+II_2.
\ea\ee
By applying integration by parts, we obtain
$$\ba
II_1 =& k\iint_{\mathbb{R}^{d} \times \mathbb{R}_{v}^{d}} u \left|\partial_{i}f\right|^{p } \langle v\rangle^{k p-2}v d v d x,
\ea$$
which implies that
\be\ba\label{3.48}
|II_1|\leq C\|u\|_{L^{\infty}\left(\mathbb{R}^{d}\right)}\left\|\langle v\rangle^{k} \nabla_{x} f\right\|_{L^{p} _{x,v}}^{p}.
\ea\ee
Applying integration by parts once more, we find
$$\ba
II_2=-d\iint_{\mathbb{R}^{d} \times \mathbb{R}_{v}^{d}} \partial_{i} f\left|\partial_{i}f\right|^{p-2} \partial_{i}f\langle v\rangle^{k p} d v d x-kp\iint_{\mathbb{R}^{d} \times \mathbb{R}_{v}^{d}} |v|^{2} \left|\partial_{i}f\right|^{p } \langle v\rangle^{k p-2} d v d x,
\ea$$
which leads to
\be\ba\label{3.49}
|II_2|\leq C\left\|\langle v\rangle^{k} \partial_{i}f \right\|_{L^{p}_{x,v}}^{p}.
\ea\ee
Combining \eqref{3.48} and \eqref{3.49}, we see that
\be\ba\label{3.50}
|II|\leq&C\|u\|_{L^{\infty} _{x}}\left\|\langle v\rangle^{k} \nabla_{x} f\right\|_{L^{p}_{x,v}}^{p}+C\left\|\langle v\rangle^{k} \nabla_{x} f\right\|_{L^{p}_{x,v}}^{p}.
\ea\ee
A straightforward calculation further yields
\be\ba\label{3.51}
|III|\leq C\left\|\langle v\rangle^{k} \nabla_{x} f\right\|_{L^{p}_{x,v}}^{p}.
\ea\ee
Therefore, combining \eqref{3.46}, \eqref{3.50}, and \eqref{3.51}, we conclude
\be\ba\label{3.52}
\frac{1}{p} \frac{d}{d t}\left\|\langle v\rangle^{k} \partial_{i}f \right\|_{L^{p}_{x,v}}^{p}
&\leq C\|\nabla u\|_{L^{\infty} _{x}}\left\|\langle v\rangle^{k} \nabla_{v} f\right\|_{L^{p}_{x,v}}\left\|\langle v\rangle^{k} \nabla_{x} f\right\|_{L^{p}_{x,v}}^{p-1}\\&+C(\|u\|_{L^{\infty} _{x}}+1)\left\|\langle v\rangle^{k} \nabla_{x} f\right\|_{L^{p}_{x,v}}^{p}.
\ea\ee
Applying the operator $\partial_{v_{j}}$ to the Vlasov equation $\eqref{NSV}_{2}$, we obtain
$$
\partial_{v_{j}}f_{t}+\partial_{v_{j}}v\cdot\nabla_{x}f
+v\cdot\nabla_{x}\partial_{v_{j}}f -\partial_{v_{j}}v \cdot\nabla_{v}f+(u-v)\cdot\nabla_{v}\partial_{v_{j}}f=d\partial_{v_{j}}f.
$$
Multiplying this equation by $\left|\partial_{v_{j}}f\right|^{p-2} \partial_{v_{j}}f\langle v\rangle^{k p}$, we deduce that
\be\ba\label{3.53}
 \frac{1}{p} \frac{d}{d t}\left\|\langle v\rangle^{k} \partial_{v_{j}}f \right\|_{L^{p}\left(\mathbb{R}^{d} \times \mathbb{R}^{d}\right)}^{p}
&=
\iint_{\mathbb{R}^{d} \times \mathbb{R}_{v}^{d}}\partial_{v_{j}}v \cdot\nabla_{v}f\left|\partial_{v_{j}}f\right|^{p-2} \partial_{v_{j}}f \langle v\rangle^{kp}d v d x \\
&\quad
- \iint_{\mathbb{R}^{d} \times \mathbb{R}_{v}^{d}} (u-v)\cdot\nabla_{v} \partial_{v_{j}}f\left|\partial_{v_{j}}f\right|^{p-2} \partial_{v_{j}}f\langle v\rangle^{k p} d v d x \\
&\quad
+d\iint_{\mathbb{R}^{d} \times \mathbb{R}_{v}^{d}} \partial_{v_{j}}f\left|\partial_{v_{j}}f\right|^{p-2} \partial_{v_{j}}f\langle v\rangle^{k p} d v d x \\
&=: I'+II'+III',
\ea\ee
where we have used the fact $\iint_{\mathbb{R}^{d} \times \mathbb{R}_{v}^{d}} v\cdot\nabla_{x}\partial_{v_{j}}f\left|\partial_{v_{j}}f\right|^{p-2} \partial_{v_{j}}f\langle v\rangle^{k p} d v d x=0.$

A simple calculation gives
\be\ba \label{3.54}
|I'|+|III'|\leq C\left\|\langle v\rangle^{k} \nabla_{v} f\right\|_{L^{p} _{x,v}}^{p},
\ea\ee
and
\be\ba
II'
 &=
-\f1p\iint_{\mathbb{R}^{d} \times \mathbb{R}_{v}^{d}} u \cdot\nabla_{v} \left|\partial_{v_{j}}f\right|^{p } \langle v\rangle^{k p} d v d x+\iint_{\mathbb{R}^{d} \times \mathbb{R}_{v}^{d}} v\cdot\nabla_{v} \left|\partial_{v_{j}}f\right|^{p } \langle v\rangle^{k p} d v d x \\
&=: II'_1+II'_2.
\ea\ee
Following exactly the same line of reasoning leading to \eqref{3.50}, we have
\be\ba\label{3.56}
|II'|\leq&C\|u\|_{L^{\infty} _{x}}\left\|\langle v\rangle^{k} \nabla_{v} f\right\|_{L^{p}_{x,v} }^{p}+C\left\|\langle v\rangle^{k} \nabla_{v} f\right\|_{L^{p}_{x,v} }^{p}.
\ea\ee
Inserting \eqref{3.54} and \eqref{3.56} into \eqref{3.53}, we obtain
\be\ba\label{3.57}
\frac{1}{p} \frac{d}{d t}\left\|\langle v\rangle^{k} \nabla_{v} f\right\|_{L^{p}_{x,v} }^{p} \leq C\|u\|_{L^{\infty} _{x}}\left\|\langle v\rangle^{k} \nabla_{v} f\right\|_{L^{p}_{x,v}}^{p}+C\left\|\langle v\rangle^{k} \nabla_{v} f\right\|_{L^{p}_{x,v}}^{p}.
\ea\ee
Combining \eqref{3.52} and \eqref{3.57}, we end up with
\be\label{pro2v1}
\begin{aligned}
 \frac{1}{p} \frac{d}{d t} \left\|\langle v\rangle^{k} D f\right\|_{L^{p}_{x,v}}^{p}  \leq
C\left(\|u\|_{L^{\infty}_{x} }+\|\nabla u\|_{L^{\infty}_{x} }+1\right)\left\|\langle v\rangle^{k} D f\right\|_{L^{p}_{x,v}}^{p}.
\end{aligned}
\ee
Applying the classical Gr\"onwall's inequality (1) in Lemma \ref{gronwall}, we conclude the proof of the lemma.
\end{proof}

\begin{lemma}\label{lemma3.4}
Let $(f, u)$ be a weak solution of the Vlasov equation $\eqref{NSV}_{2}$ on
$\mathbb{R}^{d}\times \mathbb{R}_{v}^{d}$  and the initial datum $f_{0}$ satisfy
$$
\langle v\rangle^{k} D ^{2}f_{0} \in L^{2} (\mathbb{R}^{d} \times \mathbb{R}^{d}_{v}), k\geq0, d\leq 4.
$$
If
\[
u\in L^{1}(0,T;L^{\infty} (\mathbb{R}^{d} )),
\quad
\nabla u\in L^{1}(0,T;L^{\infty} (\mathbb{R}^{d} )),
\quad
\nabla^{2} u\in L^{1}(0,T;L^{4} (\mathbb{R}^{d} )),
\]
then
\be\langle v\rangle^{k} D ^{2} f \in L^{\infty}(0,T;L^{2} (\mathbb{R}^{d} \times \mathbb{R}^{d}_{v})).
\ee
\end{lemma}

\begin{remark}
We remark that one can replace $\nabla^{2} u\in L^{1}(0,T;L^{4}\left(\mathbb{R}^{d} \right))$ by $\nabla^{2} u\in L^{1}(0,T;L^{3}\left(\mathbb{R}^{3} \right))$ in three dimensions and $\nabla^{2} u\in L^{1}(0,T;L^{2+\eta}\left(\mathbb{R}^{3} \right))$ in two dimensions in this lemma, where $\eta $ is a sufficiently small positive constant.
The proof is omitted.
\end{remark}

\begin{proof}
Applying the operator $\partial_{i}\partial_{j}$ to the equation $\eqref{NSV}_{2}$, we obtain
\be\label{3.92}
\partial_{i}\partial_{j}f_{t}+v\cdot\nabla_{x}\partial_{i}\partial_{j}f
+\partial_{i}\partial_{j}
u\cdot\nabla_{v} f+\partial_{i}
u\cdot\nabla_{v}\partial_{i}\partial_{j}f+ \partial_{j}u \cdot\nabla_{v}\partial_{i}f +(u-v)\cdot\nabla_{v}\partial_{i}\partial_{j}f
=d\partial_{i}\partial_{j}f.
\ee
Taking the $L^{2}$ inner product of the above equation with $ \partial_{i}\partial_{j}f \langle v\rangle^{2k}$, we conclude by
$$
\iint_{\mathbb{R}^{d}\times\mathbb{R}_{v}^{d}} v\cdot\nabla_{x}\partial_{i}\partial_{j}f\partial_{i}\partial_{j}f \langle v\rangle^{2k}dxdv=0
$$
 that
\be\ba\label{3.93}
&\frac{1}{2} \frac{d}{d t}\left\|\langle v\rangle^{k} \partial_{i}\partial_{j}f \right\|_{L^{2} _{x,v}}^{2}\\
=&
- \iint_{\mathbb{R}^{d}\times\mathbb{R}_{v}^{d}} \partial_{i}\partial_{j}
u\nabla_{v} f \partial_{i}\partial_{j}f \langle v\rangle^{2k}dxdv-
\iint_{\mathbb{R}^{d}\times\mathbb{R}_{v}^{d}}\partial_{i}
u\nabla_{v}\partial_{i}\partial_{j}f \partial_{i}\partial_{j}f \langle v\rangle^{2k}dxdv \\
&\quad
- \iint_{\mathbb{R}^{d}\times\mathbb{R}_{v}^{d}} \partial_{j}u \cdot\nabla_{v}\partial_{i}f \partial_{i}\partial_{j}f \langle v\rangle^{2k}dxdv-\iint_{\mathbb{R}^{d}\times\mathbb{R}_{v}^{d}} (u-v)\cdot\nabla_{v}\partial_{i}\partial_{j}f\partial_{i}\partial_{j}f \langle v\rangle^{2k}dxdv\\
&\quad
+ d\iint_{\mathbb{R}^{d}\times\mathbb{R}_{v}^{d}} \partial_{i}\partial_{j}f \partial_{i}\partial_{j}f \langle v\rangle^{2k}dxdv\\
=&: I+II+III+IV+V.
\ea\ee
Each term is now estimated using standard inequalities.
By H\"older's inequality, we write
\be\ba\label{3.94}
|I|
&\leq \B(\iint_{\mathbb{R}^{d}\times\mathbb{R}_{v}^{d}} |\partial_{i}\partial_{j}
u|^{2}|\nabla_{v} f \langle v\rangle^{2k}|^{2}\langle v\rangle^{2k}dxdv\B)^{1/2}\left\|\langle v\rangle^{k} \partial_{i}\partial_{j}f \right\|_{L^{2}_{x,v}}
\\
&\leq
\|\partial_{i}\partial_{j}
u\|_{L^{4}_{x} } \B(\int_{\mathbb{R}_{v}^{d}}\|\nabla_{v} f\|^{2}_{L^{4}_{x} }\langle v\rangle^{2k}dv\B)^{1/2}\left\|\langle v\rangle^{k} \partial_{i}\partial_{j}f \right\|_{L^{2}_{x,v}}.
\ea\ee
By the Gagliardo--Nirenberg inequality, we arrive at
$$\ba
\|\nabla_{v} f\|_{L_{x}^{4} }\leq C\|\nabla_{v} f\|_{L^{2}_{x}}^{1-\f{d}{4} }\|\nabla\nabla_{v} f\|_{L^{2}_{x}}^{\f{d}{4}}.
\ea$$
Consequently, we have
\be\ba\label{3.95}
\B(\int_{\mathbb{R}_{v}^{d}}\|\nabla_{v} f\|^{2}_{L^{4}_{x}}\langle v\rangle^{2k}dv\B)^{1/2}
&\leq
C\B(\int_{\mathbb{R}_{v}^{d}}\|\nabla_{v} f\|_{L^{2}_{x}}^{2-\f{d}{2}}\|\nabla\nabla_{v} f\|_{L^{2}_{x}}^{\f{d}{2}} \langle v\rangle^{2k}dv\B)^{1/2}\\
&\leq
C\| \nabla_{v} f\langle v\rangle^{ k}\|_{L^{2}_{x,v}}^{1-\f{d}{4}}\|\nabla\nabla_{v} f\langle v\rangle^{ k}\|_{L^{2}_{x,v}}^{\f{d}{4}}\\
&\leq
C \|\nabla\nabla_{v} f\langle v\rangle^{ k}\|_{L^{2}_{x,v}}^{\f{d}{4}},
\ea\ee
where we have used Lemma \ref{lemma3.4}.

Plugging \eqref{3.95} into \eqref{3.94}, we know that
\be\ba\label{3.96}
|I|\leq&C\|\partial_{i}\partial_{j}
u\|_{L^{4}_{x} }\|\nabla\nabla_{v} f\langle v\rangle^{ k}\|_{L^{2}_{x,v}}^{\f{d}{4}}\left\|\langle v\rangle^{k} \partial_{i}\partial_{j}f \right\|_{L^{2}_{x,v}}.
\ea\ee
Integrating by parts, we get
$$\ba
II=-\f12
\iint_{\mathbb{R}^{d}\times\mathbb{R}_{v}^{d}} \partial_{i}
u\nabla_{v}(\partial_{i}\partial_{j}f)^{2} \langle v\rangle^{2k}dxdv
=k
\iint_{\mathbb{R}^{d}\times\mathbb{R}_{v}^{d}} \partial_{i}
u (\partial_{i}\partial_{j}f)^{2} \langle v\rangle^{2k-2}v dxdv.
\ea$$
By H\"older's inequality, one deduce from the latter equation that
\be\ba |II|\leq C\|\nabla u\|_{L^{\infty} _{x}}\left\|\langle v\rangle^{k} \partial_{i}\partial_{j}f \right\|_{L^{2}_{x,v}}^{2}.
\ea\ee
By H\"older's inequality, we have
\be\ba
|III|\leq C\|\nabla u\|_{L^{\infty} _{x}}\left\|\langle v\rangle^{k} \nabla_{v}\partial_{i}f \right\|_{L^{2}_{x,v}} \left\|\langle v\rangle^{k} \partial_{i}\partial_{j}f \right\|_{L^{2}_{x,v}}.
\ea\ee
Using integration by parts, one has
$$\ba
IV
 &=
k \iint_{\mathbb{R}^{d}\times\mathbb{R}_{v} ^{d}} u (\partial_{i}\partial_{j}f)^{2} \langle v\rangle^{2k-2}vdxdv-\f{d}{2 } \iint_{\mathbb{R}^{d}\times\mathbb{R}_{v} ^{d}} (\partial_{i}\partial_{j}f)^{2} \langle v\rangle^{2k}dxdv \\
&\quad
-k \iint_{\mathbb{R}^{d}\times\mathbb{R}_{v} ^{d}} v^{2} \cdot\nabla_{v}(\partial_{i}\partial_{j}f)^{2} \langle v\rangle^{2k-2}dxdv.
\ea$$
which means that
\be\ba\label{3.100}
|IV|\leq C\| u\|_{L^{\infty} _{x}}\left\|\langle v\rangle^{k} \partial_{i}\partial_{j}f \right\|_{L^{2}_{x,v}}^{2}+C\left\|\langle v\rangle^{k} \partial_{i}\partial_{j}f \right\|_{L^{2}_{x,v}}^{2}
\ea\ee
Inserting
\eqref{3.96}-\eqref{3.100} into \eqref{3.93}, one ends up with
\be\ba\label{3.101}
\frac{1}{2} \frac{d}{d t}\left\|\langle v\rangle^{k} \partial_{i}\partial_{j}f \right\|_{L^{2}_{x,v}}^{2}
&\leq
C\|\partial_{i}\partial_{j}
u\|_{L^{4} _{x}}\|\nabla\nabla_{v} f\langle v\rangle^{ k}\|_{L^{2}_{x,v}}^{\f{d}{4}}\left\|\langle v\rangle^{k} \partial_{i}\partial_{j}f \right\|_{L^{2}_{x,v} } \\
&\quad
+ C\|\nabla u\|_{L^{\infty} _{x}}\left\|\langle v\rangle^{k} \nabla_{v}\partial_{i}f \right\|_{L^{2}_{x,v}} \left\|\langle v\rangle^{k} \partial_{i}\partial_{j}f \right\|_{L^{2}_{x,v}} \\
&\quad
+ C(\|\nabla u\|_{L^{\infty} _{x}}+1) \left\|\langle v\rangle^{k} \partial_{i}\partial_{j}f \right\|_{L^{2}_{x,v}}^{2}.
\ea\ee
To close the estimate, we need to bound $\|\nabla\nabla_{v} f\langle v\rangle^{ k}\|_{L^{2}_{x,v}}$.
A slight modification of deduction of \eqref{3.92} leads to
\be
\partial_{i}\partial_{v_{j}}f_{t}
+\partial_{v_{j}}v\cdot\nabla_{x}\partial_{i}f
+ v\cdot\nabla_{x}\partial_{i}\partial_{v_{j}}f
+\partial_{i}
u\nabla_{v} \partial_{v_{j}}f -\partial_{v_{j}}v \cdot\nabla_{v}\partial_{i}f
+(u-v)\cdot\nabla_{v}\partial_{i}\partial_{v_{j}}f
=d\partial_{v_{j}}\partial_{i}f.
\ee
Thanks to $\iint_{\mathbb{R}^{d}\times\mathbb{R}_{v}^{d}}
v\cdot\nabla_{x}\partial_{i}\partial_{v_{j}}f\partial_{i}\partial_{v_{j}}f \langle v\rangle^{2k}dxdv=0,$
we multiply the above equation by $\partial_{i}\partial_{v_{j}}f \langle v\rangle^{2k}$ and integrate over $\mathbb{R}^{d}\times\mathbb{R}^{d}_{v}$ to get
\be\ba
&\frac{1}{2} \frac{d}{d t}\left\|\langle v\rangle^{k} \partial_{i}\partial_{v_{j}}f \right\|_{L^{2}_{x,v}}^{2}\\
=&
-\iint_{\mathbb{R}^{d}\times\mathbb{R}_{v}^{d}} \partial_{i}
u\nabla_{v} \partial_{v_{j}}f\partial_{i}\partial_{v_{j}}f \langle v\rangle^{2k}dxdv+\iint_{\mathbb{R}^{d}\times\mathbb{R}_{v}^{d}} \partial_{v_{j}}v \cdot\nabla_{v}\partial_{i}f\partial_{i}\partial_{v_{j}}f \langle v\rangle^{2k}dxdv \\
&\quad
-\iint_{\mathbb{R}^{d}\times\mathbb{R}_{v}^{d}} (u-v)\cdot\nabla_{v}\partial_{i}\partial_{v_{j}}f
\partial_{i}\partial_{v_{j}}f \langle v\rangle^{2k}dxdv+
d\iint_{\mathbb{R}^{d}\times\mathbb{R}_{v}^{d}}\partial_{v_{j}}\partial_{i}f\partial_{i}\partial_{v_{j}}f \langle v\rangle^{2k}dxdv \\
&\quad
-\iint_{\mathbb{R}^{d}\times\mathbb{R}_{v}^{d}}
\partial_{v_{j}}v\cdot\nabla_{x}\partial_{i}f\partial_{i}\partial_{v_{j}}f \langle v\rangle^{2k}dxdv\\
=&: I'+II'+III'+IV'+V'.
\ea\ee
Applying H\"older's inequality, one has
\be\ba\label{3.37v1}
&|I'|\leq C \|\nabla u\|_{L^{\infty}_{x} }\left\|\langle v\rangle^{k} \nabla_{v}\partial_{v_{j}}f \right\|_{L^{2}_{x,v}} \left\|\langle v\rangle^{k} \partial_{i}\partial_{v_{j}}f \right\|_{L^{2}_{x,v}},
\\
&|II'|\leq C \left\|\langle v\rangle^{k} \nabla_{v}\partial_{i}f \right\|_{L^{2}_{x,v}} \left\|\langle v\rangle^{k} \partial_{i}\partial_{v_{j}}f \right\|_{L^{2}_{x,v}}.
\ea\ee
It is clear that
\be\ba\label{3.38v1}
|IV'|+|V'|\leq C \left\|\langle v\rangle^{k} \partial_{v_{j}}\partial_{i}f \right\|_{L^{2}_{x,v}}^{2}+C \left\|\langle v\rangle^{k} \partial_{v_{j}}\partial_{i}f \right\|_{L^{2}_{x,v}} \left\|\langle v\rangle^{k} \nabla^{2}_{x} f \right\|_{L^{2}_{x,v}}.
\ea\ee
Integrating by parts, we arrive at
$$\ba
III'
 &=
\f12 \iint_{\mathbb{R}^{d}\times\mathbb{R}_{v}^{d}} u (\partial_{i}\partial_{v_{j}}f)^{2}
\langle v\rangle^{2k-2}vdxdv- n k
\iint_{\mathbb{R}^{d}\times\mathbb{R}_{v}^{d}} (\partial_{i}\partial_{v_{j}}f)^{2}
\langle v\rangle^{2k}dxdv \\
&\quad
-k \iint_{\mathbb{R}^{d}\times\mathbb{R}_{v}^{d}}|v|^{2} (\partial_{i}\partial_{v_{j}}f)^{2}
\langle v\rangle^{2k-2}dxdv.
\ea$$
from which it follows that
\be\ba\label{3.39v1}
|III'|\leq C\| u\|_{L^{\infty} _{x}}\left\|\langle v\rangle^{k} \partial_{i}\partial_{v_{j}}f \right\|_{L^{2}_{x,v}}^{2}+C \left\|\langle v\rangle^{k} \partial_{i}\partial_{v_{j}}f \right\|_{L^{2}_{x,v}}^{2}.
\ea\ee
Consequently, we deduce from \eqref{3.37v1}-\eqref{3.39v1} that
\be\ba\label{3.108}
&\frac{1}{2} \frac{d}{d t}\left\|\langle v\rangle^{k} \partial_{i}\partial_{v_{j}}f \right\|_{L^{2}_{x,v}}^{2} \\
&\leq
C \|\nabla u\|_{L^{\infty} _{x}}\left\|\langle v\rangle^{k} \nabla_{v}\partial_{v_{j}}f \right\|_{L^{2}_{x,v}} \left\|\langle v\rangle^{k} \partial_{i}\partial_{v_{j}}f \right\|_{L^{2}_{x,v}}
+C \left\|\langle v\rangle^{k} \partial_{v_{j}}\partial_{i}f \right\|_{L^{2}_{x,v}}^{2} \\
&\quad
+C\| u\|_{L^{\infty}_{x} }\left\|\langle v\rangle^{k} \partial_{i}\partial_{v_{j}}f \right\|_{L^{2}_{x,v}}^{2}+C \left\|\langle v\rangle^{k} \partial_{v_{j}}\partial_{i}f \right\|_{L^{2}_{x,v}} \left\|\langle v\rangle^{k} \nabla^{2}_{x} f \right\|_{L^{2}_{x,v}}.
\ea\ee
Applying $\partial_{v_{i}}\partial_{v_{j}}$ to the Vlasov equation $\eqref{NSV}_{2}$ yields
$$\ba
\partial_{v_{i}}\partial_{v_{j}}f_{t}
+\partial_{v_{j}}v\cdot\nabla_{x}\partial_{v_{i}}f
+\partial_{v_{i}}v\cdot\nabla_{x}\partial_{v_{j}}f
&+v\cdot\nabla_{x}\partial_{v_{j}}\partial_{v_{i}}f -\partial_{v_{j}}v \cdot\nabla_{v}\partial_{v_{i}}f
\\&-\partial_{v_{i}}v \cdot\nabla_{v}\partial_{v_{j}}f+(u-v)\cdot\nabla_{v}\partial_{v_{j}}\partial_{v_{i}}f
=d\partial_{v_{j}}\partial_{v_{i}}f.
\ea$$
Taking the $L^2$ inner product of the above equation with $\partial_{v_{i}}\partial_{v_{j}}f \langle v\rangle^{2k}$, we discover that
\be\ba\label{3.110}
&\frac{1}{2} \frac{d}{d t}\left\|\langle v\rangle^{k} \partial_{v_{i}}\partial_{v_{j}}f \right\|_{L^{2}_{x,v}}^{2}\\
&=
\iint_{\mathbb{R}^{d}\times\mathbb{R}_{v}^{d}} \partial_{v_{j}}v\cdot\nabla_{v} \partial_{v_{j}}f\partial_{v_{i}}\partial_{v_{j}}f \langle v\rangle^{2k}dxdv+\iint_{\mathbb{R}^{d}\times\mathbb{R}_{v}^{d}} \partial_{v_{i}}v \cdot\nabla_{v}\partial_{v_{i}}f\partial_{v_{i}}\partial_{v_{j}}f \langle v\rangle^{2k}dxdv\\
&\quad
-\iint_{\mathbb{R}^{d}\times\mathbb{R}_{v}^{d}} (u-v)\cdot\nabla_{v}\partial_{v_{i}}\partial_{v_{j}}f
\partial_{v_{i}}\partial_{v_{j}}f \langle v\rangle^{2k}dxdv+
d\iint_{\mathbb{R}^{d}\times\mathbb{R}_{v}^{d}} \partial_{v_{j}}\partial_{v_{i}}f\partial_{v_{i}}\partial_{v_{j}}f \langle v\rangle^{2k}dxdv\\
&\quad
-\iint_{\mathbb{R}^{d}\times\mathbb{R}_{v}^{d}} \partial_{v_{j}}v\cdot\nabla_{x}\partial_{v_{i}}f\partial_{v_{i}}\partial_{v_{j}}f \langle v\rangle^{2k}dxdv
-\iint_{\mathbb{R}^{d}\times\mathbb{R}_{v}^{d}} \partial_{v_{i}}v\cdot\nabla_{x}\partial_{v_{j}}f\partial_{v_{i}}\partial_{v_{j}}f \langle v\rangle^{2k}dxdv\\
&=: I''+II''+III''+IV''+V''+VI''.
\ea\ee
Here we have used
$\iint_{\mathbb{R}^{d}\times\mathbb{R}_{v}^{d}}
v\cdot\nabla_{x}\partial_{v_{j}}\partial_{v_{i}}f\partial_{v_{i}}\partial_{v_{j}}f \langle v\rangle^{2k}dxdv=0.$

A straightforward calculation means
\be\ba\label{3.111}
|I''|+|II''|+ |IV''|\leq C \left\|\langle v\rangle^{k} \partial_{v_{i}}\partial_{v_{j}}f \right\|_{L^{2}_{x,v}}^{2}+C \left\|\langle v\rangle^{k} \partial_{v_{i}}\partial_{v_{j}}f \right\|_{L^{2}_{x,v}} \left\|\langle v\rangle^{k} \nabla_{x}\partial_{v_{i}} f \right\|_{L^{2}_{x,v}} .
\ea\ee
By means of integrations by parts, one obtains
$$\ba
III''
&=
k \iint_{\mathbb{R}^{d}\times\mathbb{R}_{v}^{d}} u (\partial_{v_{i}}\partial_{v_{j}}f)^{2}
\langle v\rangle^{2k-2}vdxdv+\f{d   }{2}
\iint_{\mathbb{R}^{d}\times\mathbb{R}_{v}^{d}} (\partial_{v_{i}}\partial_{v_{j}}f)^{2}
\langle v\rangle^{2k}dxdv \\
&\quad
-k \iint_{\mathbb{R}^{d}\times\mathbb{R}_{v}^{d}} -|v|^{2} (\partial_{v_{i}}\partial_{v_{j}}f)^{2}
\langle v\rangle^{2k-2}dxdv.
\ea$$
which implies that
\be\ba\label{3.112}
|III''|\leq C\| u\|_{L^{\infty} _{x}}\left\|\langle v\rangle^{k} \partial_{v_{i}}\partial_{v_{j}}f \right\|_{L^{2}_{x,v}}^{2}+C \left\|\langle v\rangle^{k} \partial_{v_{i}}\partial_{v_{j}}f \right\|_{L^{2}_{x,v}}^{2}.
\ea\ee
Putting estimates \eqref{3.110}-\eqref{3.112} together, we know that
\be\ba\label{3.114}
\frac{1}{2} \frac{d}{d t}\left\|\langle v\rangle^{k} \partial_{v_{i}}\partial_{v_{j}}f \right\|_{L^{2}_{x,v}}^{2}
&\leq
C\| u\|_{L^{\infty} _{x}}\left\|\langle v\rangle^{k} \partial_{v_{i}}\partial_{v_{j}}f \right\|_{L^{2}_{x,v}}^{2}
+C \left\|\langle v\rangle^{k} \partial_{v_{i}}\partial_{v_{j}}f \right\|_{L^{2}_{x,v}}^{2} \\
&\quad
+ C \left\|\langle v\rangle^{k} \partial_{v_{i}}\partial_{v_{j}}f \right\|_{L^{2}_{x,v}} \left\|\langle v\rangle^{k} \nabla_{x}\partial_{v_{i}} f \right\|_{L^{2}_{x,v}}.
\ea\ee
It follows from
\eqref{3.101}, \eqref{3.108} and \eqref{3.114} that
\be\ba\label{3.49v1}
\frac{1}{2} \frac{d}{d t}\left\|\langle v\rangle^{k} D^{2}f \right\|_{L^{2}_{x,v}}^{2}
&\leq
C\|\partial_{i}\partial_{j} u\|_{L^{4} _{x} }\left\|\langle v\rangle^{k} \nabla^{2}f \right\|_{L^{2}_{x,v}}^{\f{n}{4}+1}
+ C\|\nabla u\|_{L^{\infty} _{x} }\left\|\langle v\rangle^{k} \nabla^{2}f \right\|_{L^{2}_{x,v}}^{2} \\
&\quad
+ C\left\|\langle v\rangle^{k} \nabla^{2}f \right\|_{L^{2}_{x,v}}^{2}
+C\| u\|_{L^{\infty}_{x}}\left\|\langle v\rangle^{k} \nabla^{2}f \right\|_{L^{2}_{x,v}}^{2}.
\ea\ee
Finally, applying Gr\"onwall's inequality \eqref{gronwall3} for $n<4$  and \eqref{gronwall1} for $n=4$ in Lemma \ref{gronwall} yields the desired result.
\end{proof}

\begin{proof}[Proof of Theorem \ref{newthe1.1}] We finish the proof this theorem  via estimate  \eqref{3.7v1}, \eqref{pro2v1}, and
\eqref{3.49v1}.
\end{proof}
\section{Regularity of weak solutions of the 3D Navier--Stokes--Vlasov system}

This section is devoted to the proof of Theorem \ref{the1.1}. Building on the moment propagation results established in Theorem \ref{newthe1.1}, we first prove regularity criteria in terms of the velocity field and its gradient in Subsection 4.1. Then, in Subsection 4.2, we derive regularity criteria based on the pressure, using the results obtained in Subsection 4.1.

\subsection{LPS regularity
criteria involving the velocity}

\begin{lemma}\label{serrin}
Let $u_{0}\in H^{1}(\mathbb{R}^{3})$, $f_{0}v^{3}, f_{0}v^{5}\in L^{1}(\mathbb{R}^{3}\times\mathbb{R}_{v}^{3})$.
Let $(u,f)$ be the weak solutions to the Navier--Stokes--Vlasov system $\eqref{NSV}_2$.
If the velocity field satisfies one of the following conditions,
\begin{enumerate}[(1)]
\item
$u\in L^{p}(0,T;L^{q}(\mathbb{R}^{3}))$ with $2/p+3/q=1$, $q\geq24/7$;
\item $\nabla u\in L^{p}(0,T;L^{q}(\mathbb{R}^{3}))$ with $2/p+3/q=2$, $q\geq8/5$,
\end{enumerate}
then there hold
\be\ba\label{4.1} \nabla u\in L^{\infty}(0,T;L^{2}(\mathbb{R}^{3})), \nabla^{2} u\in L^{2}(0,T;L^{2}(\mathbb{R}^{3})),u\in L^{2}(0,T;L^{\infty}(\mathbb{R}^{3})).
\ea\ee
\end{lemma}
\begin{remark}Roughly speaking,
with the critical estimate $u\in L^{2}(0,T;L^{\infty}(\mathbb{R}^{3}))$, we immediately conclude by moment estimates \eqref{3.2} that, for any $1\leq p<\infty$,
$$n, j \in L^{\infty}(0, T; L^{p}(\mathbb{R}^{3} \times \mathbb{R}^{3})).$$
This observation will be
frequently used in this section.
\end{remark}
\begin{proof}
Multiplying the velocity equation by $-\Delta u$ and integrating over $\mathbb{R}^{3}$, we use integration by parts and the incompressibility condition to obtain
\be\ba\label{3.27}
\f12\f{d}{dt}\int_{\mathbb{R}^{3}}|\nabla u|^{2}dx+\int_{\mathbb{R}^{3}}|\nabla^{2}u|dx
&= \int_{\mathbb{R}^{3}} u\cdot\nabla u \Delta udx+ \int_{\mathbb{R}^{3}}un \Delta udx-\int_{\mathbb{R}^{3}}j \Delta udx \\
&=: I+II+III.
\ea\ee
By H\"older's and Young's inequalities, we estimate the term $II$ as
\be\ba\label{4.10}
|II|&\leq \|u\|_{L^{\infty}(\mathbb{R}^{3})} \|n\|_{L^{2}(\mathbb{R}^{3})} \|\nabla^{2}u\|_{L^{2}(\mathbb{R}^{3})}.
\ea\ee
Using the Gagliardo--Nirenberg inequality together with the Sobolev embedding, we get
$$\ba
\|u\|_{L^{\infty}(\mathbb{R}^{3})}\leq C\|u\|^{\f12}_{L^{6}(\mathbb{R}^{3})}
\|\nabla^{2}u\|_{L^{2}(\mathbb{R}^{3})}^{\f12}\leq C\|\nabla u\|^{\f12}_{L^{2}(\mathbb{R}^{3})}
\|\nabla^{2}u\|_{L^{2}(\mathbb{R}^{3})}^{\f12}
\ea$$
Substituting this into \eqref{4.10} yields
\be\ba\label{4.11}
|II|
&\leq
\|u\|_{L^{\infty}(\mathbb{R}^{3})} \|n\|_{L^{2}(\mathbb{R}^{3})} \|\nabla^{2}u\|_{L^{2}(\mathbb{R}^{3})} \\
&\leq
C\|\nabla u\|^{\f12}_{L^{2}(\mathbb{R}^{3})}
\|\nabla^{2}u\|_{L^{2}(\mathbb{R}^{3})}^{\f32} \|n\|_{L^{2}(\mathbb{R}^{3})}\\
&\leq
C \|n\|^{4}_{L^{2}(\mathbb{R}^{3})}\|\nabla u\|^{2}_{L^{2}(\mathbb{R}^{3})}+\f{1}{32}\|\nabla^{2}u\|_{L^{2}(\mathbb{R}^{3})}^{2}.
\ea\ee
By H\"older's and Young's inequalities, we estimate the term $III$ as
\be\ba\label{3.34}
|III|
&\leq
\|j\|_{L^{2}(\mathbb{R}^{3})} \|\nabla^{2}u\|_{L^{2}(\mathbb{R}^{3})}\\
&\leq
C\|j\|^{2}_{L^{2}(\mathbb{R}^{3})} +\f{1}{32}\|\nabla^{2}u\|_{L^{2}(\mathbb{R}^{3})}^{2}.
\ea\ee
(1) $u\in L^{p}(0,T;L^{q}(\mathbb{R}^{3}))$ with $2/p+3/q=1$ for
$\f{24}{7} \leq q\leq\infty$.

By Young's inequality, we estimate
\be\ba\label{3.28}
|I|\leq C\int _{\mathbb{R}^{3}}|u|^{2}|\nabla u|^{2}+\f{1}{64}\| \nabla^{2}u\|^{2}_{L^{2}(\mathbb{R}^{3})}.
\ea\ee
In the light of the H\"older inequality, we notice that
\be\ba\label{3.29}
\int _{\mathbb{R}^{3}}|u|^{2}|\nabla u|^{2} dx &\leq
\|u\|^{2}_{L^{ q }(\mathbb{R}^{3})}\| \nabla u\|^{2}_{L^{\f{2q}{q-2} } (\mathbb{R}^{3})}.
\ea\ee
By the Gagliardo--Nirenberg interpolation inequality, we have
\be\label{3.30}
\|\nabla u \| _{L^{\f{2q}{q-1} }(\mathbb{R}^{3})}\leq C\| \nabla u \|^{1-\f{3}{q}}_{L^{2}(\mathbb{R}^{3})}\|\nabla^{2} u\|^{\f{3}{q}}_{L^{2}(\mathbb{R}^{3})}.\ee
Combining \eqref{3.29} and \eqref{3.30}, and using Young's inequality, we get
\be\ba\label{3.31}
\int _{\mathbb{R}^{3}}|u|^{2}|\nabla u|^{2} dx
&\leq
\|u\|^{2}_{L^{ q }(\mathbb{R}^{3})}\| \nabla u\|^{2}_{L^{\f{2q}{q-2} } (\mathbb{R}^{3})} \\
&\leq
C\| u\|^{2}_{L^{q }(\mathbb{R}^{3})}\| \nabla^2 u \|^{\f6q}_{L^{2}(\mathbb{R}^{3})}\| \nabla u\|^{2-\f{6}{q}}_{L^{ 2} (\mathbb{R}^{3})}\\
&\leq
C\| u\|^{\f{2q}{q-3}}_{L^{q }(\mathbb{R}^{3})}\| \nabla u\|^{2}_{L^{ 2} (\mathbb{R}^{3})}+\f{1}{64}\| \nabla^2 u \|^{2}_{L^{2}(\mathbb{R}^{3})}.
\ea\ee
Substituting \eqref{3.31} into \eqref{3.28}, we obtain
\be\ba\label{3.32}
|I|\leq C\| u\|^{\f{2q}{q-3}}_{L^{q }(\mathbb{R}^{3})}\| \nabla u\|^{2}_{L^{ 2} (\mathbb{R}^{3})}+\f{1}{32}\| \nabla^{2}u\|^{2}_{L^{2}(\mathbb{R}^{3})}.
\ea\ee
Plugging \eqref{4.11} and \eqref{3.34} into \eqref{3.27}, we conclude
\be\ba\label{4.120}
&\f12\f{d}{dt} \int_{\mathbb{R}^{3}}|\nabla u|^{2}dx+\f12\int_{\mathbb{R}^{3}}|\nabla^{2}u|dx\\
&\leq
C\| u\|^{\f{2q}{q-3}}_{L^{q }(\mathbb{R}^{3})}\| \nabla u\|^{2}_{L^{ 2} (\mathbb{R}^{3})}+C\|n\|^{4}_{L^{2}(\mathbb{R}^{3})}\|\nabla u\|^{2}_{L^{2}(\mathbb{R}^{3})}+\|j\|^{2}_{L^{2}(\mathbb{R}^{3})}.
\ea\ee
With this in hand, it suffices to show that $j\in L^{2}(0,T;L^{2}(\mathbb{R}^{3}))$ and $n\in L^{4}(0,T;L^{2}(\mathbb{R}^{3}))$.
The key tool is the moment propagation result \eqref{03.1}.
In what follows, we divide the proof into three cases depending on the range of the exponent $q$.

$\bullet$ $u\in L^{p}(0,T;L^{q}(\mathbb{R}^{3}))$ with $2/p+3/q=1$ with $q\geq8$.

From moment estimate \eqref{3.2}, we deduce
\be\ba\label{4.160v3}
 \|n\|_{L^{\infty}(0,T;L_{x}^{2})} \leq C \|f_{0}\|_{L^{\infty}_{ x,v}}^{\f{ 1}{2}}\B[ \|f_{0} |v|^{3 } \| _{L^{1}_{x,v} }+ C \|f_{0} \|_{L^{\infty}_{x,v}} \B(\int_{0}^{T}\|u\|_{L_{x}^{6 }}ds\B)^{6}\B]^{\f{1}{2}},
\ea\ee
and
\be\ba\label{4.160}  \|j\|_{L^{\infty}(0,T;L_{x}^{2})} \leq C \|f_{0}\|_{L^{\infty}_{ x,v}}^{\f{ 1}{2}}\B[ \|f_{0} |v|^{5 } \| _{L^{1}_{x,v} }+ C \|f_{0} \|_{L^{\infty}_{x,v}} \B(\int_{0}^{T}\|u\|_{L_{x}^{8 }}ds\B)^{8}\B]^{\f{1}{2}}.
\ea\ee
These estimates imply that
\[
n\in L^{4}(0,T;L^{2}(\mathbb{R}^{3})), \quad j\in L^{2}(0,T;L^{2}(\mathbb{R}^{3})).
\]
Thus, by applying Gr\"onwall's inequality \eqref{gronwall1}   in Lemma \ref{gronwall}, we conclude by \eqref{4.120} that
$\nabla u\in L^{\infty}(0,T;L^{2}(\mathbb{R}^{3}))$ and $\nabla^{2} u\in L^{2}(0,T;L^{2}(\mathbb{R}^{3}))$ provided that $u\in L^{p}(0,T;L^{q}(\mathbb{R}^{3}))$ with $2/p+3/q=1$ with $q\geq8$.

$\bullet$ $u\in L^{p}(0,T;L^{q}(\mathbb{R}^{3}))$ with $2/p+3/q=1$ for $\f{24}{7}\leq q \leq6$.

From the preceding analysis, we observe that the restriction $q \geq 8$ originated from the requirement $j\in L^{2}(0,T;L^{2}(\mathbb{R}^{3}))$.
In this case, our goal is to establish that $u\in L^{1}(0,T;L^{8}(\mathbb{R}^{3}))$.
It is straightforward to show that $\f{8q}{q+2}\leq 6$ in this case.
Since $u\in L^{\infty}(0,T;L^{2}(\mathbb{R}^{3}))\cap L^{2}(0,T;L^{6}(\mathbb{R}^{3}))$, interpolation gives $u\in L^{2}(0,T;L^{\tilde{q}}(\mathbb{R}^{3}))$ for $\tilde{q}\in [2,6]$, and in particular, $u\in L^{2}(0,T;L^{\f{8q}{q+2}}(\mathbb{R}^{3}))$. It follows from \eqref{3.2}
that
$$\ba
 \|j\|_{L^{\infty}(0,T;L_{x}^{\f{2q}{q+2}})} \leq C \|f_{0}\|_{L^{\infty}_{ x,v}}^{\f{ q-2}{2q}}\B[ \|f_{0} |v|^{\f{8q}{q+2}-3 } \| _{L^{1}_{x,v} }+ C \|f_{0} \|_{L^{\infty}_{x,v}} \B(\int_{0}^{T}\|u\|_{L_{x}^{\f{8q}{q+2}}ds}\B)^{\f{8q}{q+2}}\B]^{\f{q+2}{2q} }.
\ea$$
This implies that
\be\label{4.13}
j\in L^{\f{2q}{2q-3}}(0,T;L^{\f{2q}{q+2}}(\mathbb{R}^{3})).
\ee
By H\"older's inequality, we estimate
\be\ba\label{4.14}
&\|-u \cdot \nabla u-nu\|_{L^{\f{2q}{2q-3}}(0,T;L^{\f{2q}{q+2}}(\mathbb{R}^{3}))} \\
&\leq C
\|u\| _{L^{\f{2q}{q-3}}(0,T;L^{q}(\mathbb{R}^{3}))}(\|n\|_{L^{2}(0,T;L^{2}(\mathbb{R}^{3}))}+ \|\nabla u\|_{L^{2}(0,T;L^{2} (\mathbb{R}^{3}))}).
\ea\ee
We now reformulate the velocity equations $\eqref{NSV}_{1}$ as
\be\ba\label{4.15}
&\partial_t u -  \Delta u +\nabla \Pi = -u \cdot \nabla u-nu+j.
\ea\ee
Combining \eqref{4.13}-\eqref{4.15} and applying the $L_{t}^{p}L_{x}^{q}$ estimates due to Giga and Sohr, we deduce
\be\label{4.16}
\nabla^{2}u\in L^{\f{2q}{2q-3}}(0,T;L^{\f{2q}{q+2}}(\mathbb{R}^{3})).
\ee
For the borderline case $q=\f{24}{7}$, we use the Sobolev embedding twice to conclude
$$\| u\| _{L^{8}(\mathbb{R}^{3})}\leq C \|\nabla u\| _{L^{\f{24}{11}}(\mathbb{R}^{3})}\leq C \|\nabla^{2}u\| _{L^{\f{24}{19}}(\mathbb{R}^{3})}.$$
Integrating in time gives
$$\| u\| _{L^{\f{48}{27}}(0,T;L^{8}(\mathbb{R}^{3}))}\leq C \|\nabla u\| _{L^{\f{48}{27}}(0,T;L^{\f{24}{11}}(\mathbb{R}^{3}))}\leq C \|\nabla^{2}u\| _{L^{\f{48}{27}}(0,T;L^{\f{24}{19}}(\mathbb{R}^{3}))},
$$
which shows that $u\in L^{1}(0,T;L^{8}(\mathbb{R}^{3}))$.

Let us now consider the rest case $\f{24}{7}<q\leq6$.
In this case, the Gagliardo--Nirenberg inequality yields
\be\label{4.16v3}
\|\nabla u\| _{L^{\f{24}{11}}(\mathbb{R}^{3})}\leq C \| u\| _{L^{2}(\mathbb{R}^{3})} ^{\f{7q-24}{8(3+q)}}\|\nabla^{2}u\| _{L^{\f{2q}{q+2}}(\mathbb{R}^{3})}^{\f{48+q }{8(3+q)}}.
\ee
We emphasize that this inequality is valid for $\f{24}{7} < q \leq 12$, and it will be used later as well.

Integrating \eqref{4.16v3} in time, we get
\be\label{4.18}
\|\nabla u\| _{L^{\f{2q}{2q-3}}(0,T;L^{\f{24}{11}}(\mathbb{R}^{3}))}\leq C \| u\| _{L^{\infty}(0,T;L^{2}(\mathbb{R}^{3}))} ^{\f{7q-24}{8(3+q)}}\|\nabla^{2}u\| _{L^{\f{2q}{2q-3}}(0,T;L^{\f{2q}{q+2}}(\mathbb{R}^{3}))}^{\f{48+q }{8(3+q)}}.
\ee
Then, by Sobolev embedding, we conclude that $u\in L^{\f{2q}{2q-3}}(0,T;L^{8}(\mathbb{R}^{3}))$.
This completes the proof for this case.

$\bullet$ $u\in L^{p}(0,T;L^{q}(\mathbb{R}^{3}))$ with $2/p+3/q=1$ for $6\leq q \leq8$.

From the assumption $u\in L^{\f{2q}{q-3}} (0,T;L^{q}(\mathbb{R}^{3}))$, we get $u\in L^{2} (0,T;L^{q}(\mathbb{R}^{3}))$.
Combined with $u\in L^{\infty}(0,T;L^{2}(\mathbb{R}^{3}))$, interpolation yields $u\in L^{2}(0,T;L^{\tilde{q}}(\mathbb{R}^{3}))$ for $\tilde{q}\in [2,q]$.
It is easy to check that $\f{8q}{q+2}\in [6.4,8]$ for $6\leq q \leq8$.
Therefore, we have $u\in L^{2}(0,T;L^{\f{8q}{q+2}}(\mathbb{R}^{3}))$.
As noted earlier, \eqref{4.16v3} is valid for $ \f{24}{7}<q \leq 12 $, which includes the current range of $q$.
Repeating the argument leading to \eqref{4.13}-\eqref{4.18}, we conclude that $u\in L^{\f{2q}{2q-3}}(0,T;L^{8}(\mathbb{R}^{3}))$.
This completes the proof for this case.

(2)
For $\frac{8}{5} \leq q < 3$, the conclusion follows directly from the Sobolev embedding.
Hence, it suffices to focus on the case $q \geq 3$.
Using integration by parts and the incompressibility condition, we obtain
\be\ba
\B|\int_{\mathbb{R}^{3}} u\cdot\nabla u \Delta udx\B|\leq
\int_{\mathbb{R}^{3}} |\nabla u|^{3}dx.
\ea\ee
From \eqref{3.30}, we have
\be\label{4.150}
\|\nabla u \|^{2}_{L^{\f{2q}{q-1} }(\mathbb{R}^{3})}\leq C\| \nabla u \|^{2-\f{3}{q}}_{L^{2}(\mathbb{R}^{3})}\|\nabla^{2} u\|^{\f{3}{q}}_{L^{2}(\mathbb{R}^{3})}.\ee
Applying H\"older's inequality, \eqref{4.150}, and Young's inequality, we estimate
\be\label{2.20}\ba
\int_{\mathbb{R}^{3}} |\nabla u|^{3}dx
&\leq
\|\nabla u\|_{L^{q }(\mathbb{R}^{3})}\|\nabla u \|^{2}_{L^{\f{2q}{q-1},2}(\mathbb{R}^{3})} \\
&\leq
C\|\nabla u\|_{L^{q }(\mathbb{R}^{3})}\| \nabla u \|^{2-\f{3}{q}}_{L^{2}(\mathbb{R}^{3})}\|\nabla^{2} u\|^{\f{3}{q}}_{L^{2}(\mathbb{R}^{3})}\\
&\leq
C \|\nabla u\|^{\f{2q}{2q-3}}_{L^{q }(\mathbb{R}^{3})}\|\nabla u\|_{L^{2}(\mathbb{R}^{3})}^{2}
+\f{1}{32}\| \nabla^{2}u\|^{2}_{L^{2}(\mathbb{R}^{3})}.
\ea\ee
Substituting this into the energy inequality \eqref{4.120}, we obtain
\be\ba
&\f12\f{d}{dt}\int_{\mathbb{R}^{3}}|\nabla u|^{2}dx+\int_{\mathbb{R}^{3}}|\nabla^{2}u|dx \\
&\leq
C \|\nabla u\|^{\f{2q}{2q-3}}_{L^{q }(\mathbb{R}^{3})}\|\nabla u\|_{L^{2}(\mathbb{R}^{3})}^{2}+\|n\|^{4}_{L^{2}(\mathbb{R}^{3})}\|\nabla u\|^{2}_{L^{2}(\mathbb{R}^{3})}+\|j\|^{2}_{L^{2}(\mathbb{R}^{3})} +\f{1}{8}\| \nabla^{2}u\|^{2}_{L^{2}(\mathbb{R}^{3})}.
\ea\ee
It remains to show that $n\in L^{4}(0,T;L^{2}(\mathbb{R}^{3}))$ and $j\in L^{2}(0,T;L^{2}(\mathbb{R}^{3}))$. To this end,
we apply the Gagliardo--Nirenberg inequality for $q\geq3$ to deduce that
\be
\ba
\|u\|_{L^{\f{2q}{2q-3}}(0,T;L^{8}(\mathbb{R}^{3}))}\leq C\|u\|_{L^{\infty}(0,T;L^{2}(\mathbb{R}^{3}))}^{\f{1-\f3q+\f38}{1-\f3q+\f32}}
\|\nabla u\|_{L^{\f{2q}{2q-3}}(0,T;L^{q}(\mathbb{R}^{3}))}^{\f{\f98}{1-\f3q+\f32}}.
\ea\ee
This together with \eqref{4.160} leads to $j\in L^{2}(0,T;L^{2}(\mathbb{R}^{3}))$. In addition, we derive from $ u\in L^{2}(0,T;L^{6}(\mathbb{R}^{3}))$ that $n\in L^{\infty}(0,T;L^{2}(\mathbb{R}^{3}))$ by \eqref{4.160v3}.

Finally, in all cases discussed above, we have shown that
\[
\nabla u\in L^{\infty}(0,T;L^{2}(\mathbb{R}^{3})), \quad \nabla^{2} u\in L^{2}(0,T;L^{2}(\mathbb{R}^{3})).
\]
Applying the Gagliardo-Nirenberg inequality once more, we conclude that $ u\in L^{2}(0,T;L^{\infty}(\mathbb{R}^{3}))$.
This completes the proof of the lemma.
\end{proof}
Next step is to deduce the Lipschitz control of the flow of the velocity field, which allows us to present the    gradient estimate of the particle distribution function (see Lemma \ref{lemma3.3}).
\begin{lemma}\label{lemma4.2}
Assume the conditions of Lemma \ref{serrin} hold.
Suppose further that the initial data satisfy $f_{0}v^{6+3\eta}, f_{0}v^{9+4\eta } \in L^{1}(\mathbb{R}^{3}\times\mathbb{R}_{v}^{3})$ for some sufficiently small constant $\eta > 0$.
Then we have $\nabla u\in L^{1}(0,T;L^{\infty} (\mathbb{R}^{3}))$.
\end{lemma}

\begin{proof}
Since $\nabla u\in L^{\infty}(0,T;L^{2}(\mathbb{R}^{3}))$ and $\nabla^{2} u\in L^{2}(0,T;L^{2}(\mathbb{R}^{3}))$, the classical interpolation inequality implies $\nabla u\in L^{\alpha}(0,T;L^{\beta}(\mathbb{R}^{3}))$ for all $\alpha$ and $\beta$ satisfying $2/\alpha+3/\beta=3/2$ and $2\leq\beta\leq6$.
In particular, we obtain, for $0<\eta<3$,
$$
\nabla u\in L^{\f{4(3+\eta)}{3+3\eta}}(0,T;L^{3+ \eta}(\mathbb{R}^{3})).
$$
By H\"older's inequality, we get
\be\ba\label{3.36}
\|u\cdot\nabla u\|_{L^{\f{12+4\eta}{9+5\eta}}(0,T;L^{3+3\eta}(\mathbb{R}^{3}))}\leq \|u\| _{L^{2}(0,T;L^{\infty}(\mathbb{R}^{3}))}\|\nabla u\|_{L^{\f{4(3+\eta)}{3+3\eta}}(0,T;L^{3+3\eta}(\mathbb{R}^{3}))}.
\ea\ee
In view of  moments propagation \eqref{3.2}, we have
\be\ba\label{3.370}
 \|n\|_{L^{\infty}(0,T;L_{x}^{3+\eta})} \leq C \|f_{0}\|_{L^{\infty}_{ x,v}}^{\f{2+\eta}{3+\eta}}\B[ \|f_{0} |v|^{6+3\eta } \| _{L^{1}_{x,v} }+ C \|f_{0} \|_{L^{\infty}_{x,v}} \B(\int_{0}^{T}\|u\|_{L_{x}^{9+3\eta }}ds\B)^{9+3\eta }\B]^{^{\f{1}{3+\eta}} },
\ea\ee
and
\be\ba\label{3.38} \|j\|_{L^{\infty}(0,T;L_{x}^{3+\eta})} \leq C \|f_{0}\|_{L^{\infty}_{ x,v}} ^{\f{2+\eta}{3+\eta}}\B[ \|f_{0} |v|^{9+4\eta } \| _{L^{1}_{x,v} }+ C \|f_{0} \|_{L^{\infty}_{x,v}} \B(\int_{0}^{T}\|u\|_{L_{x}^{12+4\eta }}ds\B)^{12+4\eta}\B]^{\f{1}{3+\eta} }.
\ea\ee
Since
\[
 \|u n\|_{L^{3+\eta} (\mathbb{R}^{3} )} \leqslant\|u\|_{L^{\infty} (\mathbb{R}^{3} )} \|n\|_{L^{3+\eta} (\mathbb{R}^{3} )},
\]
we have
\be\ba\label{3.39}
 \|u n\|_{L^{2}(0,T;L^{3+\eta} (\mathbb{R}^{3} ))} \leq C \|u\| _{L^{2}(0,T;L^{\infty}(\mathbb{R}^{3}))}  \|n\|_{L^{\infty}(0,T;L^{3+\eta} (\mathbb{R}^{3} \times \mathbb{R}_{v}^{3} ))}.
\ea\ee
Rewriting the velocity equation $\eqref{NSV}_1$ as
\be
\partial_t u-  \Delta u +\nabla \Pi = -u \cdot \nabla u -\int_{\mathbb{R}_{v}^{3}} f(u-v)dv=F,
\ee
and combining \eqref{3.36}, \eqref{3.38} and \eqref{3.39}, we conclude that
$F\in L^{\f{12+4\eta}{9+5\eta}}(0,T;L^{3+3\eta}(\mathbb{R}^{3}))$.
Therefore, by the maximal regularity estimates for the Stokes system in \cite{[GS]}, it follows that $u\in L^{\f{12+4\eta}{9+5\eta}}(0,T;W^{2,3+3\eta} (\mathbb{R}^{3}))$ and $u_0\in W^{1,3+3\eta}(\mathbb{R}^{3})\subseteq H^{2}(\mathbb{R}^{3})$.
Finally, applying the Gagliardo--Nirenberg inequality, we obtain $\nabla u\in L^{1}(0,T;L^{\infty} (\mathbb{R}^{3}))$.
This completes the proof of the lemma.
\end{proof}
The Lipschitz control of   the velocity field and  moment estimates \eqref{3.2} help us to get
the time derivative
of the velocity field.
\begin{lemma}\label{lemma4.3}
Assume the conditions of Lemma \ref{lemma4.2} hold.
Suppose in addition that $u_{0t}\in L^{2}(\mathbb{R}^{3})$ and $ \langle v\rangle^{k} \nabla_{x}f_{0}, \langle v\rangle^{k}\nabla_{v}f_{0} \in L^{2}(\mathbb{R}^{3}\times\mathbb{R}^{3}_{v})$ for some $k> 7/2$.
Then we have
\be\label{4.23} u_{t}\in L^{\infty}(0,T; L^{2}(\mathbb{R}^{3})),\nabla u_{t}\in L^{2}(0,T; L^{2}(\mathbb{R}^{3})),u\in L^{\infty}(0,T; L^{\infty}(\mathbb{R}^{3})).
\ee
\end{lemma}

\begin{proof}% By the basic energy estimate of (1.8) and integrating by parts, we infer that
Applying the time derivative $\partial_{t}$ to the velocity equations yields
$$\ba \partial_{tt} u +u_{t} \cdot \nabla u+u \cdot \nabla u_{t} - \Delta u_{t} +\nabla \Pi_{t} =-\int_{\mathbb{R}_{v}^{3}} f u_{t}dv-\int_{\mathbb{R}_{v}^{3}} f_{t}(u-v)dv.
\ea$$
Using the Vlasov equation $\eqref{NSV}_{2}$, this equation becomes
$$\ba
&\partial_{tt} u+u \cdot \nabla u_{t} - \Delta u_{t} +\nabla \Pi_{t} \\
&=
-u_{t} \cdot \nabla u-\int_{\mathbb{R}_{v}^{3}} f u_{t}dv+\int_{\mathbb{R}_{v}^{3}}(u-v) \B[v\cdot\nabla_{x}f+\nabla_{v}\cdot[f(u-v) ]\B]dv.
\ea$$
Taking the $L^2$-inner product with $u_t$ and integrating over $\mathbb{R}^{3}$ yields
\be\ba\label{3.59}
&\f{d}{dt}\|u_{t}\|^{2}_{L^{2}(\mathbb{R}^{3})}+\|\nabla u_{t}\|^{2}_{L^{2}(\mathbb{R}^{3})}\\
&=
-\int_{\mathbb{R}^{3}}u_{t} \cdot \nabla uu_{t} dx- \int_{\mathbb{R}^{3}}|u_{t}|^{2} ndx+\int_{\mathbb{R}^{3}}u_{t}\int_{\mathbb{R}_{v}^{3}}(u-v) v\cdot\nabla_{x}fdvdx \\
&\quad
+\int_{\mathbb{R}^{3}}u_{t}
\int_{\mathbb{R}_{v}^{3}}(u-v) \nabla_{v}\cdot[f(u-v) ] dvdx.
\ea\ee
We estimate each term on the right-hand side.

Clearly, we have
\be\ba\label{3.60}
\B|-\int_{\mathbb{R}^{3}}u_{t} \cdot \nabla uu_{t} dx\B|
\leq \|\nabla u\|_{L^{\infty}(\mathbb{R}^{3})}\|u_{t}\|^{2}_{L^{2}(\mathbb{R}^{3})}.
\ea\ee
Using the H\"older inequality, Sobolev inequality, and the Young inequality, we get
\be\ba\label{3.61}
\B| - \int_{\mathbb{R}^{3}}|u_{t}|^{2} ndx\B|
&\leq
\|n\|_{L^{3}(\mathbb{R}^{3})}
\|u_{t}\|_{L^{2}(\mathbb{R}^{3})} \|u_{t}\|_{L^{6}(\mathbb{R}^{3})}\\
&\leq C\|n\|_{L^{3}(\mathbb{R}^{3})}
\|u_{t}\|_{L^{2}(\mathbb{R}^{3})} \|\nabla u_{t}\|_{L^{2}(\mathbb{R}^{3})} \\
&\leq C\|n\|^{2}_{L^{3}(\mathbb{R}^{3})}
\|u_{t}\|^{2}_{L^{2}(\mathbb{R}^{3})} + \f{1}{32}\|\nabla u_{t}\|^{2}_{L^{2}(\mathbb{R}^{3})}.
 \ea\ee
We split the third term on the right-hand side of \eqref{3.59} into two parts
\be\ba\label{3.62}
\int_{\mathbb{R}^{3}}\int_{\mathbb{R}_{v}^{3}}(u-v)(v\cdot\nabla_{x}f)dvu_{t}dx
= \int_{\mathbb{R}^{3}}uu_{t}\int_{\mathbb{R}_{v}^{3}}(v\cdot\nabla_{x}f)dvdx -\int_{\mathbb{R}^{3}} u_{t}\int_{\mathbb{R}_{v}^{3}}v( v \cdot\nabla_{x}f)dvdx.
 \ea\ee
By H\"older's inequality, we have
\be\ba \label{3.63}
\B|\int_{\mathbb{R}^{3}}uu_{t}\int_{\mathbb{R}_{v}^{3}}(v\cdot\nabla_{x}f)dvdx \B|\leq\| u\|_{L^{\infty}(\mathbb{R}^{3})}\|u_{t}\|_{L^{2}(\mathbb{R}^{3})}
\B\| \int_{\mathbb{R}_{v}^{3}}(|v||\nabla_{x}f| )dv\B\|_{L^{2}(\mathbb{R}^{3})},
\ea\ee
and
\be\ba\label{3.64}
\B|-\int_{\mathbb{R}^{3}} u_{t}\int_{\mathbb{R}_{v}^{3}}(|v|^{2}\cdot\nabla_{x}f)dvdx
\B|\leq \|u_{t}\|_{L^{2}(\mathbb{R}^{3})}
\B\| \int_{\mathbb{R}_{v}^{3}}|v|^{2}|\nabla_{x}f| dv\B\|_{L^{2}(\mathbb{R}^{3})}.
\ea\ee
Using Lemma \ref{lem2.2}, we see that for $k>5/2,$
\be\ba
\label{3.65}
\B\| \int_{\mathbb{R}_{v}^{3}}(|v||\nabla_{x}f| )dv\B\|_{L^{2}(\mathbb{R}^{3})}\leq C\|\langle v\rangle^{k}|\nabla_{x}f| \|_{L^{2}(\mathbb{R}^{3}\times\mathbb{R}_{v}^{3})},
\ea\ee
and for $k>7/2,$
\be\ba \label{3.66} \B\| \int_{\mathbb{R}_{v}^{3}}(|v|^{2}|\nabla_{x}f| )dv\B\|_{L^{2}(\mathbb{R}^{3})}\leq C\|\langle v\rangle^{k}|\nabla_{x}f| \|_{L^{2}(\mathbb{R}^{3}\times\mathbb{R}_{v}^{3})}.
\ea\ee
Plugging \eqref{3.65} into
\eqref{3.63}, we find
\be\ba \label{3.67} \B|\int_{\mathbb{R}^{3}}uu_{t}\int_{\mathbb{R}_{v}^{d}}(v\cdot\nabla_{x}f)dvdx \B|&\leq C\| u\|_{L^{\infty}(\mathbb{R}^{3})}\|u_{t}\|_{L^{2}(\mathbb{R}^{3})}
\|\langle v\rangle^{k} |\nabla_{x}f| \|_{L^{2}(\mathbb{R}^{3}\times\mathbb{R}_{v}^{3})}. \ea\ee
By the same way, we observe that
\be\ba \label{3.68}\B|-\int_{\mathbb{R}^{3}} u_{t}\int_{\mathbb{R}_{v}^{d}}(|v|^{2}\cdot\nabla_{x}f)dvdx
\B|&\leq C\|u_{t}\|_{L^{2}(\mathbb{R}^{3})}
\|\langle v\rangle^{k}|\nabla_{x}f| \|_{L^{2}(\mathbb{R}^{3}\times\mathbb{R}_{v}^{3})}. \ea\ee
From \eqref{3.67} and \eqref{3.68}, we know that
\be\ba\label{3.69}
&\B|\int_{\mathbb{R}^{3}}
\int_{\mathbb{R}_{v}^{3}}(u-v)(v\cdot\nabla_{x}f)dvu_{t}dx
\B|\\
\leq&C(\| u\|_{L^{\infty}(\mathbb{R}^{3})}+1)\|u_{t}\|_{L^{2}(\mathbb{R}^{3})}
\|\langle v\rangle^{k} |\nabla_{x}f| \|_{L^{2}(\mathbb{R}^{3}\times\mathbb{R}_{v}^{3})}.
\ea\ee
The last term on the right-hand side of \eqref{3.59}, after integration by parts, becomes
$$\ba
 \int_{\mathbb{R}^{3}}u_{t}\int_{\mathbb{R}_{v}^{3}}(u-v) \nabla_{v}[f(u-v) ] dvdx =
3\int_{\mathbb{R}^{3}}u_{t}u \int_{\mathbb{R}_{v}^{3}} f dvdx-3\int_{\mathbb{R}^{3}}u_{t}\int_{\mathbb{R}_{v}^{3}} f v dvdx.\ea$$
This together with the H\"older inequality ensures that
\be\ba\label{3.71}
&\B|\int_{\mathbb{R}^{3}}u_{t}\int_{\mathbb{R}_{v}^{3}}(u-v) \nabla_{v}[f(u-v) ] dvdx\B|
\\
&\leq
C\| u\|_{L^{\infty}(\mathbb{R}^{3})}\|u_{t}\|_{L^{2}(\mathbb{R}^{3})}
(\|n\|_{L^{2}(\mathbb{R}^{3})}+\|j\|_{L^{2}(\mathbb{R}^{3})}).
\ea\ee
Substituting \eqref{3.60}, \eqref{3.61}, \eqref{3.69} and \eqref{3.71} into \eqref{3.59}, we obtain
\be\ba\label{4.44}
&\f{d}{dt}\|u_{t}\|^{2}_{L^{2}(\mathbb{R}^{3})}+\|\nabla u_{t}\|^{2}_{L^{2}(\mathbb{R}^{3})} \\
&\leq
\|\nabla u\|_{L^{\infty}(\mathbb{R}^{3})}\|u_{t}\|^{2}_{L^{2}(\mathbb{R}^{3})}+
C  \|n\|^{2}_{L^{3}(\mathbb{R}^{3})}
\|u_{t}\|^{2}_{L^{2}(\mathbb{R}^{3})} \\
&\quad
+ C\| u\|_{L^{\infty}(\mathbb{R}^{3})}\|u_{t}\|_{L^{2}(\mathbb{R}^{3})}
\|\langle v\rangle^{k} |\nabla_{x}f| \|_{L^{2}(\mathbb{R}^{3}\times\mathbb{R}_{v}^{3})} \\
&\quad
+ \| u\|_{L^{\infty}(\mathbb{R}^{3})}\|u_{t}\|_{L^{2}(\mathbb{R}^{3})}
(\|n\|_{L^{2}(\mathbb{R}^{3})}+\|j\|_{L^{2}(\mathbb{R}^{3})}).
 \ea\ee
 Making full use of moments propagation \eqref{3.2}, we infer that
\be\ba\label{3.37}
& \|n\|_{L^{\infty}(0,T;L_{x}^{2})} \leq C \|f_{0}\|_{L^{\infty}_{ x,v}}^{\f{ 1}{2}}\B[ \|f_{0} |v|^{3 } \| _{L^{1}_{x,v} }+ C \|f_{0} \|_{L^{\infty}_{x,v}} \B(\int_{0}^{T}\|u\|_{L_{x}^{6 }}ds\B)^{6}\B]^{\f{1}{2}},
\\
& \|n\|_{L^{\infty}(0,T;L_{x}^{3})} \leq C \|f_{0}\|_{L^{\infty}_{ x,v}}^{\f{ 8}{9}} \B[ \|f_{0} |v|^{6 } \| _{L^{1}_{x,v} }+ C \|f_{0} \|_{L^{\infty}_{x,v}} \B(\int_{0}^{T}\|u\|_{L_{x}^{9 }}ds\B)^{9 }\B]^{\f{1}{9} },\\
& \|j\|_{L^{\infty}(0,T;L_{x}^{2})} \leq C \|f_{0}\|_{L^{\infty}_{ x,v}}^{\f{ 1}{2}}\B[ \|f_{0} |v|^{5 } \| _{L^{1}_{x,v} }+ C \|f_{0} \|_{L^{\infty}_{x,v}} \B(\int_{0}^{T}\|u\|_{L_{x}^{8 }}ds\B)^{8}\B]^{\f{1}{2}},
\ea\ee
which implies that $n, j \in L_{t}^{\infty}(0,T;L_{x}^{2})$.
Moreover, Lemma \ref{serrin} and Lemma \ref{lemma4.2} imply that
\[
u\in L^{2}(0,T;L^{\infty}(\mathbb{R}^{3})), \quad
\nabla u\in L^{1}(0,T;L^{\infty} (\mathbb{R}^{3})).
\]
As a consequence, Lemma \ref{lemma3.3} further ensures that $\|\langle v\rangle^{k} |\nabla_{x}f| \|_{L^{2}(\mathbb{R}^{3}\times\mathbb{R}_{v}^{3})} \in L_{t}^{\infty}(0,T)$.
Finally, we are in a position to apply Gr\"onwall's inequality \eqref{gronwall3} in Lemma \ref{gronwall} to \eqref{4.44}, which yields
\be\label{3.74}
u_{t}\in L^{\infty}(0,T;L^{2}(\mathbb{R}^{3})),
\quad \nabla u_{t}\in L^{2} (0,T;L^{2}(\mathbb{R}^{3})).
\ee
To apply Gr\"onwall's inequality, we have to prove $u_{t}(0,x)\in L^{2}(\mathbb{R}^{3})$. To this end, we reformulate the velocity equations as
\be\label{4.45v1}\partial_t u(0,x)= -u_{0} \cdot \nabla u_{0} -  \Delta u_{0} +\partial_{i}\mathcal{R}_{i}\mathcal{R}_{j}(u_{i}(0,x)u_{j}(0,x)) -\int_{\mathbb{R}_{v}^{3}} f_{0}dv u_{0}+\int_{\mathbb{R}_{v}^{3}} f_{0} vdv.
\ee
 It is enough to show that the right hand side of the above equation is in $L^{2}(\mathbb{R}^{3})$.
 Owing to H\"older's inequality and the Gagliardo-Nirenberg inequality,
\be\label{4.46v1}
 \|u_{0} \cdot \nabla u_{0}\|_{L^{2}(\mathbb{R}^{3})}\leq
\|u_{0}\|_{L^{6}(\mathbb{R}^{3})} \|u_{0} \cdot \nabla u_{0}\|_{L^{3}(\mathbb{R}^{3})}\leq C
\|u_{0}\|_{L^{6}(\mathbb{R}^{3})} \|u_{0}\|_{L^{6}(\mathbb{R}^{3})} \| u_{0}\|_{H^{2}(\mathbb{R}^{3})}^{\f12}\leq C\| u_{0}\|_{H^{2}(\mathbb{R}^{3})}^{2}.
 \ee
 It is clear that
 \be\label{4.47v1}
 \|\Delta u_{0}  \|_{L^{2}(\mathbb{R}^{3})}\leq C\| u_{0}\|_{H^{2}(\mathbb{R}^{3})}.
  \ee
 Thanks to the  Calder\'on-Zygmund Theorem, we conclude by the same manner as   that
  \be\label{4.48v1}
   \|\partial_{i}\mathcal{R}_{i}\mathcal{R}_{j}(u_{i}(0,x)u_{j}(0,x))  \|_{L^{2}(\mathbb{R}^{3})} =  \|\mathcal{R}_{i}\mathcal{R}_{j}(u_{i}(0,x)\partial_{i}u_{j}(0,x))  \|_{L^{2}(\mathbb{R}^{3})}\leq C\| u_{0}\|_{H^{2}(\mathbb{R}^{3})}^{2}.\ee
Making use of the Gagliardo-Nirenberg inequality once again, we get \be\label{4.49v1}
   \B\|
  \int_{\mathbb{R}_{v}^{3}} f_{0}dv u_{0} \B\|_{L^{2}(\mathbb{R}^{3})}\leq \|u_{0}\|_{L^{\infty}(\mathbb{R}^{3})}  \B\|
  \int_{\mathbb{R}_{v}^{3}} f_{0}dv \B\|_{L^{2}(\mathbb{R}^{3})}.
 \ee
We deduce from Gagliardo-Nirenberg inequality that
 \be\label{4.50v1}
 \|u_{0}\|_{L^{\infty}(\mathbb{R}^{3})} \leq C\|u_{0}\|_{L^{6}(\mathbb{R}^{3})} ^{\f12} \|u_{0}\|_{H^{2}(\mathbb{R}^{3})}^{\f12}\leq C\|u_{0}\|_{H^{2}(\mathbb{R}^{3})}.
 \ee
 The   kinetic interpolation inequality enables us to write
 \be\label{4.51v1}
  \B\|
  \int_{\mathbb{R}_{v}^{3}} f_{0}dv \B\|_{L^{2}(\mathbb{R}^{3})}\leq  C \|f_{0}\|_{L^{\infty}(\mathbb{R}^{3}\times\mathbb{R}^{3}_{v})}
 ^{\f12}\B\|
  \int_{\mathbb{R}_{v}^{3}} f_{0}|v|^{3}dv \B\|_{L^{2}(\mathbb{R}^{3})} ^{\f12}\ee
  and
  \be\label{4.52v1}
  \B\|
  \int_{\mathbb{R}_{v}^{3}} f_{0}vdv \B\|_{L^{2}(\mathbb{R}^{3})}\leq  C \|f_{0}\|_{L^{\infty}(\mathbb{R}^{3}\times\mathbb{R}^{3}_{v})}
 ^{\f12}\B\|
  \int_{\mathbb{R}_{v}^{3}} f_{0}|v|^{5}dv \B\|_{L^{2}(\mathbb{R}^{3})} ^{\f12}\ee
  Inserting \eqref{4.50v1}
and \eqref{4.51v1} into   \eqref{4.49v1}, we see that $
  \int_{\mathbb{R}_{v}^{3}} f_{0}dv u_{0} \in L^{2}(\mathbb{R}^{3}).$
  At this stage, we have proved that$u_{t}(0,x)\in L^{2}(\mathbb{R}^{3})$ by \eqref{4.45v1}.
This completes the proof of the lemma.
\end{proof}
To achieve the higher regularity of the  particle distribution function, we need the following  following estimate of the velocity (see Lemma \ref{lemma3.4}).
\begin{lemma}\label{lemma4.4}
Assume the conditions of Lemma \ref{lemma4.3} hold.
Suppose in addition that
\[
f_{0}v^{6 }, f_{0}v^{13 }, f_{0}v^{33 }\in L^{1}(\mathbb{R}^{3}\times\mathbb{R}^{3}_{v}).
\]
Then there hold
$\nabla^{2} u\in L^{\infty}(0,T; L^{2}(\mathbb{R}^{3})), u\in L^{\infty}(0,T; L^{\infty}(\mathbb{R}^{3})), \nabla^{2}u\in L^{2}(0,T; L^{4}(\mathbb{R}^{3}))$.
\end{lemma}

\begin{proof}
We rewrite the velocity equation as
\be\label{etticeq}
- \Delta u +\nabla \Pi =\partial_t u -u \cdot \nabla u -\int_{\mathbb{R}^{3}} f(u-v)dv
\ee
By applying the Stokes estimate, we obtain
\be\label{3.75}
\|\nabla^{2} u \|^{2}_{L^{2}(\mathbb{R}^{3})}+\|\nabla \Pi\|^{2}_{L^{2}(\mathbb{R}^{3})}\leq
 C\|\partial_t u\|^{2}_{L^{2}(\mathbb{R}^{3})}+ C\|u \cdot \nabla u\|^{2}_{L^{2}(\mathbb{R}^{3})}+C\|u n\|^{2}_{L^{2}(\mathbb{R}^{3})}+C\|j\|^{2}_{L^{2}(\mathbb{R}^{3})}.
\ee
Using H\"older's inequality and Young's inequality, we estimate
\be\ba\label{3.76}
\|u \cdot \nabla u\|^{2}_{L^{2}(\mathbb{R}^{3})}
&\leq
C\|u \|^{2}_{L^{6}(\mathbb{R}^{3})}\| \nabla u\|_{L^{2}(\mathbb{R}^{3})}\| \nabla u\|_{L^{6}(\mathbb{R}^{3})}\\
&\leq
C \|\nabla u\|_{L^{2}(\mathbb{R}^{3})}^{3} \| \nabla ^{2} u\|_{L^{2}(\mathbb{R}^{3})}\\
&\leq
C \|\nabla u\|_{L^{2}(\mathbb{R}^{3})}^{6} + \f12\| \nabla ^{2} u\|_{L^{2}(\mathbb{R}^{3})}^{2}.
\ea\ee
By H\"older's inequality and the Sobolev inequality, we find
\be\ba\label{3.77}
\|u n\|^{2}_{L^{2}(\mathbb{R}^{3})}
\leq \|u \|^{2}_{L^{6}(\mathbb{R}^{3})} \| n\|^{2}_{L^{3}(\mathbb{R}^{3})}
\leq C\|\nabla u \|^{2}_{L^{2}(\mathbb{R}^{3})} \| n\|^{2}_{L^{3}(\mathbb{R}^{3})}.
\ea\ee
Substituting \eqref{3.76} and \eqref{3.77} into \eqref{3.75}, we get
\be
\ba
&\f12\|\nabla^{2} u \|^{2}_{L^{2}(\mathbb{R}^{3})}+\|\nabla \Pi\|^{2}_{L^{2}(\mathbb{R}^{3})}\\
&\leq
 \|\partial_t u\|^{2}_{L^{2}(\mathbb{R}^{3})}+ C \|\nabla u\|_{L^{2}(\mathbb{R}^{3})}^{6} +C\|\nabla u \|^{2}_{L^{2}(\mathbb{R}^{3})} \| n\|^{2}_{L^{3}(\mathbb{R}^{3})}+C\|j\|^{2}_{L^{2}(\mathbb{R}^{3})}.
\ea \ee
From \eqref{3.2}, we have
\be\ba\label{4.45} \|n \|_{L^{\infty}(0,T;L_{x}^{3})} \leq C \|f_{0}\|_{L^{\infty}_{ x,v}}^{\f{ 2}{3}}\B[ \|f_{0} |v|^{6 } \| _{L^{1}_{x,v} }+ C \|f_{0} \|_{L^{\infty}_{x,v}} \B(\int_{0}^{T}\|u\|_{L_{x}^{9 }}ds\B)^{9}\B]^{\f{1}{3} },
\ea\ee
which means $n\in L_{t}^{\infty}(0,T;L_{x}^{3})$.
In addition, \eqref{3.37} implies $j\in L_{t}^{\infty}(0,T;L_{x}^{2})$.
There holds $ u_{t}\in L^{\infty}(0,T;L^{2}(\mathbb{R}^{3}))$ in \eqref{3.74}.
Therefore,
\be\label{3.80}
\nabla^{2} u, \nabla \Pi\in L^{\infty}(0,T; L^{2}(\mathbb{R}^{3})).
\ee
By interpolation and the Sobolev embedding,
\be\ba
\|u\|_{L^{\infty}(\mathbb{R}^{3})} \leq C \|u\|_{L^{6}(\mathbb{R}^{3})} \|\nabla u\|_{L^{6}(\mathbb{R}^{3})}\leq C \|\nabla u\|_{L^{2}(\mathbb{R}^{3})} \|\nabla^{2} u\|_{L^{2}(\mathbb{R}^{3})},
\ea\ee
and thus, from \eqref{3.80}, we obtain
\be\label{v4.62}
u\in L^{\infty}(0,T; L^{\infty}(\mathbb{R}^{3})).\ee

From \eqref{etticeq}, we note that
$$
\|\nabla^{2} u \|^{2}_{L^{4}(\mathbb{R}^{3})} \leq C
 \|\partial_t u\|^{2}_{L^{4}(\mathbb{R}^{3})} + C \|u \cdot \nabla u\|^{2}_{L^{4}(\mathbb{R}^{3})} +C \|u n\|^{2}_{L^{4}(\mathbb{R}^{3})} + C \|j\|^{2}_{L^{ 4}(\mathbb{R}^{3})}.
$$
Integrating in time yields
\be\ba\label{3.84}
\int_{0}^{T}\|\nabla^{2} u \|^{2}_{L^{4}(\mathbb{R}^{3})}dt
&\leq
C \int_{0}^{T}\|\partial_t u\|^{2}_{L^{4}(\mathbb{R}^{3})}dt
+ C \int_{0}^{T}\|u \cdot \nabla u\|^{2}_{L^{4}(\mathbb{R}^{3})}dt \\
&\quad
+C \int_{0}^{T}\|u n\|^{2}_{L^{4}(\mathbb{R}^{3})}dt
+ C\int_{0}^{T}\|j\|^{2}_{L^{ 4}(\mathbb{R}^{3})}dt.
\ea \ee
From the Gagliardo-Nirenberg inequality
\be\label{3.86}
\|u_{t}\|_{L^{4}(\mathbb{R}^{3})}\leq C \|u_{t}\|^{\f14}_{L^{2}(\mathbb{R}^{3})} \|\nabla u_{t}\|^{\f34}_{L^{2}(\mathbb{R}^{3})},
 \ee
which together with \eqref{3.74} implies
\be \ba \label{3.87}
\int_{0}^{T}\|\partial_t u\|^{2}_{L^{4}(\mathbb{R}^{3})}dt
&\leq
C \int_{0}^{T} \|u_{t}\|^{\f12}_{L^{2}(\mathbb{R}^{3})} \|\nabla u_{t}\|^{\f32}_{L^{2}(\mathbb{R}^{3})}dt\\
&\leq
C \|u_{t}\|^{\f12}_{L^{\infty}(0,T; L^{2}(\mathbb{R}^{3}))} \int_{0}^{T} \|\nabla u_{t}\|^{\f32}_{L^{2}(\mathbb{R}^{3})}dt.
\ea\ee
In the same manner as \eqref{3.86}, we know
$$\ba
\| \nabla u\|_{L^{4}(\mathbb{R}^{3})}\leq C\| \nabla u\|^{\f14}_{L^{2}(\mathbb{R}^{3})} \| \nabla ^{2} u\|^{\f34}_{L^{2}(\mathbb{R}^{3})},
 \ea$$
and therefore
\be\ba \label{3.88}
\int_{0}^{T}\|u \cdot \nabla u\|^{2}_{L^{4}(\mathbb{R}^{3})} dt
&\leq
\int_{0}^{T}\|u\|^{2}_{L^{\infty}(\mathbb{R}^{3})}\| \nabla u\|^{2}_{L^{4}(\mathbb{R}^{3})}dt \\
&\leq
C\|u\|^{2}_{L^{\infty}(0,T; L^{\infty}(\mathbb{R}^{3}))}\int_{0}^{T}\| \nabla u\|^{\f12}_{L^{2}(\mathbb{R}^{3})} \| \nabla u^{2}\|^{\f32}_{L^{2}(\mathbb{R}^{3})}dt \\
&\leq
C\|u\|^{2}_{L^{\infty}(0,T; L^{\infty}(\mathbb{R}^{3}))}\| \nabla u\|^{\f12}_{L^{\infty}(0,T; L^{2}(\mathbb{R}^{3}))}\int_{0}^{T} \| \nabla ^{2} u\|^{\f32}_{L^{2}(\mathbb{R}^{3})}dt.
 \ea\ee
H\"older's inequality and Sobolev embedding yield
\be\ba\label{3.89}
\int_{0}^{T}\|u n\|^{2}_{L^{4}(\mathbb{R}^{3})}dt
&\leq
C \int_{0}^{T}\|u \|^{2}_{L^{6}(\mathbb{R}^{3})}\| n\|^{2}_{L^{12}(\mathbb{R}^{3})}dt\\
&\leq
C \int_{0}^{T}\|\nabla u \|^{2}_{L^{2}(\mathbb{R}^{3})}\| n\|^{2}_{L^{12}(\mathbb{R}^{3})}dt\\
&\leq
C \|\nabla u \|^{2}_{L^{\infty}(0,T;L^{2}(\mathbb{R}^{3}))}
\int_{0}^{T}\| n\|^{2}_{L^{12}(\mathbb{R}^{3})}dt.
\ea\ee
Inserting \eqref{3.87}, \eqref{3.88} and \eqref{3.89} into \eqref{3.84}, we obtain
\be\ba\label{3.40}
\int_{0}^{T}\|\nabla^{2} u \|^{2}_{L^{4}(\mathbb{R}^{3})}dt
&\leq C\|u_{t}\|^{\f12}_{L^{\infty}(0,T; L^{2}(\mathbb{R}^{3}))} \int_{0}^{T} \|\nabla u_{t}\|^{\f32}_{L^{2}(\mathbb{R}^{3})}dt\\
&\quad
+ C\|u\|^{2}_{L^{\infty}(0,T; L^{\infty}(\mathbb{R}^{3}))}\| \nabla u\|^{\f12}_{L^{\infty}(0,T; L^{2}(\mathbb{R}^{3}))}\int_{0}^{T} \| \nabla ^{2} u\|^{\f32}_{L^{2}(\mathbb{R}^{3})}dt \\
&\quad
+ C \|\nabla u \|^{2}_{L^{\infty}(0,T;L^{2}(\mathbb{R}^{3}))} \int_{0}^{T} \| n\|^{2}_{L^{12}(\mathbb{R}^{3})}ds+C\int_{0}^{T}\|j\|^{2}_{L^{ 4}(\mathbb{R}^{3})}dt.
\ea\ee
It remains  to show that
\[
n\in L_{t}^{2}(0,T;L_{x}^{12}(\mathbb{R}^{3})), \quad
j\in L_{t}^{2}(0,T;L_{x}^{4}(\mathbb{R}^{3})).
\]
From \eqref{3.2}, we conclude by \eqref{v4.62} and $u\in L^{\infty}(0,T; L^{2}(\mathbb{R}^{3}))$ that
$$\ba
& \|n\|_{L^{\infty}(0,T;L_{x}^{12})} \leq C \|f_{0}\|_{L^{\infty}_{ x,v}}^{\f{ 11}{12}}\B[ \|f_{0} |v|^{33 } \| _{L^{1}_{x,v} }+ C \|f_{0} \|_{L^{\infty}_{x,v}} \B(\int_{0}^{T}\|u\|_{L_{x}^{36 }}ds\B)^{36}\B]^{\f{1}{12} },\\
& \|j \|_{L^{\infty}(0,T;L_{x}^{4})} \leq C \|f_{0}\|_{L^{\infty}_{ x,v}}^{\f{ 3}{4}}\B[ \|f_{0} |v|^{13 } \| _{L^{1}_{x,v} }+ C \|f_{0} \|_{L^{\infty}_{x,v}} \B(\int_{0}^{T}\|u\|_{L_{x}^{16 }}ds\B)^{16}\B]^{\f{1}{4} },
\ea$$
which yields $n\in L_{t}^{\infty}(0,T;L_{x}^{12}(\mathbb{R}^{3}))$ and $L_{t}^{\infty}(0,T;L_{x}^{4}(\mathbb{R}^{3}))$.
This completes the proof of this lemma.
\end{proof}

\begin{proof}[Proof of Theorem \ref{the1.1} concerning the velocity and its gradient]\ \\
By Lemma \ref{lemma4.4}, we can apply Lemma \ref{lemma3.4} to obtain $f\in L^{\infty}([0,T]; H^{2}(\mathbb{R}^{3}))$.
In summary, from Lemmas \ref{serrin}-\ref{lemma4.4}, we have proved that $f\in L^{\infty} [0,T]; H^{2}(\mathbb{R}^{3}))$
$u\in L^{\infty} ([0,T]; H^{2}(\mathbb{R}^{3})) $, $u_{t}\in L^{\infty} ([0,T];L^{2}(\mathbb{R}^{3}))\cap L^{2}([0,T];H^{1}(\mathbb{R}^{3}))$,
$\nabla \Pi\in L^{\infty} ([0,T];L^{2}(\mathbb{R}^{3})) $.
To conclude that $(u, f)$ is a strong solution in the sense of Definition \ref{ss}, it remains to prove $\nabla^{2} \partial_{i}u,\partial_{i}\Pi\in L^{2}(0,T;L^{2}(\mathbb{R}^{3}))$.
To this end, we differentiate equation \eqref{etticeq} with respect to $x_i$ to get
\eqref{etticeq} that
\be
- \Delta \partial_{i}u +\nabla \partial_{i}\Pi =\partial_t \partial_{i}u -\partial_{i}u \cdot \nabla u-\partial_{i}u \cdot \nabla \partial_{i}u -\int_{\mathbb{R}_{v}^{3}} \partial_{i}f(u-v)dv-\int_{\mathbb{R}_{v}^{3}} f\partial_{i}u dv.
\ee
Applying the Stokes estimate to this equation gives
\be\label{4.57}\ba
&\|\nabla^{2} \partial_{i}u \|_{L^{2}(\mathbb{R}^{3})}+\|\nabla \partial_{i}\Pi\|_{L^{2}(\mathbb{R}^{3})} \\
&\leq
 C\|\partial_t \partial_{i}u\|_{L^{2}(\mathbb{R}^{3})}
+ C\|\partial_{i}u \cdot \nabla u\|_{L^{2}(\mathbb{R}^{3})}
+ C\|u \cdot \nabla \partial_{i}u\|_{L^{2}(\mathbb{R}^{3})} \\
&\quad
+C\B\|\int_{\mathbb{R}^{3}_{v}} \partial_{i}f u dv\B\|_{L^{2}(\mathbb{R}^{3})}
+C\B\|\int_{\mathbb{R}^{3}_{v}} \partial_{i}f v dv\B\|_{L^{2}(\mathbb{R}^{3})}
+C\B\|\int_{\mathbb{R}^{3}_{v}} f\partial_{i}u dv\B\|_{L^{2}(\mathbb{R}^{3})}.
\ea \ee
Integrating in time, we find
\be\label{4.58}\ba
&\|\nabla^{2} \partial_{i}u \| _{L^{2}(0,T;L^{2}(\mathbb{R}^{3}))}+\|\nabla \partial_{i}\Pi\|_{L^{2}(0,T;L^{2}(\mathbb{R}^{3}))} \\
&\leq
 C\|\partial_t \partial_{i}u\|_{L^{2}(0,T;L^{2}(\mathbb{R}^{3}))}+ C\|\partial_{i}u \cdot \nabla u\|_{L^{2}(0,T;L^{2}(\mathbb{R}_{v}^{3}))}+ C\|u \cdot \nabla \partial_{i}u\|_{L^{2}(0,T;L^{2}(\mathbb{R}_{v}^{3}))} \\
&\quad
+C\B\|\int_{\mathbb{R}_{v}^{3}} \partial_{i}f u dv\B\|_{L^{2}(0,T;L^{2}(\mathbb{R}_{v}^{3}))}+C\B\|\int_{\mathbb{R}_{v}^{3}} \partial_{i}f v dv\B\|_{L^{2}(0,T;L^{2}(\mathbb{R}^{3}))} \\
&\quad
+C\B\|\int_{\mathbb{R}^{3}} f\partial_{i}u dv\B\|_{L^{2}(0,T;L^{2}(\mathbb{R}_{v}^{3}))}.
\ea \ee
From \eqref{4.23}, we have $\nabla u_{t}\in L^{2}(0,T; L^{2}(\mathbb{R}^{3}))$.
In addition, we derive from $\nabla^{2} u\in L^{\infty}(0,T; L^{2}(\mathbb{R}^{3}))$ in Lemma \ref{lemma4.4} that $\nabla u\in L^{\infty}(0,T; L^{6}(\mathbb{R}^{3}))$. This together with $\nabla u\in L^{\infty}(0,T; L^{2}(\mathbb{R}^{3}))$ in Lemma \ref{serrin} implies that $\nabla u\in L^{\infty}(0,T; L^{4}(\mathbb{R}^{3}))$, which controls the second term on the right-hand of \eqref{4.58}.
Since $u\in L^{\infty}(0,T; L^{\infty}(\mathbb{R}^{3}))$ and $\nabla^{2} u\in L^{\infty}(0,T; L^{2}(\mathbb{R}^{3}))$ in Lemma \ref{lemma4.4}, we estimate the second term on the right-hand of \eqref{4.58} as
$$\ba
\|u \cdot \nabla \partial_{i}u\|_{L^{2}(0,T;L^{2}(\mathbb{R}^{3}))}\leq
\|u \|_{L^{\infty}(0,T;L^{\infty}(\mathbb{R}^{3}))}\| \nabla^{2} u\|_{L^{2}(0,T;L^{2}(\mathbb{R}^{3}))}.
\ea $$
Clearly,
$$\ba
\B\|\int_{\mathbb{R}^{3}_{v}} \partial_{i}f u dv\B\|_{L^{2}(0,T;L^{2}(\mathbb{R}^{3}))}\leq
\|u \|_{L^{\infty}(0,T;L^{\infty}(\mathbb{R}^{3}_{v}))}\B\|\int_{\mathbb{R}_{v}^{3}} \partial_{i}f dv\B\|_{L^{2}(0,T;L^{2}(\mathbb{R}_{v}^{3}))}.
\ea $$
As in the derivation of \eqref{3.65}, we have
\be\ba\label{4.61}
 \B\|\int_{\mathbb{R}^{3}_{v}} \partial_{i}f dv\B\|_{L^{2}(0,T;L^{2}(\mathbb{R}^{3}))}\leq C\B\| \int_{\mathbb{R}_{v}^{3}} \langle v\rangle|\nabla_{x}f| dv\B\|_{L^{\infty}(0,T;L^{2}(\mathbb{R}^{3}))}\leq C\|\langle v\rangle^{k}|\nabla_{x}f| \|_{L^{\infty}(0,T;L^{2}(\mathbb{R}^{3}\times\mathbb{R}_{v}^{3}))}.
\ea\ee
which turns out that
$$\ba\B\|\int_{\mathbb{R}^{3}_{v}} \partial_{i}f u dv\B\|_{L^{2}(0,T;L^{2}(\mathbb{R}^{3}))}\leq
\|u \|_{L^{\infty}(0,T;L^{\infty}(\mathbb{R}^{3}))}\|\langle v\rangle^{k}|\nabla_{x}f| \|_{L^{\infty}(0,T;L^{2}(\mathbb{R}^{3}\times\mathbb{R}_{v}^{3}))}.
\ea $$
Along the line of \eqref{4.61}, we have
$$\ba\B\|\int_{\mathbb{R}^{3}_{v}} \partial_{i}f v dv\B\|_{L^{2}(0,T;L^{2}(\mathbb{R}^{3}_{v}))} \leq C\|\langle v\rangle^{k}|\nabla_{x}f| \|_{L^{\infty}(0,T;L^{2}(\mathbb{R}^{3}\times\mathbb{R}_{v}^{3}))}.
\ea $$
Finally, using $\nabla u\in L^{\infty}(0,T; L^{6}(\mathbb{R}^{3}))$
and \eqref{4.45}, we obtain
$$\ba
\B\|\int_{\mathbb{R}^{3}_{v}} f\partial_{i}u dv\B\|_{L^{2}(0,T;L^{2}(\mathbb{R}^{3}))}\leq C \|n\|_{L^{\infty}(0,T;L^{3}(\mathbb{R}^{3}))} \|\nabla u\|_{L^{2}(0,T;L^{6}(\mathbb{R}^{3}))}.
\ea$$
Combining all these estimates, we conclude that $\nabla^{2} \partial_{i}u,\partial_{i}\Pi\in L^{2}(0,T;L^{2}(\mathbb{R}^{3}))$, which completes the proof of the theorem.
\end{proof}

\subsection{Ladyzhenskaya--Prodi--Serrin regularity
criteria based on the pressure}

In this subsection, we will show that the regularity of the pressure implies that the velocity lies in the Ladyzhenskaya--Prodi--Serrin class, as established in the previous subsection.

\begin{proof}[Proof of Theorem \ref{the1.1} involving the pressure and its gradient]\ \\
Multiplying the Navier--Stokes equations \eqref{NSV} by $u|u|$, integrating by parts, and using the divergence-free condition, we obtain
\be\ba\label{6.23}
&\f{1}{3}\f{d}{dt}\int_{\mathbb{R}^{3}}|u|^{3}dx+\int_{\mathbb{R}^{3}}|\nabla u|^{2}|u|^{2}dx
+\f{1}{9}\int_{\mathbb{R}^{3}}|\nabla| u|^{\f32}|^{2}dx\\
&=
-\int_{\mathbb{R}^{3}} u\cdot\nabla \Pi |u| dx-\int_{\mathbb{R}^{3}}nuu|u| dx+\int_{\mathbb{R}^{3}}ju|u| dx\\
&=: I+II+II.
\ea\ee
In what follows, based on \eqref{serrin1}, it suffices to bound $I$ under the hypothesis of Theorem \ref{the1.1}.

(1)
Integration by parts, using the incompressible condition, and H\"older's inequality, we obtain
\be\ba\label{6.24}
|I|
&=
\B|-\int_{\mathbb{R}^{3}} \Pi u \cdot \nabla|u| dx \B|\\
&\leq
C\int_{\mathbb{R}^{3}} |\Pi| |u|^{\f12} \nabla|u|^{\f32} dx
\\
&\leq
C \| \Pi\|_{L^{q}}\||u|^{\f12}\|_{L^{\f{2q}{q-2}}(\mathbb{R}^{3})}
\|\nabla|u|^{\f32}\|_{L^{2}(\mathbb{R}^{3})}
\\
&\leq
C \| \Pi\|_{L^{q}(\mathbb{R}^{3})}\||u|\|^{\f12}_{L^{\f{ q}{q-2}}(\mathbb{R}^{3})}\|\nabla|u|^{\f32}\|_{L^{2}(\mathbb{R}^{3})}.
\ea\ee
For $\f{9}{4} \leq q \leq 3$, we interpolate in Lebesgue spaces to get
$$\ba
\| u \|_{L^{\f{ q}{q-2}}(\mathbb{R}^{3})}\leq\| u \|_{L^{3}(\mathbb{R}^{3})}^{\f{4q-9}{q}}\| u \|_{L^{9}(\mathbb{R}^{3})}^{\f{ 9-3q}{q} }.
\ea$$
Using Sobolev embedding, we have
$$\ba
\| u \|_{L^{9}(\mathbb{R}^{3})}=\||u|^{\f32} \|_{L^{6}(\mathbb{R}^{3})}^{\f23} \leq C\|\nabla|u|^{\f32} \|_{L^{2}(\mathbb{R}^{3})}^{\f23},
\ea$$
and hence
\be\ba\label{6.26}
\| u \|_{L^{\f{ q}{q-2}}(\mathbb{R}^{3})}\leq\| u \|_{L^{3}(\mathbb{R}^{3})}^{\f{4q-9}{q}}\| u \|_{L^{9}(\mathbb{R}^{3})}^{\f{ 9-3q}{q}}\leq C\| u \|_{L^{3} (\mathbb{R}^{3})}^{\f{4q-9}{q}}\|\nabla|u|^{\f32} \|_{L^{2}(\mathbb{R}^{3})}^{\f{ 2(3- q)}{q}}.
\ea\ee
Substituting this into \eqref{6.24} and using Young's inequality, we obtain
\be\ba\label{6.27}
|I|
&\leq
C \| \Pi\|_{L^{q}(\mathbb{R}^{3})} \| u \|_{L^{3}(\mathbb{R}^{3})}^{\f{4q-9}{2q}}\|\nabla|u|^{\f32} \|_{L^{2}(\mathbb{R}^{3})}^{\f{ (3- q)}{q}}\|\nabla|u|^{\f32}\|_{L^{2}(\mathbb{R}^{3})} \\
&\leq
C \| \Pi\|_{L^{q}(\mathbb{R}^{3})}^{\f{2q}{2q-3}} \| u \|^{\f{4q-9}{2q-3}} _{L^{3}(\mathbb{R}^{3})} +\f{1}{54}\|\nabla|u|^{\f32}\|_{L^{2}(\mathbb{R}^{3})}^{2}.
\ea\ee
By H\"older's inequality
\be\ba
\B|-\int_{\mathbb{R}^{3}}nuu|u| dx \B|
&\leq \|n\|_{L^{2}(\mathbb{R}^{3})}\|u\|^{3}_{L^{6}(\mathbb{R}^{3})}.
\ea\ee
By the Gagliardo-Nirenberg inequality, we obtain
$$\ba
\||u|^{\f32}\| _{L^{4}(\mathbb{R}^{3})}\leq C\||u|^{\f32}\|^{\f14} _{L^{2}(\mathbb{R}^{3})} \|\nabla|u|^{\f32}\|^{\f34} _{L^{2}(\mathbb{R}^{3})},
\ea$$
which gives
\be\ba\label{6.28}
\|u\| _{L^{6}(\mathbb{R}^{3})}=\||u|^{\f32}\|^{\f23} _{L^{4}(\mathbb{R}^{3})}\leq C\||u|^{\f32}\|^{\f16} _{L^{2}(\mathbb{R}^{3})} \|\nabla|u|^{\f32}\|^{\f12} _{L^{2}(\mathbb{R}^{3})}.
\ea\ee
Hence, we see that
\be\ba\label{6.29}
II
&\leq
\|n\|_{L^{2}(\mathbb{R}^{3})}\||u|^{\f32}\|^{\f12} _{L^{2}(\mathbb{R}^{3})} \|\nabla|u|^{\f32}\|^{\f32} _{L^{2}}\\
&\leq
\|n\|_{L^{2}(\mathbb{R}^{3})}^{4}\||u| \|^{3} _{L^{3}(\mathbb{R}^{3})} +\f{1}{54}\|\nabla|u|^{\f32}\|^{2} _{L^{2}(\mathbb{R}^{3})}
\ea\ee
by Young's inequality.

It remains to bound $III$.
Using the H\"older inequality and \eqref{6.28}, we have
$$\ba
\B|\int_{\mathbb{R}^{3}}ju|u| dx\B|\leq \|j\|_{L^{\f32}(\mathbb{R}^{3})} \|u\|^{2}_{L^{6}(\mathbb{R}^{3})}\leq C\|j\|_{L^{\f32}(\mathbb{R}^{3})} \||u|^{\f32}\|^{\f13} _{L^{2}(\mathbb{R}^{3})} \|\nabla|u|^{\f32}\| _{L^{2}(\mathbb{R}^{3})}.
\ea$$
By Young's inequality, we get
\be\ba\label{6.30}
|III| \leq C\|j\|^{2}_{L^{\f32}(\mathbb{R}^{3})} \||u| \| _{L^{3}(\mathbb{R}^{3})} +\f{1}{54}\|\nabla|u|^{\f32}\|^{2} _{L^{2}(\mathbb{R}^{3})}.
\ea\ee
Substituting \eqref{6.27}, \eqref{6.29}, and \eqref{6.30} into \eqref{6.23}, we obtain
\be\ba\label{6.31}
&\f{1}{3}\f{d}{dt}\int_{\mathbb{R}^{3}}|u|^{3}dx+\int_{\mathbb{R}^{3}}|\nabla u|^{2}|u|^{2}dx +\f{1}{18}\int_{\mathbb{R}^{3}}|\nabla| u|^{\f32}|^{2}dx\\
&\leq
C \| \Pi\|_{L^{q}(\mathbb{R}^{3})}^{\f{2q}{2q-3}}\| u \|^{\f{4q-9}{2q-3}} _{L^{3}(\mathbb{R}^{3})}+\|n\|_{L^{2}(\mathbb{R}^{3})}^{4}\| u  \|^{3} _{L^{3}(\mathbb{R}^{3})} +C\|j\|^{2}_{L^{\f32}(\mathbb{R}^{3})} \|u \|_{L^{3}(\mathbb{R}^{3})}.
\ea\ee
It is enough to show that $n \in L^{4}(0,T;L^{2}(\mathbb{R}^{3}))$ and $ j\in L^{2}(0,T;L^{\f32}(\mathbb{R}^{3}))$. Indeed, employing
\eqref{3.2}, we have
\be\ba\label{6.32}
\B\|\int_{\mathbb{R}^{3}}
f dv\B\|_{L^{\infty}(0,T;L_{x}^{2})} \leq C \|f_{0}\|_{L^{\infty}_{ x,v}}^{\f{ 1}{2}}\B[ \|f_{0} |v|^{3 } \| _{L^{1}_{x,v} }+ C \|f_{0} \|_{L^{\infty}_{x,v}} \B(\int_{0}^{T}\|u\|_{L_{x}^{6 }}ds\B)^{6}\B]^{\f{1}{2}},
\ea\ee
and
\be\ba\label{6.33}
\B\|\int_{\mathbb{R}^{d}}
f v dv\B\|_{L_{t}^{\infty}(0,T;L_{x}^{\f{3}{2}})} \leq C \|f\|_{L^{\infty}_{t,x,v}}^{\f{1}{3}}\B[ \int_{\mathbb{T}^{n}}f_{0}v^{3 } dvdx+ C \|f_{0} \|_{L^{\infty}_{x}L^{\infty}_{v}} \B(\int_{0}^{T}\|u\|_{L_{x}^{6 }}ds\B)^{6}\B]^{\f{2}{3} }.
\ea\ee
It can easily be check $\f{4q-9}{2q-3}\leq3$.
By Gr\"onwall inequality \eqref{gronwall3} in Lemma \ref{gronwall} to \eqref{6.31}, we get $|u|^{\f3 2}\in L^{2}(0,T; L^{6}(\mathbb{R}^{3}))$, and hence $u \in L^{3}(0,T; L^{9}(\mathbb{R}^{3}))$.
This completes the proof of the assertion based on the first part of the theorem.

(2) By H\"older's inequality
\be\ba \label{6.34}
\B|\int_{\mathbb{R}^{3}} u\cdot\nabla \Pi |u| dx\B|\leq C\|\Pi\|_{L^{q}(\mathbb{R}^{3})}
\|u\|^{2}_{L^{\f{2q}{q-1}}(\mathbb{R}^{3})}.
\ea\ee
By the Gagliardo-Nirenberg inequality, we obtain for $\f97 \leq q \leq 3$,
$$\ba
\||u|^{\f32}\| _{L^{\f{4q}{3(q-1)}}(\mathbb{R}^{3})}\leq C\||u|^{\f32}\|^{\f{7q-9}{4q}} _{L^{2}(\mathbb{R}^{3})}\|\nabla|u|^{\f32}\| _{L^{2}(\mathbb{R}^{3})}^{\f{ 9-3q}{4q}},
\ea$$
which yields
$$\ba
\|u\|^{2}_{L^{\f{2q}{q-1}}(\mathbb{R}^{3})}=\||u|^{\f32}\|^{\f{4}{3}} _{L^{\f{4q}{3(q-1)}}(\mathbb{R}^{3})}\leq C\||u|^{\f32}\|^{\f{7q-9}{3q}} _{L^{2}(\mathbb{R}^{3})}\|\nabla|u|^{\f32}\| _{L^{2}(\mathbb{R}^{3})}^{\f{ 3- q}{ q}}.
\ea$$
Inserting this into \eqref{6.34} and using Young's inequality, we get
\be\ba\label{6.35}
\B|\int_{\mathbb{R}^{3}} u\cdot\nabla \Pi |u| dx\B|
&\leq
C\|\nabla\Pi\|_{L^{q}(\mathbb{R}^{3})}
\||u|^{\f32}\|^{\f{7q-9}{3q}} _{L^{2}(\mathbb{R}^{3})}\|\nabla|u|^{\f32}\| _{L^{2}(\mathbb{R}^{3})}^{\f{ 3- q}{ q}}\\
&\leq
C\|\nabla\Pi\|_{L^{q}(\mathbb{R}^{3})}^{\f{2q}{3 q-3}}
\| u \|^{\f{7q-9}{3q-3}} _{L^{3}(\mathbb{R}^{3})}+\f{1}{54}\|\nabla|u|^{\f32}\| _{L^{2}(\mathbb{R}^{3})}^{2}.
\ea\ee
Substituting \eqref{6.35}, \eqref{6.29} and \eqref{6.30} into \eqref{6.23}, we obtain
\be\ba
&\f{1}{3}\f{d}{dt}\int_{\mathbb{R}^{3}}|u|^{3}dx+\int_{\mathbb{R}^{3}}|\nabla u|^{2}|u|^{2}dx
+\f{1}{18}\int_{\mathbb{R}^{3}}|\nabla| u|^{\f32}|^{2}dx\\
&\leq
C \|\nabla\Pi\|_{L^{q}(\mathbb{R}^{3})}^{\f{2q}{3 q-3}}
\| u \|^{\f{7q-9}{3q-3}} _{L^{3}(\mathbb{R}^{3})}+\|n\|_{L^{2}(\mathbb{R}^{3})}^{4}\| u  \|^{3} _{L^{3}(\mathbb{R}^{3})} +C\|j\|^{2}_{L^{\f32}(\mathbb{R}^{3})} \|u \|_{L^{3}(\mathbb{R}^{3})}.
\ea\ee
It is not difficult to verify that that $\f{7q-9}{3q-3}\leq3$. According to \eqref{6.32} and \eqref{6.33}, we may apply Gr\"onwall's inequality  \eqref{gronwall3} in Lemma \ref{gronwall} to \eqref{6.33}    to deduce that $|u|^{\f3 2}\in L^{2}(0,T; L^{6}(\mathbb{R}^{3}))$, and hence $u \in L^{3}(0,T; L^{9}(\mathbb{R}^{3}))$.
This completes the proof.
\end{proof}

\appendix
\section{Higher moment estimates}\label{appendix}

Before we present  higher moment estimates $D^{s}   f\langle v\rangle^{k}$ for $s\geq 3$, we give a comparison between
$\iint_{\mathbb{R}^{d}\times\mathbb{R}_{v}^{d}}f^{\ell}|v|^{k\ell}dvdx$ and $\iint_{\mathbb{R}^{d}\times\mathbb{R}_{v}^{d}} f^{\ell}\langle v\rangle^{k\ell}dvdx$.

\begin{lemma}\label{higermomentsestimate2}
Let $(f, u)$ be a weak solution to the Vlasov equation $\eqref{NSV}_{2}$ on $\mathbb{R}^{d}\times \mathbb{R}^{d}_{v}$.
If the initial datum $f_{0}$ satisfies
$$
\left\|\langle v\rangle^{k} f_{0}\right\|_{L^{\ell}\left(\mathbb{R}^{d} \times \mathbb{R}^{d}_{v}\right)}<\infty, ~ \ell \geq1, ~ k\geq0,
$$
and $u\in L^{1}(0,T;L^{k\ell+d +\eta }(\mathbb{R}^{d}))$, where $\eta $ is a sufficiently small positive constant, then
\be\ba
\left\|\langle v\rangle^{k} f\right\|_{L^{\infty}(0,T;L^{\ell}_{x,v} )}
 &\leq e^{\left(d-\frac{d}{\ell}\right)T}\B[ \left\|\langle v\rangle^{k} f_{0}\right\|_{ L^{\ell}_{x,v} } +\|f\|_{L ^{\infty}_{x,v}} \B(\int_{0}^{T}
\|u\|_{L_{x}^{d+k\ell+\eta} }ds\B)^{\f{d+k\ell+\eta}{\ell} } \B].
\ea\ee
Moreover, if $u\in L^{1}(0,T;L^{kp+dp +(p+1)\eta}(\mathbb{R}^{d}))$, then
\be\ba \label{3.14} &\B\| \int_{\mathbb{R}_{v}^{d}} f |v|^{k} dv\B\|_{L^{\infty}(0,T;L_{x}^{p}) }\\
\leq & C e^{\left(d-\frac{n}{\ell}\right)T}\B[ \left\|\langle v\rangle^{k+\f{d(p-1)}{p}+\eta} f_{0}\right\|_{ L^{p}_{x,v} } +\|f_{0}\|_{L_{v,x}^{\infty} } \B(\int_{0}^{T}
\|u\|_{L^{dp+kp+(p+1)\eta}_{x} }ds\B)^{\f{dp+kp+(p+1)\eta}{p} } \B].
\ea\ee
\end{lemma}

\begin{remark}
In contrast to \eqref{3.2}, the estimate \eqref{3.14} requires stronger regularity assumptions on the velocity field $u$.
However, it seems that the approach used for \eqref{3.2} appears to break down, whereas the technique applied to \eqref{3.14} remains effective for estimating higher moments.
\end{remark}

\begin{proof}
By multiplying the Vlasov equation $\eqref{NSV}_{2}$ by $\langle v\rangle^{k \ell} f^{\ell-1}$ and integrating over $\mathbb{R}^{d} \times \mathbb{R}^{d}$, we obtain
\be\ba\label{3.280}
\frac{1}{\ell} \frac{d}{d t}\left\|\langle v\rangle^{k} f\right\|_{L^{\ell}\left(\mathbb{R}^{d} \times \mathbb{R}^{d}_{v}\right)}^{\ell}
&=
-\frac{1}{\ell} \iint_{\mathbb{R}^{d}\times\mathbb{R}_{v}^{d}} \left(u \cdot \nabla_{v}\right) f^{\ell}\langle v\rangle^{k \ell} d v d x \\
&\quad
+ \iint_{\mathbb{R}^{d}\times\mathbb{R}_{v}^{d}} \nabla_{v} \cdot(v f)\langle v\rangle^{k\ell} f^{\ell-1} d x d v \\
&=: I + II
\ea\ee
Here, we have used $ \iint_{\mathbb{R}^{d}\times\mathbb{R}_{v}^{d}} v\cdot\nabla_{x}f f^{\ell-1} \langle v\rangle^{\ell k}dvdx=0$.

Thanks to  integration by parts, we conclude by  H\"older's inequality that
\be\ba\label{3.290}
|I |
&\leq C \iint_{\mathbb{R}^{d}\times\mathbb{R}_{v}^{d}} \langle v\rangle^{k \ell-1}|u| f^{\ell} d v d x\\
&\leq C \|u\|_{L^{d+k\ell+\eta}\left(\mathbb{R}^{d}\right)}
\B\|\int_{\mathbb{R}^{d}_{v}}\langle v\rangle^{k\ell-1}|u| f^{\ell} d v \B\|_{L^{\f{d+k\ell+\eta}{d-1+k\ell+\eta}}\left(\mathbb{R}^{d}\right)}.
\ea\ee
Utilizing the H\"older inequality again, we get
$$\ba
\B|\int_{\mathbb{R}^{d}_{v}}\langle v\rangle^{k\ell-1} f^{\ell} d v\B|
&=
\B|\int_{\mathbb{R}^{d}_{v}} \langle v\rangle^{\f{k\ell(d-1+k\ell+\eta)}{d+k\ell+\eta}} f^{\ell}
\langle v\rangle^{\f{k\ell }{d+k\ell+\eta}-1} dv\B| \\
&\leq
\left(\int_{\mathbb{R}^{d}_{v}}\langle v\rangle^{k \ell} f^{ \frac{\ell(d+1+k\ell+\eta)}{d-1+k\ell+\eta}} d v\right)^{1-\frac{1}{d+k\ell+\eta}}\left(\int_{\mathbb{R}_{v}^{d}} \frac{1}{\left.\langle v\rangle^{d+\eta} d v\right.}\right)^{\frac{1}{d+k\ell+\eta}}\\
&\leq
C \left\|\langle v\rangle^{k} f\right\|^{\ell-\frac{\ell}{d+k\ell+\eta}}_{L^{\ell}\left( \mathbb{R}^{d}_{v}\right)}\|f\|_{L^{\infty}\left(\mathbb{R}^{d}_{v} \right)}^{\frac{\ell}{d+k\ell+\eta}}.
\ea$$
Substituting this into \eqref{3.290}, and using \eqref{lwuqiong}, we deduce
\be\label{3.300}
\begin{aligned}
I
&\leq
C \|u\|_{L^{d+k\ell+\eta}\left(\mathbb{R}^{d}\right)} \left\|\langle v\rangle^{k} f\right\|^{\ell-\frac{\ell}{d+k\ell+\eta}}_{L^{\ell}\left( \mathbb{R}^{d}\times\mathbb{R}_{v}^{d}\right)}\|f\|_{L^{\infty}\left(\mathbb{R}^{d}\times\mathbb{R}_{v}^{d} \right)}^{\frac{\ell}{d+k\ell+\eta}} \\
&\leq
C \|u\|_{L^{d+k\ell+\eta}\left(\mathbb{R}^{d}\right)} \left\|\langle v\rangle^{k} f\right\|^{\ell-\frac{\ell}{d+k\ell+\eta}}_{L^{\ell}\left( \mathbb{R}^{d}\times\mathbb{R}_{v}^{d}\right)}\|f_{0}\|_{L^{\infty} \left( \mathbb{R}^{d}\times\mathbb{R}_{v}^{d}\right)}^{\frac{\ell}{d+k\ell+\eta}}.
\end{aligned}
\ee
Now we estimate the term $II$. Integrating by parts, we write
$$
\begin{aligned}
II
&=
n \iint_{\mathbb{R}^{d}\times\mathbb{R}_{v}^{d}} \langle v\rangle^{k \ell} f^{\ell} d v d x+\frac{1}{\ell}\iint_{\mathbb{R}^{d}\times\mathbb{R}_{v}^{d}} v \cdot \nabla_{v} f^{\ell}\langle v\rangle^{k \ell} d v d x \\
&=
\left(d-\frac{d}{\ell}\right) \iint_{\mathbb{R}^{d}\times\mathbb{R}_{v}^{d}} \langle v\rangle^{k \ell} f^{\ell} d v d x-k \iint_{\mathbb{R}^{d}\times\mathbb{R}_{v}^{d}} |v|^{2}\langle v\rangle^{k \ell-2} f^{\ell} d v d x,
\end{aligned}
$$
from which it follows that
\be\ba\label{3.301}
&\frac{1}{\ell} \frac{d}{d t}\left\|\langle v\rangle^{k} f\right\|_{L^{\ell}\left(\mathbb{R}^{d} \times \mathbb{R}^{d}_{v}\right)}^{\ell}+ k \iint_{\mathbb{R}^{d}\times\mathbb{R}_{v}^{d}} |v|^{2}\langle v\rangle^{k \ell-2} f^{\ell} d v d x \\
&=
I + \left(d-\frac{d}{\ell}\right)\iint_{\mathbb{R}^{d}\times\mathbb{R}_{v}^{d}} \langle v\rangle^{k \ell} f^{\ell} d v d x.
\ea\ee
A combination of \eqref{3.301} and \eqref{3.290}, we see that
$$
\begin{aligned}
& \frac{1}{\ell} \frac{d}{d t}\left\|\langle v\rangle^{k} f\right\|_{L^{\ell} (\mathbb{R}^{d} \times \mathbb{R}^{d}_{v} )}^{\ell}
+k \iint_{\mathbb{R}^{d}\times\mathbb{R}_{v}^{d}} |v|^{2}\langle v\rangle^{k \ell-2} f^{\ell} d v d x\\
&\leq  (d-\frac{d}{\ell} )\left\|\langle v\rangle^{k} f\right\|_{L^{\ell}\left(\mathbb{R}^{d} \times \mathbb{R}^{d}_{v}\right)}^{\ell}+C \|u\|_{L^{d+k\ell+\eta} (\mathbb{R}^{d} )} \left\|\langle v\rangle^{k} f\right\|^{\ell-\frac{\ell}{d+k\ell+\eta}}_{L^{\ell} ( \mathbb{R}^{d} \times \mathbb{R}^{d}_{v} )}\|f\|_{L^{\infty} (\mathbb{R}^{d} \times \mathbb{R}^{d}_{v} )}^{\frac{\ell}{d+k\ell+\eta}}.
\end{aligned}
$$
Applying Gr\"onwall's inequality \eqref{gronwall3} in Lemma \ref{gronwall}, we conclude that
\be\ba
\left\|\langle v\rangle^{k} f\right\|_{L^{\infty}(0,T;L_{x,v}^{\ell} )}^{\ell}
\leq e^{\left(d-\frac{d}{\ell}\right)T}\B[ \left\|\langle v\rangle^{k} f_{0}\right\|_{ L_{x,v}^{\ell} }^{\ell}+\|f_{0}\|_{L_{v,x}^{\infty} }^{ \ell }\B(\int_{0}^{T}\|u\|_{L_{x}^{d+k\ell+\eta} }ds \B)^{d+k\ell+\eta}\B],
\ea\ee
that is,
\be\ba
\left\|\langle v\rangle^{k} f\right\|_{L^{\infty}(0,T;L_{x,v}^{\ell} )}
&\leq e^{\left(d-\frac{d}{\ell}\right)T}\B( \left\|\langle v\rangle^{k} f_{0}\right\|_{ L_{x,v}^{\ell} } +\|f_{0}\|_{L_{v,x}^{\infty} } \B( \int_{0}^{T}
\|u\|_{L_{x}^{d+k\ell+\eta} }ds \B)^{\f{d+k\ell+\eta}{\ell} }\B].
\ea\ee
Finally, applying inequality \eqref{2.10} from Lemma \ref{lem2.2}, we further observe that
\be\ba
\B\| \int_{\mathbb{R}_{v}^{d}} f |v|^{k}dv \B\|_{L_{x}^{p} }
&\leq
C\|\langle v\rangle^{k+\f{d(p-1)}{p}+\eta}f \|_{L^{p}_{x,v} } \\
&\leq
C e^{\left(d-\frac{d}{\ell}\right)T}\B[ \left\|\langle v\rangle^{k+\f{d(p-1)}{p}+\eta} f_{0}\right\|_{ L^{p}_{x,v} } \\
&\quad
+ \|f_{0}\|_{L_{v,x}^{\infty} } \B(\int_{0}^{T}
\|u\|_{L_{x}^{dp+kp+(p+1)\eta} }ds\B)^{\f{dp+kp+(p+1)\eta}{p} } \B].
\ea\ee
This completes the proof of the lemma.
\end{proof}
For higher moment estimates, we have
\begin{lemma}\label{a.2}
Let the pair $(u, f)$ be a solution
of the Vlasov equation \eqref{Vlasov} on $\mathbb{R}^{d}$ for $d=2, 3$.
Then, there holds
\begin{align}
 \f{d}{dt}\|D^{3}   f\langle v\rangle^{k}\|_{L^{2}_{x,v}}^{2} \leq&  C
  \|D^{3}_{x} u \|_{L^{2}_{x}}\| \nabla_{x}^{d-2}\partial_{ v_{i}} f\langle v\rangle^{k}\|_{L_{x}^{2}(\mathbb{R}^{d})}^{\f12} \| D^{3} f\langle v\rangle^{ k}\|_{L^{2}_{x,v}}^{\f32}\nonumber\\  &+C(\|\nabla u\|_{L^{\infty}_x}+\|u\|_{L^{\infty}_x}+
  \|D_{x}^{2}u\|_{L^{3}_{x}} +1) \|D^{3} f\langle v\rangle^{ k}\|_{L^{2}_{x,v}}^{2},\label{pri4}\\
  \f{d}{dt}\|D^{s}  f\langle v\rangle^{k}\|_{L^{2}_{x,v}}^{2} \leq
  &C
  \|D^{s}_{x} u \|_{L^{2}_{x}}\|  \nabla_{x}^{d-2}\partial_{ v_{i}} f\langle v\rangle^{k}\|_{L_{x}^{2}(\mathbb{R}^{2})}^{\f12}\|\nabla^{2} \partial_{ v_{i}} f\langle v\rangle^{k}\|_{L_{x}^{2}(\mathbb{R}^{2})}^{\f12}\| D^{s} f\langle v\rangle^{ k}\|_{L^{2}_{x,v}}\nonumber\\  &+C(\|\nabla u\|_{L^{\infty}_x} +\|u\|_{L^{\infty}_x}+\|D_{x}^{2}u\|_{L^{3}_{x}} +1)\|D^{s }  f\langle v\rangle^{ k}\|_{L^{2}_{x,v}}^{2}\nonumber\\
   &+C\|D_{x}^{\alpha}u\|_{L^{2}_x}^{1-\f{n}{4}} \|D_{x}^{\alpha+1}u\|_{L^{2}_x}^{\f{n}{4}}
   \|D^{\beta}_{v}\partial_{v_{i}}f\langle v\rangle^{ k}\|_{L^{2}_{x,v}}^{1-\f{n}{4}}
\|D_{x}D^{\beta}_{v}\partial_{v_{i}}f\langle v\rangle^{ k}\|_{L^{2}_{x,v}}^{\f{n}{4}}\| D_{x}^{\alpha}D^{\beta}_{v}f
  \langle v\rangle^{ k}\|_{L^{2}_{x,v}},\label{pri5}
 \end{align}
 where $s\geq4$,
$3\leq\alpha\leq s-1,1\leq\beta\leq s-3$.
\end{lemma}
\begin{proof}
$\bullet$ Estimation of  $\f{d}{dt}\|D^{s}_{x}  f\langle v\rangle^{k}\|_{L^{2}_{x,v}}^{2}$. \\
Applying the   operator $D^{s}_{x}$ to the particle equation $\eqref{Vlasov}_{2}$, we discover  that
 \be
  D^{s}_{x}\partial_t f+v\cdot\nabla_{x}D^{s}_{x}f+\partial_{v_{i}}\cdot D^{s}_{x} (f u_i )  -\partial_{v_{i}}[D^{s}_{x}f   v_{i}  ] =0.
 \ee
 Taking product of  this equation with  $D^{s}_{x}f\langle v\rangle^{2k}$ and integrating, we get
  \be\ba\label{A.13}
 &\f{d}{dt}\|D^{s}_{x}  f\langle v\rangle^{k}\|_{L^{2}_{x,v}}^{2}\\
 =&- \iint_{\mathbb{R}^{d}\times\mathbb{R}_{v}^{d}}  v\cdot\nabla_{x}D^{s}_{x} D^{s}_{x}f\langle v\rangle^{2k}fdxdv- \iint_{\mathbb{R}^{d}\times\mathbb{R}_{v}^{d}}\partial_{v_{i}}D^{s}_{x} (f u_i )  D^{s}_{x}f\langle v\rangle^{2k}dxdv \\&+ \iint_{\mathbb{R}^{d}\times\mathbb{R}_{v}^{d}}\partial_{v_{i}}[D^{s}_{x}f   v_{i}  ] D^{s}_{x}f\langle v\rangle^{2k} dxdv\\&=I'+II'+III'.
 \ea\ee
 Due to integration by parts, we notice that
  \be\label{A.14}
I'=  \iint_{\mathbb{R}^{n}\times\mathbb{R}^{n}} v\cdot\nabla_{x}D^{s}_{x}f D^{s}_{x}f\langle v\rangle^{2k}dvdx=0.
  \ee
We reformulate $II'$ as
  \be\ba\label{A.15}
  II'  =&-\iint_{\mathbb{R}^{d}\times\mathbb{R}_{v}^{d}} [D^{s}_{x} (u_i\partial_{ v_{i}}f ) -u_iD^{s}_{x}  \partial_{v_{i}}f ] D^{s}_{x}f\langle v\rangle^{2k}dvdx+\f{1}{2}\iint_{\mathbb{R}^{d}\times\mathbb{R}_{v}^{d}} u_i (  D^{s}_{x}f)^{2}\langle v\rangle^{2k-2}v_{i}dvdx,
   \ea \ee
   where the integration by parts was used.\\
Making use of the H\"older inequality, we remark that
    \be\ba\label{A.16}
&\B|\iint_{\mathbb{R}^{d}\times\mathbb{R}_{v}^{d}} [D^{s}_{x} (u_i\partial_{ v_{i}}f ) -u_iD^{s}_{x}  \partial_{v_{i}}f ] D^{s}_{x}f\langle v\rangle^{2k}dvdx\B|\\
\leq&  \B(\iint_{\mathbb{R}^{d}\times\mathbb{R}_{v}^{d}} [D^{s}_{x} (u_i\partial_{ v_{i}}f ) -u_iD^{s}_{x}  \partial_{v_{i}}f ]^{2}\langle v\rangle^{2k}dvdx\B)^{\f12}\| D^{s}_{x}f\langle v\rangle^{ k}\|_{L^{2}_{x,v}}.
       \ea \ee
It follows from the Kato-Ponce commutator estimate
that
$$\ba
   \|  D^{s}_{x} (u_i\partial_{ v_{i}}f ) -u_iD^{s}_{x}  \partial_{v_{i}}f  \|_{L_{x}^{2}}\leq
   C(\|\partial_{ v_{i}}f\|_{L^{\infty}_{x}}\|D^{s}_{x} u \|_{L^{2}_{x}}+\|\nabla u\|_{L^{\infty}_x}\|D^{s-1}_{x}\partial_{ v_{i}}f\|_{L^{2}_{x}}).
   \ea $$
Hence, we rewrite \eqref{A.16} as
 % \be\ba   \B(\iint_{\mathbb{R}^{d}\times\mathbb{R}_{v}^{d}} [D^{s}_{x} (u_i\partial_{ v_{i}}f ) -u_iD^{s}_{x}  \partial_{v_{i}}f ]^{2}\langle v\rangle^{2k}dvdx\B)^{\f12}=
 % \B( \int\|  D^{s}_{x} (u_i\partial_{ v_{i}}f ) -u_iD^{s}_{x}  \partial_{v_{i}}f  \|^{2}_{L_{x}^{2}}\langle v\rangle^{2k}dv \B)^{\f12}  \ea \ee
   \be\ba \label{A.17}  &\B(\iint_{\mathbb{R}^{d}\times\mathbb{R}_{v}^{d}} [D^{s}_{x} (u_i\partial_{ v_{i}}f ) -u_iD^{s}_{x}  \partial_{v_{i}}f ]^{2}\langle v\rangle^{2k}dvdx\B)^{\f12}\\
       \leq& C
  \|D^{s}_{x} u \|_{L^{2}_{x}}\B(\int_{\mathbb{R}_{v}^{d}} \|\partial_{ v_{i}}f\|^{2}_{L^{\infty}_{x}}\langle v\rangle^{2k}dv \B)^{\f12}+C\|\nabla u\|_{L^{\infty}_x} \|D^{s-1}_{x}\partial_{ v_{i}}f\langle v\rangle^{ k}\|_{L^{2}_{x,v}}.
   \ea \ee
It is not difficult to verify that
\be\ba\label{A.18}
-\f{1}{2}\iint_{\mathbb{R}^{d}\times\mathbb{R}_{v}^{d}} u_i (  D^{s}_{x}f)^{2}\langle v\rangle^{2k-2}v_{i}dvdx\leq C\|u\|_{L^{\infty}_x} \| D^{s}_{x}f\langle v\rangle^{ k}\|_{L^{2}_{x,v}}^{2}.
\ea \ee
Inserting \eqref{A.17} and \eqref{A.18} into \eqref{A.15}, we know that
\be\ba\label{A.19}
|II'|\leq C
  \|D^{s}_{x} u \|_{L^{2}_{x}}\B(\int_{\mathbb{R}_{v}^{n}} \|\partial_{ v_{i}}f\|^{2}_{L^{\infty}_{x}}\langle v\rangle^{2k}dv \B)^{\f12}+C(\|\nabla u\|_{L^{\infty}_x} +\|u\|_{L^{\infty}_x}) \| D^{s} f\langle v\rangle^{ k}\|_{L^{2}_{x,v}}^{2}.
\ea \ee
By a simple manipulation, we conclude by the integration by parts that
$$\ba III' =&d\iint_{\mathbb{R}^{d}\times\mathbb{R}_{v}^{d}}  [D^{s}_{x}f   ]D^{s}_{x}f\langle v\rangle^{2k } dvdx-\f12\iint_{\mathbb{R}^{d}\times\mathbb{R}_{v}^{d}}   [D^{s}_{x}f   ] ^{2}\partial_{v_{i}}(\langle v\rangle^{2k-2 } v_{i}) dvdx,
 \ea $$
which means
      \be\ba\label{A.20} |III'|\leq C\| D^{s}_{x}f\langle v\rangle^{ k}\|_{L^{2}_{x,v}}^{2}.
   \ea \ee
Plugging   \eqref{A.14}, \eqref{A.19} and \eqref{A.20} into \eqref{A.13}, we find
\be\ba\label{key1}
\f{d}{dt}\|D^{s}_{x}  f\langle v\rangle^{k}\|_{L^{2}_{x,v}}^{2}\leq &C
  \|D^{s}_{x} u \|_{L^{2}_{x}}\B(\int_{\mathbb{R}^{n}}\|\partial_{ v_{i}}f\|^{2}_{L_{x}^{\infty}}\langle v\rangle^{2k}dv \B)^{\f12}\| D^{s}_{x}f\langle v\rangle^{ k}\|_{L^{2}_{x,v}}\\&+ C(\|\nabla u\|_{L^{\infty}_x} +\|u\|_{L^{\infty}_x}+1) \| D^{s} f\langle v\rangle^{ k}\|_{L^{2}_{x,v}}^{2}.
\ea \ee
$\bullet$ Estimation of  $\f{d}{dt}\|D^{s}_{v}  f\langle v\rangle^{k}\|_{L^{2}_{x,v}}^{2}$.

Applying the   operator $D^{s}_{v}$ to the Vlsov equation $\eqref{Vlasov}$, one writes
    $$\ba
     \partial_t D^{s}_{v}f+D_{v}^{s} (v\cdot\nabla_{x}f)+\nabla_{v}\cdot[D^{s}_{v}f u   ]- D^{s}_{v}\nabla_{v}\cdot [f v  ] =0.
   \ea $$
Dotting the above equation by $D^{s}_{v}f\langle v\rangle^{2k}$, one gets
    \be\ba
    \f{d}{dt}\|D^{s}_{v}  f\langle v\rangle^{k}\|_{L^{2}_{x,v}}^{2} =&-\iint_{\mathbb{R}^{d}\times\mathbb{R}_{v}^{d}} D_{v}^{s} (v\cdot\nabla_{x}f)D^{s}_{v}f\langle v\rangle^{2k} dxdv-\iint_{\mathbb{R}^{d}\times\mathbb{R}_{v}^{d}}\nabla_{v}[D^{s}_{v}f u   ]D^{s}_{v}f\langle v\rangle^{2k} dxdv\\&+\iint_{\mathbb{R}^{d}\times\mathbb{R}_{v}^{d}}\nabla_{v}D^{s}_{v}[f v  ]D^{s}_{v}f\langle v\rangle^{2k} dxdv.
   \ea \ee
   A routine computation gives rise to
$$\ba
  \iint_{\mathbb{R}^{d}\times\mathbb{R}_{v}^{d}} D^{s}_{v}(v\cdot\nabla_{x}f)  D^{s}_{v}f\langle v\rangle^{2k} dxdv
     =  \iint_{\mathbb{R}^{d}\times\mathbb{R}_{v}^{d}}  D_{v} v\cdot\nabla_{x}D^{s-1}_{v}fD^{s}_{v}f\langle v\rangle^{2k} dxdv.
   \ea $$
Hence, we derive from the H\"older inequality
\be\ba\label{A.23}
\B|\iint_{\mathbb{R}^{d}\times\mathbb{R}_{v}^{d}} D^{s}_{v}(v\cdot\nabla_{x}f)  D^{s}_{v}f\langle v\rangle^{2k} dxdv\B|\leq \|  \nabla_{x}D^{s-1}_{v}f\langle v\rangle^{ k}\|_{L^{2}_{x,v}}\| D^{s}_{v}f\langle v\rangle^{ k}\|_{L^{2}_{x,v}}.
 \ea \ee
Integration by parts gives
$$\ba\iint_{\mathbb{R}^{d}\times\mathbb{R}_{v}^{d}} \nabla_{v}\cdot[D^{s}_{v}f u   ]D^{s}_{v}f\langle v\rangle^{2k} dxdv =&-\iint_{\mathbb{R}^{d}\times\mathbb{R}_{v}^{d}}[D^{s}_{v}f  ]^{2} u \cdot\nabla_{v}(\langle v\rangle^{2k}) dxdv,
\ea $$
which it follows from that
\be\ba\label{A.24}\iint_{\mathbb{R}^{d}\times\mathbb{R}_{v}^{d}} \nabla_{v}[D^{s}_{v}f u   ]D^{s}_{v}f\langle v\rangle^{2k} dxdv   \leq C\|u\|_{L^{\infty}_x}\| D^{s}_{v}f\langle v\rangle^{ k}\|_{L^{2}_{x,v}}^{2}.
\ea \ee
Using integrating by parts once again, we arrive at
$$\ba
&\iint_{\mathbb{R}^{d}\times\mathbb{R}_{v}^{d}}
\nabla_{v}\cdot D^{s}_{v}[f v  ]D^{s}_{v}f \langle v\rangle^{2k} dxdv\\
=&d\iint_{\mathbb{R}^{d}\times\mathbb{R}_{v}^{d}}   D^{s}_{v}[f   ]D^{s}_{v}f \langle v\rangle^{2k} dxdv-\f{1}{2}\iint_{\mathbb{R}^{d}\times\mathbb{R}_{v}^{d}} D^{s}_{v}fD^{s}_{v}f \partial_{v_j}[\langle v\rangle^{2k}v_j ]dxdv\\&+d\iint_{\mathbb{R}^{d}\times\mathbb{R}_{v}^{d}}[D_{v}v_j D^{s-1}_{v}\partial_{v_j}f ]D^{s}_{v}f \langle v\rangle^{2k} dxdv.
\ea $$
As a consequence, we discover that
\be\ba\label{A.25}
\B|\iint_{\mathbb{R}^{d}\times\mathbb{R}_{v}^{d}}
 \nabla_{v}\cdot D^{s}_{v}[f v  ]D^{s}_{v}f \langle v\rangle^{2k} dxdv\B|
\leq  C\|D^{s}_{v}f \langle v\rangle^{ k}\|_{L^{2}_{x,v}}^{2}.
\ea \ee
Collecting the previous estimates \eqref{A.23}-\eqref{A.25}, we arrive at
\be\ba\label{key2}
\f{d}{dt}\|D^{s}_{v}  f\langle v\rangle^{k}\|_{L^{2}_{x,v}}^{2}\leq &C\|  \nabla_{x}D^{s-1}_{v}f\langle v\rangle^{ k}\|_{L^{2}_{x,v}}\| D^{s}_{v}f\langle v\rangle^{ k}\|_{L^{2}_{x,v}}+C\|u\|_{L^{\infty}_x}\| D^{s}_{v}f\langle v\rangle^{ k}\|_{L^{2}_{x,v}}^{2}+C\|D^{s}_{v}f \langle v\rangle^{ k}\|_{L^{2}_{x,v}}^{2}.
\ea \ee
%After an integration in time and the use of the Gronwall inequality we obtain
%The Gronwall inequality gives the claimed result.
$\bullet$ Estimation of  $\f{d}{dt}\|D^{s-1}_{v}D_{x}f\langle v\rangle^{ k}\|_{L^{2}_{x,v}}^{2}$.\\
From the kinetic equation $\eqref{NSV}$, we observe that
\be\ba  \partial_tD_{x} D^{s-1}_{v}f+D_{v}^{s-1} (v\cdot\nabla_{x}D_{x}f)+\nabla_{v}\cdot D_{x}[D^{s-1}_{v}f u   ]-\nabla_{v}D^{s-1}_{v}[D_{x}f v  ] =0.
   \ea \ee
Taking the $L^{2}$ inner product of the above equation with $ D^{s-1}_{v}f\langle v\rangle^{ 2k}$, we notice that
 \be\ba \label{a3.39}
    \f{d}{dt}\|D^{s-1}_{v}D_{x}f\langle v\rangle^{ k}\|_{L^{2}_{x,v}}^{2}
  =&-\iint_{\mathbb{R}^{d}\times\mathbb{R}_{v}^{d}} D_{v}^{s-1} (v\cdot\nabla_{x}D_{x}f)D_{x} D^{s-1}_{v}f\langle v\rangle^{ 2k}dxdv\\&-\iint_{\mathbb{R}^{d}\times\mathbb{R}_{v}^{d}}\nabla_{v}\cdot D_{x}[D^{s-1}_{v}f u   ]D_{x} D^{s-1}_{v}f\langle v\rangle^{ 2k}dxdv\\&+\iint_{\mathbb{R}^{d}\times\mathbb{R}_{v}^{d}}\nabla_{v}D^{s-1}_{v}[D_{x}f v  ]D_{x} D^{s-1}_{v}f\langle v\rangle^{ 2k}dxdv.
    \ea \ee
It is straightforward to show that
   $$\ba  \iint_{\mathbb{R}^{d}\times\mathbb{R}_{v}^{d}} D_{v}^{s-1} (v\cdot\nabla_{x}D_{x}f)D_{x} D^{s-1}_{v}f\langle v\rangle^{ 2k}dxdv
      = \iint_{\mathbb{R}^{d}\times\mathbb{R}_{v}^{d}} (D_{v}v\cdot\nabla_{x}D_{x} D_{v}^{s-2}f)D_{x} D^{s-1}_{v}f\langle v\rangle^{ 2k}dxdv,
   \ea$$
   which leads to
   \be\ba\label{A.29}
   \B|\iint_{\mathbb{R}^{d}\times\mathbb{R}_{v}^{d}} D_{v}^{s-1} (v\cdot\nabla_{x}D_{x}f)D_{x} D^{s-1}_{v}f\langle v\rangle^{ 2k}dxdv
  \B| \leq \|\nabla_{x}D_{x} D_{v}^{s-2}f\langle v\rangle^{  k}\|_{L^{2}_{x,v}}\|D_{x} D^{s-1}_{v}f\langle v\rangle^{  k}\|_{L^{2}_{x,v}}.
\ea \ee
In the light of integrating by parts, one obtains
$$\ba
 & \iint_{\mathbb{R}^{d}\times\mathbb{R}_{v}^{d}}\nabla_{v}\cdot D_{x}[D^{s-1}_{v}f u   ]
D^{s-1}_{v}D_{x}f\langle v\rangle^{ 2k}dxdv \\ =&
-\f{1}{2}\iint_{\mathbb{R}^{d}\times\mathbb{R}_{v}^{d}}[D^{s-1}_{v}D_{x}f ]^{2}
 u\cdot\nabla_{v}[\langle v\rangle^{ 2k}]dxdv+\iint_{\mathbb{R}^{d}\times\mathbb{R}_{v}^{d}} D_{x}u\cdot\nabla_{v} [D^{s-1}_{v}f ]
D^{s-1}_{v}D_{x}f\langle v\rangle^{ 2k}dxdv.
\ea $$
Consequently, we infer that
\be\ba\label{A.30}
 & \iint_{\mathbb{R}^{d}\times\mathbb{R}_{v}^{d}}\nabla_{v}\cdot D_{x}[D^{s-1}_{v}f u   ]
D^{s-1}_{v}D_{x}f\langle v\rangle^{ 2k}dxdv\\
\leq &C\|u\|_{L^{\infty}_x}\|D_{x} D^{s-1}_{v}f\langle v\rangle^{  k}\|_{L^{2}_{x,v}}^{2}+C\|\nabla u\|_{L^{\infty}_x}\|  D^{s }_{v}f\langle v\rangle^{  k}\|_{L^{2}_{x,v}}\|D_{x} D^{s-1}_{v}f\langle v\rangle^{  k}\|_{L^{2}_{x,v}}.
\ea \ee
It is clear that
\be\ba\label{A.31}
&\iint_{\mathbb{R}^{d}\times\mathbb{R}_{v}^{d}}
\partial_{v_{j}}D^{s-1}_{v}[D_{x}f v_{j}  ]D^{s-1}_{v}D_{x}f\langle v\rangle^{ 2k}dxdv\\=&d\iint_{\mathbb{R}^{d}\times\mathbb{R}_{v}^{d}}
 D^{s-1}_{v}[D_{x}f    ]D^{s-1}_{v}D_{x}f\langle v\rangle^{ 2k}dxdv+\iint_{\mathbb{R}^{d}\times\mathbb{R}_{v}^{d}}
D^{s-1}_{v}[v_{j}D_{x}\partial_{v_{j}}f ]D^{s-1}_{v}D_{x}f\langle v\rangle^{ 2k}dxdv.
\ea \ee
We deduce from  integrating by parts that
$$\ba
&\iint_{\mathbb{R}^{d}\times\mathbb{R}_{v}^{d}}
D^{s-1}_{v}[v_{j}D_{x}\partial_{v_{j}}f ]D^{s-1}_{v}D_{x}f\langle v\rangle^{ 2k}dxdv \\=&-\f{1}{2}\iint_{\mathbb{R}^{d}\times\mathbb{R}_{v}^{d}}
 [ D^{s-1}_{v}D_{x}f]^{2}\partial_{v_{j}}[v_{j}\langle v\rangle^{ 2k}]dxdv+d\iint_{\mathbb{R}^{d}\times\mathbb{R}_{v}^{d}}
[D_{v}v_{j}D_{x}D^{s-2}_{v}\partial_{v_{j}}f ]D^{s-1}_{v}D_{x}f\langle v\rangle^{ 2k}dxdv,
\ea $$
which implies
$$\ba&
\B|\iint_{\mathbb{R}^{d}\times\mathbb{R}_{v}^{d}}
\partial_{v_{j}}D^{s-1}_{v}[D_{x}f v_{j}  ]D^{s-1}_{v}D_{x}f\langle v\rangle^{ 2k}dxdv\B|\leq C\|D^{s-1}_{v}D_{x}f\langle v\rangle^{ k}\|_{L^{2}_{x,v}}^{2}.
 \ea $$
Substituting this, \eqref{A.30}  and \eqref{A.31} into \eqref{a3.39}, we arrive at
 \be\ba\label{key3}
 &\f{d}{dt}\|D^{s-1}_{v}D_{x}f\langle v\rangle^{ k}\|_{L^{2}_{x,v}}^{2}
 \\ \leq&C\|\nabla_{x}D_{x} D_{v}^{s-2}f\langle v\rangle^{  k}\|_{L^{2}_{x,v}}\|D_{x} D^{s-1}_{v}f\langle v\rangle^{  k}\|_{L^{2}_{x,v}}+C\|u\|_{L^{\infty}_x}\|D_{x} D^{s-1}_{v}f\langle v\rangle^{  k}\|_{L^{2}_{x,v}}^{2}\\&+C\|\nabla u\|_{L^{\infty}_x}\|  D^{s }_{v}f\langle v\rangle^{  k}\|_{L^{2}_{x,v}}\|D_{x} D^{s-1}_{v}f\langle v\rangle^{  k}\|_{L^{2}_{x,v}} +C\|D^{s-1}_{v}D_{x}f\langle v\rangle^{ k}\|_{L^{2}_{x,v}}^{2}.
 \ea \ee
$\bullet$  Estimation of $\f{d}{dt}\|D_{x}^{\alpha}D^{\beta}_{v}f\langle v\rangle^{ k}\|  _{L^{2}}^{2}$.

The Vlasov  equation $\eqref{NSV}$
can be recast as
  \be\ba \label{A.39} \partial_t D_{x}^{\alpha}D^{\beta}_{v}f+D_{v}^\beta (v\cdot\nabla_{x}D_{x}^{\alpha}f)+\nabla_{v}D_{x}^{\alpha}[D^{\beta}_{v}f u   ]-\nabla_{v}D^{\beta}_{v}[D_{x}^{\alpha}f v  ] =0.
   \ea \ee
Multiplying both sides of the above equation by $D_{x}^{\alpha}D^{\beta}_{v}f
  \langle v\rangle^{2k}$ and integrating on $\mathbb{R}^{d}\times\mathbb{R}_{v}^{d}$, we infer that
\be\ba\label{a3.49}
&\f{d}{dt}\|D_{x}^{\alpha}D^{\beta}_{v}f\langle v\rangle^{ k}\|  _{L^{2}_{x,v}}^{2}\\=&- \iint_{\mathbb{R}^{d}\times\mathbb{R}_{v}^{d}} D_{v}^\beta (v\cdot\nabla_{x}D_{x}^{\alpha}f)D_{x}^{\alpha}D^{\beta}_{v}f
  \langle v\rangle^{2k} dxdv-\iint_{\mathbb{R}^{d}\times\mathbb{R}_{v}^{d}}
\nabla_{v}D_{x}^{\alpha}[D^{\beta}_{v}f u   ]D_{x}^{\alpha}D^{\beta}_{v}f
  \langle v\rangle^{2k} dxdv\\&+\iint_{\mathbb{R}^{d}\times\mathbb{R}_{v}^{d}}\nabla_{v}D^{\beta}_{v}[D_{x}^{\alpha}f v  ]D_{x}^{\alpha}D^{\beta}_{v}f
  \langle v\rangle^{2k} dxdv.
\ea \ee
It is not difficult to verify that
 \be\ba&\iint_{\mathbb{R}^{d}\times\mathbb{R}_{v}^{d}}
D_{v}^\beta (v\cdot\nabla_{x}D_{x}^{\alpha}f) D_{x}^{\alpha}D^{\beta}_{v}f
  \langle v\rangle^{2k} dxdv \\
   =& \iint_{\mathbb{R}^{d}\times\mathbb{R}_{v}^{d}}
(D_{v}v\cdot\nabla_{x}D_{v}^{\beta-1} D_{x}^{\alpha }f) D_{x}^{\alpha}D^{\beta}_{v}f
   \langle v\rangle^{2k}dxdv.
   \ea \ee
This and the Cauchy-Schartz inequality yield   that
\be\ba\label{a3.50}
\B|\iint_{\mathbb{R}^{d}\times\mathbb{R}_{v}^{d}}
D_{v}^\beta (v\cdot\nabla_{x}D_{x}^{\alpha}f) D_{x}^{\alpha}D^{\beta}_{v}f
  \langle v\rangle^{2k} dxdv\B|\leq C\| D_{v}^{\beta-1} D_{x}^{\alpha+1}f
   \langle v\rangle^{ k}\|_{L^{2}_{x,v}} \|D_{x}^{\alpha}D^{\beta}_{v}f
   \langle v\rangle^{ k}\|_{L^{2}_{x,v}}.
 \ea \ee
 Owing to the integration by parts, we conclude by the H\"older inequality that
 \be\ba\label{a3.51}&\B|\iint_{\mathbb{R}^{d}\times\mathbb{R}_{v}^{d}}
 \partial_{v_{i}}D_{x}^{\alpha}[D^{\beta}_{v}f u_{i}   ] D_{x}^{\alpha}D^{\beta}_{v}f
  \langle v\rangle^{2k} dxdv\B|\\=&
  \B|\iint_{\mathbb{R}^{d}\times\mathbb{R}_{v}^{d}}
 [\partial_{v_{i}}D_{x}^{\alpha}(D^{\beta}_{v}f u_{i} )-u_{i}\partial_{v_{i}}D_{x}^{\alpha} D^{\beta}_{v}f ]D_{x}^{\alpha}D^{\beta}_{v}f
  \langle v\rangle^{2k} dxdv+\iint_{\mathbb{R}^{d}\times\mathbb{R}_{v}^{d}} u_{i}\partial_{v_{i}}D_{x}^{\alpha} D^{\beta}_{v}f  D_{x}^{\alpha}D^{\beta}_{v}f
  \langle v\rangle^{2k} dxdv\B|\\=&\B|
  \iint_{\mathbb{R}^{d}\times\mathbb{R}_{v}^{d}}
 [\partial_{v_{i}}D_{x}^{\alpha}(D^{\beta}_{v}f u_{i} )-u_{i}\partial_{v_{i}}D_{x}^{\alpha} D^{\beta}_{v}f ]D_{x}^{\alpha}D^{\beta}_{v}f
  \langle v\rangle^{2k} dxdv-\f12\iint_{\mathbb{R}^{d}\times\mathbb{R}_{v}^{d}} u_{i}(D_{x}^{\alpha} D^{\beta}_{v}f  )^{2}
  \partial_{v_{i}}\langle v\rangle^{2k} dxdv
  \\\leq&\B( \iint_{\mathbb{R}^{d}\times\mathbb{R}_{v}^{d}}
[\partial_{v_{i}}D_{x}^{\alpha}(D^{\beta}_{v}f u_{i} )-u_{i}\partial_{v_{i}}D_{x}^{\alpha} D^{\beta}_{v}f ]^{2}
  \langle v\rangle^{2k} dxdv\B)^{\f12}\| D_{x}^{\alpha}D^{\beta}_{v}f
  \langle v\rangle^{ k}\|_{L^{2}_{x,v}}+ \|u\|_{L^{\infty}_x}\|D_{x}^{\alpha}D^{\beta}_{v}f
  \langle v\rangle^{ k}\|^{2}_{L^{2}_{x,v}}.
\ea \ee
The Kato-Ponce commutator estimate  allows us to get
  \be\ba\label{a3.52}
&\B( \iint_{\mathbb{R}^{d}\times\mathbb{R}_{v}^{d}}
[D_{x}^{\alpha}(D^{\beta}_{v}\partial_{v_{i}}f u_{i} )-u_{i}\partial_{v_{i}}D_{x}^{\alpha} D^{\beta}_{v}f ]^{2}
  \langle v\rangle^{2k} dxdv\B)^{\f12}\\\leq&
  \B( \int_{\mathbb{R}_{v}^{d}} (\|D_{x}^{\alpha}u\|_{L^{4}_{x}}^{2}\|D^{\beta}_{v}\partial_{v_{i}}f\|^{2}_{L^{4}_{x}}
  +\|\nabla u\|^{2}_{L^{\infty}_{x}}\|D_{x}^{\alpha-1}D^{\beta}_{v}\partial_{v_{i}}f\|^{2}_{L^{2}_{x}})
\langle v\rangle^{2k}dv\B)^{\f12}\\\leq&
  \|D_{x}^{\alpha}u\|_{L^{4}} \B( \int_{\mathbb{R}_{v}^{d}} \|D^{\beta}_{v}\partial_{v_{i}}f\|^{2}_{L^{4}_{x}}
 \langle v\rangle^{2k}dv\B)^{\f12} +\|\nabla u\|_{L^{\infty}_x}\B( \int_{\mathbb{R}_{v}^{d}}\|D_{x}^{\alpha-1}
 D^{\beta}_{v}\partial_{v_{i}}f\|^{2}_{L^{2}_{x}}
\langle v\rangle^{2k}dv\B)^{\f12}.
\ea \ee
We deduce from the
Gagliardo-Nirenberg inequality that
  \be\ba\label{a3.53} &\|D_{x}^{\alpha}u\|_{L^{4}_{x}}
\leq C\|D_{x}^{\alpha}u\|_{L^{2}_x}^{1-\f{n}{4}} \|D_{x}^{\alpha+1}u\|_{L^{2}_x}^{\f{n}{4}},   \\
&\|D^{\beta}_{v}\partial_{v_{i}}f\|_{L^{4}_{x}}\leq C\|D^{\beta}_{v}\partial_{v_{i}}f\|_{L^{2}_{x}}^{1-\f{n}{4}}
\|D_{x}D^{\beta}_{v}\partial_{v_{i}}f\|_{L^{2}_{x}}^{\f{n}{4}}.
  \ea \ee
  In view of the H\"older inequality, we remark that
\be\ba\B( \int_{\mathbb{R}^{n}}   \|D^{\beta}_{v}\partial_{v_{i}}f\|^{2}_{L^{4}}
 \langle v\rangle^{2k}dv\B)^{\f12}\leq
 C\|D^{\beta}_{v}\partial_{v_{i}}f\langle v\rangle^{ k}\|_{L^{2}_{x,v}}^{1-\f{n}{4}}
\|D_{x}D^{\beta}_{v}\partial_{v_{i}}f\langle v\rangle^{ k}\|_{L^{2}_{x,v}}^{\f{n}{4}}.
 \ea \ee
Substituting this and \eqref{a3.53} into \eqref{a3.52}, we infer that
   \be\ba
&\B( \iint_{\mathbb{R}^{d}\times\mathbb{R}_{v}^{d}}
[D_{x}^{\alpha}(D^{\beta}_{v}\partial_{v_{i}}f u_{i} )-u_{i}\partial_{v_{i}}D_{x}^{\alpha} D^{\beta}_{v}f ]^{2}
  \langle v\rangle^{2k} dxdv\B)^{\f12}\\\leq&
  C\|D_{x}^{\alpha}u\|_{L^{2}_x}^{1-\f{n}{4}} \|D_{x}^{\alpha+1}u\|_{L^{2}_x}^{\f{n}{4}}\|D^{\beta}_{v}\partial_{v_{i}}f\langle v\rangle^{ k}\|_{L^{2}_{x,v}}^{1-\f{n}{4}}
\|D_{x}D^{\beta}_{v}\partial_{v_{i}}f\langle v\rangle^{ k}\|_{L^{2}_{x,v}}^{\f{n}{4}}\\&+C\|\nabla u\|_{L^{\infty}_x}
\|D_{x}^{\alpha-1}D^{\beta}_{v}\partial_{v_{i}}f\langle v\rangle^{ k}\|_{L^{2}_{x,v}}.
  \ea \ee
Combining this and \eqref{a3.51}, we discover that
  \be\ba\label{a3.57}&\B|\iint_{\mathbb{R}^{d}\times\mathbb{R}_{v}^{d}}
 \partial_{v_{i}}D_{x}^{\alpha}[D^{\beta}_{v}f u_{i}   ] D_{x}^{\alpha}D^{\beta}_{v}f
  \langle v\rangle^{2k} dxdv\B|\\\leq&
  C\|D_{x}^{\alpha}u\|_{L^{2}_x}^{1-\f{n}{4}} \|D_{x}^{\alpha+1}u\|_{L^{2}_x}^{\f{n}{4}}\|D^{\beta}_{v}\partial_{v_{i}}f\langle v\rangle^{ k}\|_{L^{2}_{x,v}}^{1-\f{n}{4}}
\|D_{x}D^{\beta}_{v}\partial_{v_{i}}f\langle v\rangle^{ k}\|_{L^{2}_{x,v}}^{\f{n}{4}}\| D_{x}^{\alpha}D^{\beta}_{v}f
  \langle v\rangle^{ k}\|_{L^{2}_{x,v}}\\&+C\|\nabla u\|_{L^{\infty}_x}
\|D_{x}^{\alpha-1}D^{\beta}_{v}\partial_{v_{i}}f\langle v\rangle^{ k}\|_{L^{2}_{x,v}} \| D_{x}^{\alpha}D^{\beta}_{v}f
  \langle v\rangle^{ k}\|_{L^{2}_{x,v}}.
    \ea \ee
A simple calculation  means
  \be\ba\label{A.43}
  &\iint_{\mathbb{R}^{d}\times\mathbb{R}_{v}^{d}} \partial_{v_{j}}D^{\beta}_{v}[D_{x}^{\alpha}f v _{j} ] D_{x}^{\alpha}D^{\beta}_{v}f
  \langle v\rangle^{2k} dxdv\\ =&\iint_{\mathbb{R}^{d}\times\mathbb{R}_{v}^{d}} [ D^{\beta}_{v}D_{x}^{\alpha}f   ]^{2}
  \langle v\rangle^{2k} dxdv+\iint_{\mathbb{R}^{d}\times\mathbb{R}_{v}^{d}} D^{\beta}_{v}[D_{x}^{\alpha}\partial_{v_{j}}f v _{j} ] D_{x}^{\alpha}D^{\beta}_{v}f
  \langle v\rangle^{2k} dxdv.
 \ea \ee
On account of integration by parts, one has
  \be\ba\label{A.44}
 &\iint_{\mathbb{R}^{d}\times\mathbb{R}_{v}^{d}} D^{\beta}_{v}[D_{x}^{\alpha}\partial_{v_{j}}f v _{j} ] D_{x}^{\alpha}D^{\beta}_{v}f
  \langle v\rangle^{2k} dxdv \\= &-\f12\iint_{\mathbb{R}^{d}\times\mathbb{R}_{v}^{d}}    (D_{x}^{\alpha}D^{\beta}_{v}f)^{2}
  \partial_{v_{j}}(v _{j}\langle v\rangle^{2k}) dxdv+\iint_{\mathbb{R}^{d}\times\mathbb{R}_{v}^{d}} D_{v}v _{j} D^{\beta-1}_{v} D_{x}^{\alpha}\partial_{v_{j}}f   D_{x}^{\alpha}D^{\beta}_{v}f
  \langle v\rangle^{2k} dxdv,
  \ea \ee
As a consequence, combining \eqref{A.43} and \eqref{A.44}, we infer that
 \be\ba\label{a3.60}
  &\B|\iint_{\mathbb{R}^{d}\times\mathbb{R}_{v}^{d}} \partial_{v_{j}}D^{\beta}_{v}[D_{x}^{\alpha}f v _{j} ] D_{x}^{\alpha}D^{\beta}_{v}f
  \langle v\rangle^{2k} dxdv\B|\leq C\|  D^{\beta}_{v}D_{x}^{\alpha}f   \langle v\rangle^{ k}\|_{L^{2}_{x,v}}^{2}.
 \ea \ee
Inserting \eqref{a3.50},  \eqref{a3.57}, \eqref{a3.60} into \eqref{a3.49}, one has
 \be\ba \label{key4}
 &\f{d}{dt}\|D_{x}^{\alpha}D^{\beta}_{v}f\langle v\rangle^{ k}\|  _{L^{2}_{x,v}}^{2}\\
 \leq& C\| D_{v}^{\beta-1} D_{x}^{\alpha+1}f
   \langle v\rangle^{ k}\|_{L^{2}_{x,v}} \|D_{x}^{\alpha}D^{\beta}_{v}f
   \langle v\rangle^{ k}\|_{L^{2}_{x,v}}\\+&C\|D_{x}^{\alpha}u\|_{L^{2}_x}^{1-\f{n}{4}} \|D_{x}^{\alpha+1}u\|_{L^{2}_x}^{\f{n}{4}}\|D^{\beta}_{v}\partial_{v_{i}}f\langle v\rangle^{ k}\|_{L^{2}_{x,v}}^{1-\f{n}{4}}
\|D_{x}D^{\beta}_{v}\partial_{v_{i}}f\langle v\rangle^{ k}\|_{L^{2}_{x,v}}^{\f{n}{4}}\| D_{x}^{\alpha}D^{\beta}_{v}f
  \langle v\rangle^{ k}\|_{L^{2}_{x,v}}\\&+C\|\nabla u\|_{L^{\infty}_x}
\|D_{x}^{\alpha-1}D^{\beta}_{v}\partial_{v_{i}}f\langle v\rangle^{ k}\|_{L^{2}_{x,v}} \| D_{x}^{\alpha}D^{\beta}_{v}f
  \langle v\rangle^{ k}\|_{L^{2}_{x,v}}+C\|  D^{\beta}_{v}D_{x}^{\alpha}f   \langle v\rangle^{ k}\|_{L^{2}_{x,v}}^{2}.
 \ea \ee
From \eqref{A.39}, we see that
  $$\ba
   \partial_t D_{x}^{2}D^{s-2}_{v}f+D_{v}^{s-2} (v\cdot\nabla_{x}D_{x}^{2}f)+\nabla_{v}D_{x}^{2}[D^{s-2}_{v}f u   ]-\nabla_{v}D^{s-2}_{v}[D_{x}^{2}f v  ] =0.
   \ea $$
 By suitable modification to the deduction of \eqref{a3.52}, we conclude by the Sobolev embedding theorem that
 \be\ba\label{aa3.68}
&\B( \iint_{\mathbb{R}^{d}\times\mathbb{R}_{v}^{d}}
[D_{x}^{2}(D^{s-2}_{v}\partial_{v_{i}}f u_{i} )-u_{i}\partial_{v_{i}}D_{x}^{2} D^{s-2}_{v}f ]^{2}
  \langle v\rangle^{2k} dxdv\B)^{\f12}\\\leq&
  \B( \int_{\mathbb{R}_{v}^{d}}  (\|D_{x}^{2}u\|_{L^{3}_{x}}^{2}
  \|D^{s-2}_{v}\partial_{v_{i}}f\|^{2}_{L^{6}_{x}}
  +\|\nabla u\|^{2}_{L^{\infty}}\|D_{x} D^{s-2}_{v}
  \partial_{v_{i}}f\|^{2}_{L^{2}_{x}})
\langle v\rangle^{2k}dv\B)^{\f12} \\\leq&C
  \|D_{x}^{2}u\|_{L^{3}_{x}}  \|D_{x}D^{s-2}_{v}\partial_{v_{i}}f \langle v\rangle^{ k} \|_{L^{2}}
+\|\nabla u\|_{L^{\infty}_x} \|D_{x} D^{s-2}_{v}\partial_{v_{i}}f\langle v\rangle^{ k}\|_{L^{2}_{x,v}}.
 \ea \ee
Inserting this into \eqref{3.51}, one obtains
\be\ba\label{a3.69} &\B|\iint_{\mathbb{R}^{d}\times\mathbb{R}_{v}^{d}}
 \partial_{v_{i}}D_{x}^{2}[D^{s-2}_{v}f u_{i}   ] D_{x}^{2}D^{s-2}_{v}f
  \langle v\rangle^{2k} dxdv\B|\\\leq&
 C
  \|D_{x}^{2}u\|_{L^{3}_{x}}  \|D_{x}D^{s-2}_{v}\partial_{v_{i}}f \langle v\rangle^{ k} \|_{L^{2}_{x,v}}\| D_{x}^{2}D^{s-2}_{v}f
  \langle v\rangle^{ k}\|_{L^{2}_{x,v}}
\\&+\|\nabla u\|_{L^{\infty}_x} \|D_{x} D^{s-2}_{v}\partial_{v_{i}}f\langle v\rangle^{ k}\|_{L^{2}_{x,v}}  \| D_{x}^{2}D^{s-2}_{v}f
  \langle v\rangle^{ k}\|_{L^{2}_{x,v}}.
    \ea \ee
    which together with \eqref{a3.50} and
\eqref{a3.60} means that
  \be\ba\label{A.55}
 &\f{d}{dt}\|D_{x}^{2}D^{s-2}_{v}f\langle v\rangle^{ k}\|  _{L^{2}_{x,v}}^{2}\\
 \leq& C\| D_{v}^{s-3} D_{x}^{3}f
   \langle v\rangle^{ k}\|_{L^{2}_{x,v}} \|D_{x}^{2}D^{s-2}_{v}f
   \langle v\rangle^{ k}\|_{L^{2}_{x,v}}\\&+C
  \|D_{x}^{2}u\|_{L^{3}_{x}}  \|D_{x}D^{s-2}_{v}\partial_{v_{i}}f \langle v\rangle^{ k} \|_{L^{2}_{x,v}}\| D_{x}^{2}D^{s-2}_{v}f
  \langle v\rangle^{ k}\|_{L^{2}_{x,v}}
\\&+\|\nabla u\|_{L^{\infty}_x} \|D_{x} D^{s-2}_{v}\partial_{v_{i}}f\langle v\rangle^{ k}\|_{L^{2}_{x,v}}  \| D_{x}^{2}D^{s-2}_{v}f
  \langle v\rangle^{ k}\|_{L^{2}_{x,v}}+C\|  D^{s-2}_{v}D_{x}^{2}f   \langle v\rangle^{ k}\|_{L^{2}_{x,v}}^{2}.
 \ea \ee
It is clear that
   $$\ba
 D^{3}=\{D_{x}^{3},D_{v}^{3},D_{x}D_{v}^{2},D^{2}_{x}D_{v} \}.
\ea $$
Hence,
putting \eqref{key1}, \eqref{key2}, \eqref{key3} and \eqref{A.55} together, we end up with
\be\ba\label{a.50}
\f{d}{dt}\|D^{3}   f\langle v\rangle^{k}\|_{L^{2}_{x,v}}^{2} \leq & C
  \|D^{s}_{x} u \|_{L^{2}_{x}}\B(\int_{\mathbb{R}^{n}}\|\partial_{ v_{i}}f\|^{2}_{L_{x}^{\infty}}\langle v\rangle^{2k}dv \B)^{\f12}\| D^{s}_{x}f\langle v\rangle^{ k}\|_{L^{2}_{x,v}}\\ &+C(\|\nabla u\|_{L^{\infty}_x}+\|u\|_{L^{\infty}_x}+
  \|D_{x}^{2}u\|_{L^{3}_{x}} +1) \|D^{3} f\langle v\rangle^{ k}\|_{L^{2}_{x,v}}^{2}.
 \ea \ee
 The
Gagliardo-Nirenberg inequality ensures that
$$\ba
&\|\partial_{ v_{i}}f\| _{L_{x}^{\infty}(\mathbb{R}^{2})}\leq C\|\partial_{ v_{i}}f\|^{\f12} _{L_{x}^{2}(\mathbb{R}^{2})}
\|\nabla^{2} \partial_{ v_{i}}f\| _{L_{x}^{2}(\mathbb{R}^{2})}^{\f12} \\&\|\partial_{ v_{i}}f\| _{L_{x}^{\infty}(\mathbb{R}^{3})} \leq C\|  \partial_{ v_{i}}f\|^{\f12} _{L_{x}^{6}(\mathbb{R}^{3})}
\|\nabla^{2} \partial_{ v_{i}}f\| _{L_{x}^{2}(\mathbb{R}^{3})}^{\f12}\leq C\|\nabla \partial_{ v_{i}}f\|^{\f12} _{L_{x}^{2}(\mathbb{R}^{3})}
\|\nabla^{2} \partial_{ v_{i}}f\| _{L_{x}^{2}(\mathbb{R}^{3})}^{\f12},\ea$$
which together with H\"older inequality means
\be\ba\label{a3.78}&\B(\int_{\mathbb{R}_{v}^{2}}  \|\partial_{ v_{i}}f\|^{2}_{L^{\infty}(\mathbb{R}^{2})}\langle v\rangle^{2k}dv \B)^{\f12}\leq  C\| \partial_{ v_{i}} f\langle v\rangle^{k}\|_{L_{x}^{2}(\mathbb{R}^{2}\times\mathbb{R}^{2}_{v})}^{\f12}\|\nabla^{2} \partial_{ v_{i}} f\langle v\rangle^{k}\|_{L_{x}^{2}(\mathbb{R}^{2}\times\mathbb{R}^{2}_{v})}^{\f12}\\&
\B(\int_{\mathbb{R}_{v}^{3}} \|\partial_{ v_{i}}f\|^{2}_{L^{\infty}(\mathbb{R}^{3})}\langle v\rangle^{2k}dv \B)^{\f12}\leq C  \|\nabla \partial_{ v_{i}} f\langle v\rangle^{k}\|_{L_{x,v}^{2}(\mathbb{R}^{3}\times\mathbb{R}^{3}_{v})}^{\f12}\|\nabla^{2} \partial_{ v_{i}}f\langle v\rangle^{k}\|_{L_{x,v}^{2}(\mathbb{R}^{3}\times\mathbb{R}^{3}_{v})}^{\f12}.
\ea\ee
It follows from \eqref{a.50} and \eqref{a3.78} that \eqref{pri4}.

In addition, it is trivial to  see  that $$\ba
 D^{s}=\{D_{x}^{s},D_{v}^{s},D_{v}^{s-1}D_{x},D_{x}^{2}D^{s-2}_{v},D_{v}^{\beta}D^{\alpha}_{x},
 3\leq\alpha\leq s-1,1\leq\beta\leq s-3, \alpha+\beta=s, s\geq4 \}.
\ea$$
As a consequence, we derive  from \eqref{key1}, \eqref{key2}, \eqref{key3}, \eqref{key4} and \eqref{A.55}
that
 \be\ba\label{a3.62}
 &\f{d}{dt}\|D^{s} f\langle v\rangle^{k}\|_{L^{2}_{x,v}}^{2}
\\
 \leq &C
  \|D^{s}_{x} u \|_{L^{2}_{x}}\B(\int_{\mathbb{R}^{d}_{v}} \|\partial_{ v_{i}}f\|^{2}_{L^{\infty}}\langle v\rangle^{2k}dv \B)^{\f12}\| D^{s}_{x}f\langle v\rangle^{ k}\|_{L^{2}_{x,v}}\\&+C(\|\nabla u\|_{L^{\infty}_x}  + \|u\|_{L^{\infty}_x} + 1+\|D_{x}^{2}u\|_{L^{3}_{x}})\| D^{s}_{x}f\langle v\rangle^{ k}\|_{L^{2}_{x,v}}^{2}\\&   +C\|D_{x}^{\alpha}u\|_{L^{2}_x}^{1-\f{n}{4}} \|D_{x}^{\alpha+1}u\|_{L^{2}_x}^{\f{n}{4}}\|D^{\beta}_{v}\partial_{v_{i}}f\langle v\rangle^{ k}\|_{L^{2}_{x,v}}^{1-\f{n}{4}}
\|D_{x}D^{\beta}_{v}\partial_{v_{i}}f\langle v\rangle^{ k}\|_{L^{2}_{x,v}}^{\f{n}{4}}\| D_{x}^{\alpha}D^{\beta}_{v}f
  \langle v\rangle^{ k}\|_{L^{2}_{x,v}}.
\ea \ee
Substituting  \eqref{a3.78} into \eqref{a3.62}, we get the desired estimate \eqref{pri5}. This completes the proof of  higher moment estimates.
\end{proof}
\section*{Acknowledgement}
The authors appreciate sincerely thank Professor Minsuk Yang for his discussion on this problem.
Wang was partially sponsored by Natural Science
Foundation of Henan (No. 232300421077) and supported by the National Natural
Science Foundation of China under grant (No. 11971446). Wei was partially supported by the National Natural Science Foundation of China under grant (No. 12271433).
  Ye is supported by the
Training Plan for Young Backbone Teachers in Colleges and Universities of Henan Province
(No. 2024GGJS022). Yu is partially supported by the National Natural
Science Foundation of China under grant (No. 12571245).

%\end{CJK*}
\end{document}